\documentclass[11pt,a4paper,english]{amsart}
\usepackage[T1]{fontenc}
\usepackage[latin9]{inputenc}
\usepackage{babel}
\usepackage{amstext}
\usepackage{amsthm}
\usepackage{comment}
\usepackage{amssymb}
\usepackage[unicode=true,pdfusetitle,
 bookmarks=true,bookmarksnumbered=false,bookmarksopen=false,
 breaklinks=false,pdfborder={0 0 1},backref=false,colorlinks=false]
 {hyperref}
 \usepackage[colorinlistoftodos]{todonotes}
\usepackage{amsmath}
\makeatletter
\pdfpageheight\paperheight
\pdfpagewidth\paperwidth

\numberwithin{equation}{section}
\numberwithin{figure}{section}
\theoremstyle{plain}
\newtheorem{thm}{\protect\theoremname}[section]
\theoremstyle{definition}
\newtheorem{defn}[thm]{\protect\definitionname}
\theoremstyle{plain}
\newtheorem{lem}[thm]{\protect\lemmaname}
\theoremstyle{definition}
\newtheorem{example}[thm]{\protect\examplename}
\theoremstyle{plain}
\newtheorem{cor}[thm]{\protect\corollaryname}
\theoremstyle{plain}
\newtheorem{prop}[thm]{\protect\propositionname}
\theoremstyle{remark}
\newtheorem{claim}[thm]{\protect\claimname}
\theoremstyle{remark}
\newtheorem{remark}[thm]{\protect\remarkname}

\usepackage{tikz-cd}
\usepackage{tikz}
\usepackage{adjustbox}
\usepackage{mathabx}
\usepackage[font=small,labelfont=bf]{caption}
\usepackage{graphicx}
\usepackage{amsfonts, amsmath, amsthm, amssymb}
\usepackage{wrapfig}
\usepackage[font=small,labelfont=bf]{caption} 
\usepackage{caption,subcaption,tikz}
\usepackage{amscd}

\xdefinecolor{darkgreen}{RGB}{0, 100, 0}
\definecolor{redviolet}{rgb}{0.78,0.08,0.52}
\definecolor{bu-red}{rgb}{0.8,0,0}

\newcommand{\R}{\mathbb{R}}
\newcommand{\Q}{\mathbb{Q}}

\newcommand{\bC}{\mathbb{C}}
\newcommand{\Z}{\mathbb{Z}}
\newcommand{\Spec}{\operatorname{Spec}}
\newcommand{\Nef}{\operatorname{Nef}}
\newcommand{\Proj}{\operatorname{Proj}}

\newcommand{\Pic}{\operatorname{Pic}}
\newcommand{\NE}{\operatorname{NE}}
\newcommand{\Hom}{\operatorname{Hom}}
\newcommand{\ord}{\operatorname{ord}}
\newcommand{\trop}{\operatorname{trop}}
\newcommand{\SK}{\operatorname{SK}}
\newcommand{\Aut}{\operatorname{Aut}}
\newcommand{\cX}{\mathcal{X}}

\newcommand{\cY}{\mathcal{Y}}
\newcommand{\cZ}{\mathcal{Z}}
\newcommand{\cD}{\mathcal{D}}
\newcommand{\cO}{\mathcal{O}}

\newcommand{\xyR}[1]{
  \xydef@\xymatrixrowsep@{#1}}
\newcommand{\xyC}[1]{
  \xydef@\xymatrixcolsep@{#1}}

 \allowdisplaybreaks

\usetikzlibrary{shapes.geometric, arrows}
\usetikzlibrary{arrows}
\usetikzlibrary{intersections}
\usetikzlibrary{calc,patterns,angles,quotes}
\usepgflibrary{patterns}
\usetikzlibrary{patterns}
\usetikzlibrary{decorations.markings}
\usetikzlibrary{positioning}

\theoremstyle{definition}
\newtheorem{cons}[thm]{Construction}

\makeatother

\providecommand{\claimname}{Claim}
\providecommand{\corollaryname}{Corollary}
\providecommand{\definitionname}{Definition}
\providecommand{\examplename}{Example}
\providecommand{\lemmaname}{Lemma}
\providecommand{\propositionname}{Proposition}
\providecommand{\theoremname}{Theorem}
\providecommand{\remarkname}{Remark}

\begin{document}
\title{Mirror Symmetry for Log Calabi-Yau Surfaces II}
\author{Jonathan Lai}
\address{Department of Mathematics \\ Imperial College London \\ 180 Queen's Gate \\ London SW7 2AZ\\ United Kingdom}
\email{j.lai@imperial.ac.uk}
\author{Yan Zhou}
\address{Department of Mathematics \\ Northeastern University \\ 360 Huntington Avenue \\ MA 02115  \\ United States of America}
\email{y.zhou@northeastern.edu}
\maketitle

\begin{abstract}
We show that the ring of regular functions of every smooth affine log Calabi-Yau surface with maximal boundary  has a vector space basis parametrized by its set of integer tropical points and a $\mathbb{C}$-algebra structure with structure coefficients given by the geometric construction of \cite{KY}. To prove this result, we first give a canonical compactification of the mirror family associated with a pair $(Y,D)$ constructed in \cite{GHK}  where $Y$ is a smooth projective rational surface, $D$ is  an anti-canonical cycle of rational curves and $Y\setminus D$ is the minimal resolution of an affine surface with, at worst, du Val singularities. Then, we compute periods for the compactified family using techniques from \cite{RS19} and use this to give a modular interpretation of the compactified mirror family.
    
\end{abstract}

\tableofcontents

\section{Introduction} 
\noindent For this paper, we work over the ground field $\bC$.\\
\subsection{Main results}
Let $(Y,D)$ be a pair where $Y$ is a smooth projective rational surface and $D\in |-K_Y|$ is either an irreducible rational nodal curve or a cycle of $n\geq 2$ smooth rational curves. We call such a pair a \emph{Looijenga pair}.  Write $D=D_1+..+D_n$ where $D_i$ is an irreducible component. We say that $(Y,D)$ is \emph{positive} if the intersection matrix $(D_i\cdot D_j)$ is not negative semi-definite.

Recall that a smooth surface $U$ is log Calabi-Yau with maximal boundary if and only if $U$ can be realized as the complement of $D$ in $Y$ for a Looijenga pair $(Y,D)$. Additionally, $U$ is affine if and only if $D$ supports an ample divisor. Another equivalent condition for a Looijenga pair $(Y,D)$ to be positive is that the complement $U=Y\setminus D$ is the minimal resolution of an affine log Calabi-Yau surface with at worst du Val singularities. 

Let $U$ be a smooth affine log-Calabi-Yau surface with maximal boundary. Let $\omega$ be the canonical volume holomorphic form (unique up to scaling) of $U$ and $\mathbb{C}(U)$ the function field of $U$. The set of integer tropical points $U^{\trop}(\Z)$ of $U$ is the set of divisorial discrete valuations $\nu:\mathbb{C}(U)\setminus \{0\} \rightarrow \Z$ such that $\nu(\omega)<0$ together with the zero valuation.

The main result of the paper is the proof of a refined version of Conjecture 0.6 announced in \cite{GHK} in the case of dimension $2$:
\begin{thm} \label{thm:mai-main-intro}
Let  $U$ be a smooth affine log Calabi-Yau surface with maximal boundary. Then, there is a  decomposition 
\begin{align}
    H^0(U,\cO_U)=\bigoplus_{q\in U^{\operatorname{trop}}(\Z)} \mathbb{C}\cdot \theta_q \label{eq:basis-intro}
\end{align}
of the vector space of regular functions on $U$  into one-dimensional subspaces, canonical up to the choice of one of the two possible orientations of $U^{\trop}(\mathbb{Z})$. 
\end{thm}

The canonical decomposition \ref{eq:basis-intro} implies the existence of  distinguished vector space basis elements, unique up to scaling, called theta functions in \cite{GHK}. The original conjecture of \cite{GHK} says that the multiplication rule of theta functions has a combinatorial description in terms of broken lines defined in \cite{GHK}. After \cite{GHK}, there has been rapid development of the geometric construction of structure constants for the multiplication of theta functions. Using these recent developments, we are able to prove a more refined version of the original conjecture in \cite{GHK}, as stated later in Theorem \ref{thm:main-main-intro-2}, which states that the structure constants for the multiplication rule of the canonical basis elements of $H^0(U,\cO_U)$ has a nice geometric construction via naive counts of analytic discs in the sense of non-archimedean geometry as in \cite{KY}. There, in \cite{KY}, broken lines are geometrically realized as the tropicalization, or the essential skeletons, of punctured non-archimedean analytic discs.  Alternatively, these structure constants can also be constructed via the punctured Gromov-Witten invariants in \cite{ACGS20} as we will explain later in the introduction.

Theorem \ref{thm:mai-main-intro}, and later Theorem \ref{thm:main-main-intro-2}, follow from another main result of the paper, the identification of the mirror family associated with a positive Looijenga pair $(Y,D)$ constructed in \cite{GHK} as the universal family over a certain moduli space, as stated in Theorem \ref{main theorem}. The proof of this result forms the bulk of this paper. 

Before we give the precise statement of Theorem \ref{main theorem}, let us briefly recall the construction from \cite{GHK}.  Let $(Y,D)$ be a positive Looijenga pair and $\NE(Y)$ the monoid generated by effective curve classes on $Y$. We can assume that $U=Y\setminus D$ is affine.  The mirror algebra $A$ constructed in \cite{GHK}, as a $\mathbb{C}[\NE(Y)]$-module, has a basis of theta functions parametrized by $U^{\trop}(\Z)$. The structure coefficients of $A$ in \cite{GHK}
are obtained combinatorially via the machinery of the canonical scattering diagram built from the data of relative Gromov-Witten invariants of $(Y,D)$, which counts rational curves on $Y$ meeting $D$ at a single point. Taking the spectrum of $A$, we get an algebraic family 
\[ \cX:= \Spec(A) \rightarrow S:=\Spec(\bC[\NE(Y)]).\] 
In \cite{GHK}, it is proved that the mirror family is a flat affine family of Gorenstein semi-log canonical surfaces with smooth generic fiber. 
However, it is hard to give a further concrete description of what fibers in the mirror family are by simply studying the mirror algebra $A$ itself. 

We aim to describe the mirror family associated with a positive Looijenga pair $(Y,D)$ as the universal family of deformations of $U$ plus some extra data.  To achieve this goal, we first give a canonical, fiberwise compactification of the mirror family to obtain a family of pairs, $(\bar{\cX},\cD)\rightarrow S$. We will be interested in studying the restriction of this family to the algebraic torus $T_Y=\Pic(Y)\otimes \mathbb{G}_m\subset S$. Over the torus, each fiber, $(\bar{\mathcal{X}}_t,\cD_t)$,  is a pair where $\bar{\mathcal{X}}_t$ is a projective rational surface with at worst du Val singularity and smooth along $\cD_t$, and $\mathcal{D}_t\in |-K_{\bar{\mathcal{X}}_t}|$  is a rational singular nodal curve. Naturally, we would like to know what surface pairs appear as fibers in this family. To this end, we are able to realize our family as the universal family of, what we call, \textit{generalized marked pairs}. The moduli functor in question is closely related to those considered in \cite{moduli}, which gives moduli stacks for Looijenga pairs (with and without markings). The main difference in our work is that we have to account for certain types of singular surfaces. We briefly describe the objects that our mirror family will parametrize. 

\begin{defn}\label{generalized pair}
A \textit{generalized pair} is a pair $(Y,D)$ where $Y$ is a projective rational surface, $D\in |-K_Y|$ is either an irreducible rational nodal curve or a cycle of $n\geq 2$ smooth rational curves, $Y$ is smooth along $D$ and $Y\setminus D$ is an affine surface with at worst du Val singularities.
\end{defn}

In particular, if $(X,\tilde{D})$ is the minimal resolution of a generalized pair $(Y,D)$ together with the strict transform $\tilde{D}$ of $D$, $\tilde{D}$ is anti-canonical. Then,  $(X,\tilde{D})$ is a positive Looijenga pair.

 We will then decorate  a generalized pair with two types of markings.

\begin{defn}
    Given a generalized pair $(Y,D)$, a \emph{marking} of $D$ is a choice of point $p_i\in D_i^{\circ}$ for each $i$, where $D_i^{\circ}$ is the intersection of $D_i$ with the smooth locus of $D$.
\end{defn}

\begin{defn}
An \emph{internal} (-2)-\emph{curve} on a Looijenga pair $(Y,D)$ means a smooth rational curve of self-intersection -2 disjoint from $D$.  A pair $(Y,D)$ is called \textit{generic} if $Y$ has no internal $-2$ curves.
\end{defn}

\begin{defn}\label{def:R(Y)-I(Y)}
Suppose the minimal resolution of $Y$ is given by $f:X\rightarrow Y$. Then, define
\[ R(Y):=\{\delta=\sum a_i\delta_i \in \Pic(X):\delta_i \text{ is a (-2)-curve}\, 
\text{contracted by} f\}\]
and let $I(Y)=R(Y)^{\perp}$.
\end{defn}

\begin{defn} \label{def:gen-mar}
Fix a generic Looijenga pair $(Y_0, D)$. A \textit{marking of} $\Pic(Y)$ \textit{by} $\Pic(Y_0)$ is an embedding $\mu: I(Y)\rightarrow \Pic(Y_0)$ which extends to an isometry $\tilde{\mu}:\Pic(X)\rightarrow \Pic(Y_0)$ where $X\rightarrow Y$ is the minimal resolution of $Y$ and $\tilde{\mu}$ is a marking as in Definition \ref{smooth marking}.
\end{defn}

The objects that our mirror family will parametrize are then given by the following objects.

\begin{defn}\label{generalized marked pair}
A \textit{generalized marked pair} is a tuple $((Y,D),p_i,\mu)$ such that $(Y,D)$ is a generalized pair, $p_i$ is a marking of the boundary and $\mu$ is a marking of $\Pic(Y)$.  We say $p_i$ together with $\mu$ gives a \emph{marking of} $(Y,D)$.
\end{defn}

Now, we give the second main theorem of this paper which leads to Theorem \ref{thm:mai-main-intro}:
\begin{thm}\label{main theorem}
Given a positive Looijenga pair $(Y,D)$ and its mirror family $\cX \rightarrow S$, there exists a canonical compactification of $\mathcal{X}$ to a family of surface pairs $(\bar{\cX},\cD)\rightarrow S$ such that the restriction to the algebraic torus $T_Y=\Pic(Y)\otimes \mathbb{G}_m$ is a family $(\bar{\cX}\mid_{T_Y},\cD \mid_{T_Y})$ of generalized pairs deformation equivalent to $(Y,D)$. Furthermore, $(\bar{\cX}\mid_{T_Y},\cD \mid_{T_Y})$ comes with a natural marking $p_i$ of the boundary $\cD$ and a marking $\mu$ of the Picard group by $\Pic(Y)$ such that the resulting marked family, $((\bar{\cX} \mid_{T_Y},\cD \mid_{T_Y}),p_i,\mu)$, is isomorphic to the universal family of generalized marked pairs.
\end{thm}

Now, given a smooth affine log Calabi-Yau surface $U$, we fix a Looijenga pair $(Y,D)$ that compactifies $U$ and a marking of $D$. By Proposition 4.10 (cf. Proposition \ref{prop:compar-ky}) of \cite{HKY}, the structure constants of the mirror algebra $A$ associated with $(Y,D)$ can be equivalently given by the geometric construction of \cite{KY}. By Theorem \ref{main theorem}, $(Y,D)$ will appear in the canonically compactified mirror family. So, to get a basis of theta functions for $H^0(U,\cO_U)$, it suffices to specialize the coefficients of $A$.  However, the identification of $(Y,D)$ with a fiber of the compactified mirror family is not unique. As a result, the structure constants of $H^0(U,\cO_U)$ we get by specializing coefficients depend on the identification we choose. However, once we fix a marking of $D$, we will fix the set of basis elements of theta functions. Thus, we are able to prove Theorem \ref{thm:mai-main-intro} and the following theorem, whose detailed proofs are carried out in Section \ref{sec:keel-yu}:
\begin{thm} \label{thm:main-main-intro-2}
Let  $U$ be a smooth affine log Calabi-Yau surface with maximal boundary. Fix a Looijenga pair $(Y,D)$ that compactifies $U$ and a marking of $D$. Then the decomposition given in \ref{eq:basis-intro} has a canonical basis  of theta functions such that the structure constants for multiplication of theta functions are given by the geometric construction in \cite{KY}. Different choices of compactifications  and markings will only result in the rescaling of basis elements and not affect the decomposition in \ref{eq:basis-intro}. 
\end{thm}

The structure constants of the mirror algebra $A$ in the construction of \cite{KY} are given by naive counts of non-archimedean analytic discs in the Berkovich analytification \cite{Ber90} of $U$. In particular, the structure coefficients of $A$ are actually in $\mathbb{N}[\NE(Y)]$. We refer readers to Section \ref{sec:keel-yu} for the background of \cite{KY} and Section 1.1 of \cite{KY} for the technical definition of the structure constants of mirror algebras. 

We want to emphasize that the structure constants of the multiplication rule of theta functions in Theorem \ref{thm:main-main-intro-2} can be constructed through other geometric approaches. As a special instance of the more general result in \cite{Man19}, the structure constants of theta functions for smooth affine log Calabi-Yau surfaces can be given the geometric interpretation as descendant log Gromov-Witten invariants. Another more systematic approach is that of punctured Gromov-Witten invariants. Originally announced in \cite{GS19} and realized in \cite{GS19} and \cite{GS21}, the authors construct the mirror of a log Calabi-Yau pair $(X,D)$ as the spectrum of a degree zero relative quantum cohomology ring for $(X,D)$. The definition of the ring and the corresponding structure constants use punctured Gromov-Witten invariants, as developed in \cite{ACGS20}. Furthermore, Example 3.14 of \cite{GS21} shows that the canonical wall structure in \cite{GS21} associated with a Looijenga pair agrees with that of \cite{GHK} and Theorem C of \cite{GS21} shows that mirror families constructed in \cite{GS21} agree with those in \cite{GS19}, where structure constants of mirror algebras are constructed by the punctured Gromov-Witten invariants of \cite{ACGS20}. 

\subsection{Connections to other works.}
\subsubsection{Tropicalization of theta functions in dimension 2}
In \cite{Man16}, the tropicalization of theta functions of an affine log Calabi-Yau surface is studied and used to generalize the dual pairing between the character lattice and cocharacter lattice of an algebraic torus. Moreover, \cite{Man16} generalizes the notion of Newton polytopes, Minkowski sums, and finite Fourier series expansions in toric geometry for regular functions on an affine log Calabi-Yau surface. It would be interesting to generalize techniques of \cite{Man16} and Section \ref{sec:compa} of the current work to the study of higher dimensional log Calabi-Yau varieties.

\subsubsection{Theta functions in algebraic geometry}


Generalized from classical theta functions of Abelian varieties, theta functions give canonical coordinates on moduli spaces of (log) Calabi-Yau varieties and therefore have a deep reach into moduli and deformation problems in classical algebraic geometry. Theorem \ref{thm:main-main-thm} and Theorem \ref{thm:main-main-intro-2} in the present work give an instance of the existence of theta functions on classical objects in algebraic geometry. For some explicit equations for mirror families of del Pezzo surface pairs where the boundary $D$ is maximally degenerate, see \cite{Bar18, GHKS}. Other examples of applications of mirror symmetry and theta functions to algebraic geometry include the proof of Looijenga's conjecture on smoothability of 2-dimensional cusp singularities \cite{GHK}, theta functions and canonical compactifications of moduli spaces of polarized K3 surfaces \cite{GHKS} and the study of unirationality of Koll\'ar-Shepherd-Barron-Alexeev moduli spaces of stable pairs \cite{HKY}.

\subsubsection{The symplectic cohomology of smooth affine log Calabi-Yau surfaces.}

In this paper, we do not justify the use of the adjective `mirror' before `family', i.e., we do not prove  how the family is the mirror family to the original Looijenga pair in the sense of Homological Mirror Symmetry. In the initial arXiv version of \cite{GHK}, a series of HMS conjectures are formulated for the mirror family. Since then, there has been significant progress toward proving these conjectures.

One main HMS conjecture in \cite{GHK} states that the wrapped Fukaya category of a smooth affine log Calabi-Yau surface $U$, viewed as a symplectic manifold with a choice of symplectic form, should be equivalent to the bounded derived category of coherent of its corresponding mirror in the mirror family. Under this conjecture, the Hochschild cohomology of the two categories is isomorphic and in particular, the degree zero part of the symplectic cohomology $SH^0(U)$ of $U$ will be isomorphic to the ring of regular functions of its mirror. Theorem \ref{thm:main-main-thm} and Theorem \ref{thm:main-main-intro-2} of the present work then imply that $SH^0(U)$ should be equipped with a canonical basis as a vector space. \cite{Pa15} computed $SH^0(U)$  and showed that it is equipped with a basis parametrized by $U^{\rm{trop}}(\mathbb{Z})$.

\subsection{Outline of the paper.}The paper starts off with a review of notions pertaining to Looijenga pairs, namely markings, marked periods, and the Torelli theorem, as in \cite{moduli}. We then modify those constructions to allow du Val singularities on $Y$ and provide an adapted Torelli theorem for the new class of objects, which we call generalized marked pairs. The moduli space and universal family are made explicit as well.

After a review of the mirror construction made in \cite{GHK}, we look at how to construct compactifications of the mirror families. This is carried out in Section \ref{sec:compa}. Similar to how a convex polygon in a lattice gives rise to a polarized toric variety, the mirror family is compactified using a more general notion of a convex polygon, which no longer lives in a lattice, but in an integral affine manifold, the tropicalization of $U=Y\setminus D$. Once the compactification is given, we then investigate the modular meaning of this family. It ends up that the compactified family can be realized as a family of generalized marked pairs, as defined in Subsection \ref{ssec:ger-mar}.

In order to make this identification, two problems must be addressed. The first is to determine the deformation type of the family, and the second is to compute the marked periods of the compactified mirror family. To this end, modern techniques from mirror symmetry and tropical geometry are used to construct curves that correspond to exceptional curves of a toric model of $(Y,D)$ and to compute their corresponding periods. In doing so, we are able to determine the deformation type of the family along the way. The fibers in the family are deformation equivalent to the original pair, $(Y,D)$. This is consistent with SYZ mirror symmetry where mirrors are given by dual torus fibrations. In this perspective, the compactified mirror should be  diffeomorphic to the original pair by Poincar\'e Duality in dimension $2$. The deformation type and period computations are carried out in Section \ref{sec:per-int}. The final result is that when restricted over the algebraic torus, the marked period map to $T_Y \simeq \Hom(\Pic(Y),\mathbb{G}_m)$ is the isomorphism (except for one special case, it is actually the identity map). The work done in these sections uses results from \cite{RS19}. This readily gives us the interpretation of our compactified family as the universal family. 

In Section \ref{sec:keel-yu}, we prove Theorem \ref{thm:mai-main-intro} and Theorem \ref{thm:main-main-intro-2}.

\subsection{Acknowledgements}
This project originates from a set of preliminary notes by M. Gross, P. Hacking, and S. Keel. We thank the three of them deeply for sharing their original notes and ideas with us. We want to give special thanks to P. Hacking for suggesting using the work of H. Ruddat and B. Siebert to produce cycles homologous to $(-1)$-curves and compute the corresponding periods. We would also like to thank B. Siebert for numerous conversations related to the period formula. The first author was supported by EPSRC grant EP/S025839/1. The second author was partially supported by the National Key Research and Development Program of China (No. 2021YFA1002000) and by the National Natural Science Foundation of China (No. 12171006).


\section{Periods for Looijenga Pairs}

\begin{defn}
A \textit{Looijenga pair} is a pair $(Y,D)$ where $Y$ is a smooth projective surface with $D\in |-K_Y|$ a reduced nodal curve. We write $D=D_1+..+D_n$ where $D_i$ is an irreducible component of $D$. Denote $U:=Y\setminus D$, which has a holomorphic symplectic $2$-form $\omega$, unique up to scaling, with simple poles along $D$, making $U$ a log Calabi-Yau surface.
\end{defn}

\begin{defn}
Let $(\bar{Y},\bar{D})$ be a Looijenga pair given by a smooth projective toric surface and its toric boundary. Let $\pi: Y\rightarrow \bar{Y}$ be the blowup of $\bar{Y}$ at a finite number of smooth points of $\bar{D}$ and let $D$ be the strict transform of $\bar{D}$. Then $(Y,D)$ is also a Looijenga pair and we call $(\bar{Y},\bar{D})$ a \textit{toric model} for $(Y,D)$.
\end{defn}





For a Looijenga pair $(Y,D)$, a blowup at a point in $D$ is a \textit{toric blowup} if the blowup point is a node. Otherwise, it is called a \textit{non-toric blowup}. Many constructions we consider are not affected by toric blowups. For convenience, we will use the following proposition:

\begin{prop}[Proposition $1.3$, \cite{GHK}]
Every Looijenga pair $(Y,D)$ has a toric blowup $(Y',D')$ such that $(Y',D')$ has a toric model.
\end{prop}
We now cover the notion of periods for Looijenga pairs. Fix a Looijenga pair $(Y,D)$ and let
\[ D^{\perp}:=\{\alpha\in \Pic(Y):\alpha\cdot [D_i]=0, \text{ for all } i\}.\]

We have an exact sequence 
\[
0 \rightarrow \mathbb{Z} \rightarrow H_2(U, \mathbb{Z}) \rightarrow D^{\perp}.
\]
The kernel of the above exact sequence is generated by the homology class $\Gamma$ of a  real 2-torus in $U$. In this paper, we always assume that we normalize the holomorphic volume form $\omega$ of $U$ so that 
\[
\int_{\Gamma} \omega = (2\pi i )^{2}.
\]
\begin{defn}\label{defn:unmarkedperiod}
Let $T_{(D^{\perp})^*}=\Hom(D^{\perp},\mathbb{G}_m)  $.  The \textit{period point} $\varphi_Y\in  T_{(D^{\perp})^*}$ of $(Y,D)$ is the homomorphism given by\\ 
\[ \varphi_Y:D^{\perp}\rightarrow \Pic^0(D)=\mathbb{G}_m, \hspace{.5cm}L\mapsto L|_D.\] 
\end{defn}
In justification of calling these homomorphisms periods, the mixed Hodge structure of $H^2(Y\setminus D)$ is classified by the period (see \cite{F84}). Briefly, given $\gamma\in D^{\perp}$, $\gamma$ can be represented by a class $\Sigma=\Sigma_k\pm [C_k]$ where $[C_k]$ are smooth curves in $Y$ meeting $D$ transversally at distinct points of the smooth locus of $D$. A cycle, $\Sigma'$, that is contained in $U$ and homologous to $\Sigma$, can be formed and used to define the period map $\varphi_Y':D^{\perp}\rightarrow \mathbb{G}_m$ by
\[ \gamma\mapsto \exp\left(\frac{1}{2\pi i}\int_{\Sigma'}\omega\right).\]
The mixed Hodge structures on $H^2(U)$ are indeed classified by such integrals. Then, Proposition $3.12$ of \cite{F84} shows that these integrals are equal to the periods as defined in Definition \ref{defn:unmarkedperiod}.

We will use an extended notion of the period point, which is that of a marked period point. 

\begin{defn}
For a Looijenga pair $(Y,D)$, a marking of the boundary is a choice of points $p_i\in D_i^{\circ}$, for each $i$, where $D_i^{\circ}$ is the intersection of $D_i$ with the smooth locus of $D$. 
\end{defn}

\begin{defn}
Let $(Y,D)$ be a Looijenga pair.
\begin{enumerate}\label{def:cones}
\item $\{x\in \Pic(Y)_{\R}:x^2>0\}$ is a cone with two connected components, one of which contains all ample classes. We denote this component by $C^+$.
\item For an ample class $H$, let $\tilde{M}$ be the set of classes $E$ such that $E^2=K_Y\cdot E=-1$ and $E\cdot H>0$. Then $C^{++}$ is defined by the inequalities $x\cdot E\ge 0$ for all $E\in \tilde{M}$.
\item Let $C_{D}^{++} \subset C^{++}$ be the subcone where additionally $x \cdot \left[ D_i \right] \geq 0$ for all $i$.
\end{enumerate}
\end{defn}

\begin{defn}
Let $(Y,D)$ be a Looijenga pair. 
\begin{enumerate}
\item The roots $\Phi\subset\Pic(Y)$ are classes in $D^{\perp}$ with square $-2$.
\item $\Delta_Y\subset  \Pic(Y)$ is the set of internal $(-2)$-curves. Here, an internal curve is a smooth rational curve disjoint from $D$ with self-intersection $-2$.
\item Let $\Phi_{Y}\subset \Phi \subset \Pic(Y)$ be the subset of roots, $\alpha$, with $\varphi_{Y}(\alpha) =1$. Note that $\Delta_Y \subset \Phi_Y \subset \Phi$.
\item $W\subset\Aut(\Pic(Y))$ is the subgroup generated by reflections $\{r_{\alpha}:\alpha\in \Phi\}$ where
\[ r_{\alpha}(\beta)=\beta+\langle \alpha,\beta\rangle \alpha.\]
Let $W_Y\subset W$ be the subgroup generated by $\{r_{\alpha}:\alpha\in \Delta_Y\}$.
\end{enumerate}
\end{defn}
By Proposition 3.4 of \cite{moduli}, we have $\Phi_Y =W_Y \cdot \Delta_Y$.

 \begin{defn}\label{smooth marking}
  Let $(Y,D)$ be a Looijenga pair and fix a generic Looijenga pair $(Y_0,D)$. Then a marking of $\Pic(Y)$ by $\Pic(Y_0)$ is an isomorphism of lattices $\mu:\Pic(Y)\rightarrow \Pic(Y_0)$ such that $\mu([D_i])=[D_i]$ for each $i$ and $\mu(C^{++})=C^{++}$. 
  \end{defn}
  The marked period can now be defined as follows.
 
 \begin{defn}\label{defn:markedperiod}
  Given a tuple $((Y,D),p_i,\mu)$, the \textit{marked period point} is the homomorphism $\varphi_{((Y,D),p_i,\mu)}\in \Hom(\Pic(Y_0),\Pic^0(D))$ given by
\[ \varphi(L):=(\mu^{-1}(L)|_D)^{-1}\otimes \cO_D\left(\sum (L\cdot D_i)p_i\right).\]
\end{defn} 

\begin{defn}
Two marked Looijenga pairs $((Y,D),p_i,\mu)$ and $((Y',D'),p'_i,\mu')$ are said to be isomorphic if there is an isomorphism $f:Y\rightarrow Y'$ such that $f(D_i)=D'_i$, $f(p_i)=f(p_i')$ and the induced map $f^*:\Pic(Y')\rightarrow\Pic(Y)$ satisfies $f^{*} = \mu^{-1} \circ \mu'$.
\end{defn} 

The following is a wide known result that we reproduce to make future computations explicit.
\begin{lem}[Lemma $1.6$, \cite{Fr2} and Lemma $2.1$, \cite{moduli}] \label{lem:pic-mul}
The choice of an orientation of $D$ gives a canonical isomorphism $\psi: \Pic^{\circ}(D) \simeq \mathbb{G}_m$. Via this identification, given two points $p,q \in D_{i}^{\circ}$, $\psi(\mathcal{O}(p-q))= \frac{z(q)}{z(p)}$ where $z$ is a coordinate we choose on $D_i$. 
\end{lem}

\begin{proof} 
Let $\nu:\tilde{D}\rightarrow D$ be the normalization map and $\tilde{D}_i$ the irreducible component of $D$ that is mapped to $D_i$. Suppose we have $L\in \Pic^0(D)$. Then consider $L_i=L|_{\tilde{D}_i}$ with trivializations $\sigma_i: L_i \simeq \mathcal{O}_{\tilde{D}_i}$. Denote $0$ and $\infty$ on $\tilde{D}_i$ by $0_i$ and $\infty_i$. Then to get back to $L$ and $D$, we have identifications $\mu_i: (L_i)_{0_i}\simeq(L_{i-1})_{\infty_{i-1}}$. Then using the trivializations $\sigma_i$'s, each $\mu_i$ can be identified with an element of $\mathbb{G}_m$. Then to $L\in \Pic^0(D)$, the corresponding element of $\mathbb{G}_m$ is $\mu_1...\mu_n$. 

Consider the case where $L=\mathcal{O}(p-q)$ where $p$ and $q$ are on $D_i^{\circ}$. Then, we get a trivialization of $L_i$ by using the nowhere vanishing global section $\frac{z-z(q)}{z-z(p)}$. Then, we see that $\mu_i=\frac{z(q)}{z(p)}$ and $\mu_j=1$ for $j\ne i$ and $\psi(\mathcal{O}(p-q)) = \frac{z(q)}{z(p)}$. 
\end{proof}

In \cite{moduli}, Theorem $1.8$ gives a global Torelli theorem for Looijenga pairs. We will use the following consequence of the theorem. 

\begin{thm}[Theorem $1.8$, Remark $1.9$, Corollary 2.10 of \cite{moduli}]\label{}
Two marked Looijenga pairs $((Y,D),\mu,{p_i})$ and $((Y',D'),\mu',{p'_i})$ are isomorphic if and only if $\varphi_{((Y,D),\mu,{p_i})}=\varphi_{((Y',D'),\mu',{p'_i})}$ and $\mu(\Nef(Y))=\mu'(\Nef(Y'))$. 
\end{thm}

The notion of marked Looijenga pairs extends to families and gives rise to a marked period map. Meaning, for a family of Looijenga pairs together with a marking of $\cD =\cD_{1}+\cdots+\cD_{n}$ and by $\Pic(Y)$,
\[ ((\cZ,\cD),p_i,\mu)\rightarrow S,\]
there is a marked period map
\[ \Phi:S\rightarrow\Hom(\Pic(Y),\Pic^0(D))\]
where $s\in S$ gets mapped to the homomorphism
\[ L\mapsto (\mu^{-1}_s(L)|_D)^{-1}\otimes \cO\left(\sum (L\cdot D_i)p_i(s)\right)\]
where for $X_s$, the fiber over $s$, $\mu_s$ is the marking of $\Pic(X_s)$ by $\Pic(Y)$ and $p_{i}(s)$ the marking of $\mathcal{D}_{i,s}$

\begin{cons}\label{def:cons-excep}
Fix $((Y,D),p_i)$, a Looijenga pair with a marked boundary and a toric model $p:(Y,D)\rightarrow (\bar{Y},\bar{D})$ such that the exceptional curves are given by $\{E_{ij}\}$, where $E_{ij}$ is an exceptional curve corresponding to a point blown up on the component $D_i$. Fix a single $E_{ij}$. Let $v_{ij}$ be the point defined by $E_{ij}\cap D_i$ and consider a path $\gamma_{ij}$ in $D_i$ that connects $v_{ij}$ and $p_i$ (the marked point on $D_i$), with the orientation from $p_i$ to $v_{ij}$. Let $\tau(\gamma_{ij})$ be a tubular neighborhood of $\gamma_{ij}$. We can glue this to the complement in $E_{ij}$ of a small disk around $v_{ij}$ to get a new cycle (with boundary) $\tilde{\beta}$. 
\end{cons} 

The corresponding integral $\int_{\tilde{\beta}} \omega$ can be computed. Using the standard Stokes' theorem and the fact that the restriction of $\Omega$ to $E_{ij}\setminus D$ is zero (since it is holomorphic on $U$), it follows that
\begin{equation}\label{eq:B}
    \frac{1}{2\pi i}\int_{\tilde{\beta}} \omega= \int_{p_i}^{v_{ij}} \frac{dz}{z} = \log z(v_{ij})-\log z(p_i)
\end{equation}
where $z$ is the coordinate we choose on $D_i^{\circ} \simeq \mathbb{C}^{*}$.

\begin{lem}\label{lem:per-excep}
We have
\begin{equation} \label{eq:A}
    \exp\left(\frac{1}{2\pi i}\int_{\tilde{\beta}}\omega\right)=\varphi_{((Y,D),p_i,\mu)}(\cO_{Y}(E_{ij}))
\end{equation}
where we assume that $\mu:\Pic(Y)\rightarrow\Pic(Y)$ is the identity. 
\end{lem}
\begin{proof}
By equation \ref{eq:B}, the left hand side of equality \ref{eq:A} is equal to $\frac{z(v_{ij})}{z(p_i)}$. The marked period for $E_{ij}$ is given by:
\begin{align*}
\varphi_{(Y,D,p_i,\mu)}(\cO_{Y}(E_{ij}))&=(\cO_{Y}(E_{ij}))|_D^{-1}\otimes \mathcal{O}_Y\left(\sum_k (E_{ij}\cdot D_k) p_k\right)\\
&=\mathcal{O}(-v_{ij})\otimes \mathcal{O}(p_i)\\
&=\mathcal{O}(p_i-v_{ij})\\
&=\frac{z(v_{ij})}{z(p_i)}
\end{align*}
where the last equality comes from the canonical identification $\Pic^{\circ}(D)=\mathbb{G}_m$ in Lemma \ref{lem:pic-mul}.
\end{proof}

We now construct universal families as in \cite{moduli}. Let $(Y,D)$ be a Looijenga pair, $\pi:(Y,D)\rightarrow (\bar{Y},\bar{D})$ a toric model with exceptional divisors ${E_{ij}}$. Denote $T_{Y}^{*}:=\Hom(\Pic(Y),\mathbb{G}_m)$. Then there are sections (as in Proposition $2.9$ of \cite{moduli})
\[ p_i:T_{Y}^{*}\rightarrow T_{Y}^{*} \times D_i^{\circ}\] 
\[ q_{ij}:T_{Y}^{*}\rightarrow T_{Y}^{*}\times D_i^{\circ}\]
such that for $\varphi\in T_{Y}^{*}$, we have
\[ \varphi(E_{ij})=\cO_D(q_{ij}(\varphi))^{-1}\otimes O_D(p_i(\varphi))\in \Pic^0(D).\]
Let $\Pi:(\cY_{\left\{ {E_{ij}}\right\}},\cD)\rightarrow T_{Y}^{*}\times \bar{Y}$ be the blowup along the sections $\left\{q_{ij}\right\}$. This family then comes with a marking of $\Pic(Y)$ on each fiber. Then
\[ \lambda:((\cY_{\left\{E_{ij} \right\}},\cD),p_i,\mu)\rightarrow T_{Y}^{*}\]
is what \cite{moduli} calls a universal family. Note that these universal families are not unique since there is a choice of order of blowup and changing the order will give rise to a non-isomorphic family over the torus, though they will be birational to any other universal family. In order to avoid confusion, we call these families in this paper by the alternative term, \textit{almost-universal families}.

\subsection{Generalized marked pairs}\label{ssec:ger-mar}
The fibers of the mirror family are not necessarily smooth and this section presents adjustments to the definition of a Looijenga pair in a way that allows for the singularities that appear. The construction here is similar to the situation for the moduli space of polarized K3 surfaces with simple surface singularities (see \cite{Mor}). 

Assume that the Looijenga pair $(Y,D)$ is positive. As in Lemma $6.9$ of \cite{GHK}, being positive is equivalent to the existence of $c_i$ such that the divisor $-(K_Y+\sum c_iD_i)$ is big and nef. It follows that for $m\in \mathbb{N}$ sufficiently large, the linear system $|mR|$ defines a birational morphism with the exceptional locus given by all curves disjoint from $D$ (which is a configuration of $-2$ curves). After contracting the internal curves, we get a pair $(Y',D)$ where $Y'$ has du Val singularities and the push forward of $-(K_Y+\sum c_iD_i)$ is supported on $D$ and is ample. Motivated by this, we now consider generalized pairs $(Y,D)$ as in Definition \ref{generalized pair}.


Recall now Definition \ref{generalized marked pair} and Definition \ref{def:R(Y)-I(Y)}. Suppose $(Y,D)$ is a generalized pair such that $Y$
has a minimal resolution given by $X\rightarrow Y$ and a marking of $\Pic$ given by $\mu$. Note that when $Y$ has no internal $(-2)$ curves, $\mu$ gives a marking of all of $\Pic(X)$ so the regular definition of a marked pair is recovered.

\begin{defn} 
An isomorphism of generalized marked pairs $((Y,D),p_i,\mu)$ and $((Y',D'),p'_i,\mu')$ is given by an isomorphism $f:Y\rightarrow Y'$ that lifts to an isomorphism  $\tilde{f}:X\rightarrow X'$ between corresponding minimal resolutions such that $f(D_i)=D'_i$, $f(p_i)=f(p'_i)$ and $\tilde{f}$ induces an isomorphism of lattices $\tilde{f}^*:I(Y')\rightarrow I(Y)$ such that $\tilde{f}^*=\mu^{-1}\circ\mu'$
\end{defn}

\begin{defn}\label{defn:markedperiodgen}
The marked period point of a generalized marked pair $((Y,D),p_i,\mu)$ with minimal resolution $X\rightarrow Y$, is given by
\[ \varphi_{((Y,D),p_i,\mu)}:=\varphi_{((X,D),p_i,\tilde{\mu})},\]
\end{defn}
\noindent where the right-hand side is as defined in Definition \ref{defn:markedperiod}.


\begin{remark}
 The definition of the marked period does not depend on the choice of extension, $\tilde{\mu}$. Fix a generalized mark pair $((Y,D),p_i,\mu)$ with minimal resolution $f:X\rightarrow Y$ and extensions $\tilde{\mu},\tilde{\mu'}:\Pic(X)\rightarrow \Pic(Y_0)$. By definition, $\tilde{\mu}|_{I(Y)}=\tilde{\mu'}|_{I(Y)}$ and they only differ on $R(Y)$. However, since $R(Y)$ is generated by irreducible internal (-2)-curves on $X$ contracted by $f$ and thus have periods are equal to $1$, we have that $\varphi_{((X,D),p_i,\tilde{\mu})}=\varphi_{((X,D),p_i,\tilde{\mu'})}$.
\end{remark}

\begin{prop}[Global Torelli for Generalized Marked Pairs]\label{Torelli}
Two generalized marked pairs, $((Y_1,D_1),p_1,\mu_1)$ and $((Y_2,D_2),p_2,\mu_2)$, are isomorphic if and only if $\varphi_{((Y_1,D_1),p_1,\mu_1)}=\varphi_{((Y_2,D_2),p_2,\mu_2)}$.
\end{prop}
\begin{proof} 
Let $f_i: X_i\rightarrow Y_i$ be the minimal resolution of $Y_i$. If the two marked generalized marked pairs are isomorphic, then it is clear that the marked periods must coincide. 

Suppose the two generalized marked pairs have the same period. Consider the induced isometry on $I(Y_i)^{\perp}$ which gives an isomorphism of root systems. Let $\Phi$ be the set of roots of $\Pic(Y_0)$. Let $\Phi' = \{ \delta \mid \delta \in \Phi,\, \varphi (\delta) =1\}$. By Proposition 3.4 of \cite{moduli}, $\Phi'$ is contained in $\tilde{\mu_i}(R({Y_i}))$. 
By Lemma $2.15$ and Theorem $3.2$ of \cite{moduli}, each $ \tilde{\mu_i}(\operatorname{Nef}(X_i))$ is  the closure of a connected component of 
\[
C_{D}^{++} \setminus \bigcup_{\alpha \in \Phi'}{\alpha}^{\perp}
\]
and the action of the Weyl group $W(\Phi')$ on the set of connected components of $C_{D}^{++} \setminus \bigcup_{\alpha \in \Phi'}{\alpha}^{\perp}$ is transitive. 
Thus, there must exist $w\in W (\Phi')$ such that $w \circ \tilde{\mu_1}(\operatorname{Nef}(X_1))=\tilde{\mu_2}(\operatorname{Nef}(X_2))$. Composing an isometry with $w$ is still an isometry, and $w$ does not affect the isometry on the level of $I(Y_i)=R(Y_i)^{\perp}$. Thus, replacing $\tilde{\mu_1}$ with $w\circ \tilde{\mu_1}$ keeps the marked period points for $((X_i,D_i),p_i,\tilde{\mu_i})(i=1,2)$ the same but their Nef cones are identified. By the Global Torelli theorem, they are isomorphic marked tuples. Thus, the same is true for $((Y_i,D_i),p_i,\mu_i) (i=1,2)$. 
\end{proof}

We can also make sense of families of generalized marked pairs. To make this precise, we first quote a result from \cite{Mor}. First recall Definition $\ref{def:R(Y)-I(Y)}$ for $I(Y)$. 
\begin{lem}[Corollary $3$, \cite{Mor}]\label{lem:I-sheaf}
If $p:\cX\rightarrow S$ is a proper family of surfaces with only rational double points as singularities, then there is a sheaf $I(\cX/S)$ such that for each $s\in S$ and $U$ a sufficiently small neighborhood of $s\in S$ and $X_s$ the fiber of $s$, it follows that $\Gamma(U,I(\cX/S))=I(X_s)$.
\end{lem}

Now fix a generic generalized pair $(Y_0,D)$ and let $\cX \rightarrow S$ be a proper family of rational surfaces with only rational double points as singularities. Then we consider the sheaf $I(\cX/S)$ as given in Lemma \ref{lem:I-sheaf}. Then a family of generalized marked pairs over $S$ is a morphism $f:(\cY,\cD)\rightarrow S$ together with sections $p_i:S\rightarrow \cD_i$ and an injection $\mu:I(\cY/S)\rightarrow {\underline{\Pic(Y_0)}}_S$ such that
\begin{enumerate}

\item $\cY/S$ is a flat family of surfaces,
\item $\cD_i$ is a Cartier divisor on $\cY/S$ for each $i$
\item For each $s\in S$, $\mu|_s$ extends to an isometry of $\Pic(Y_s)$ with $\Pic(Y_0)$, and
\item Each closed fiber $(\cY_s,\cD_s,p_i(s),\mu|_s)$ is a generalized marked pair with marking of the boundary given by $p_i(s)$ and marking of $\Pic$ by $\mu|_s$.
\item Furthermore, in the case $n=1,2$, we assume that we are given an orientation of $\mathcal{D}$, that is, an identification $\mathbb{Z} \times S \xrightarrow{\sim} R^{1}f_{*}\mathbb{Z}_{\cD}$.
\end{enumerate}

Notice that the sections $p_i$ and the orientation $\cD$ determine a canonical isomorphism $D\times S \xrightarrow{\sim} \cD$. This is used to define the marked period mapping for families over an arbitrary base $S$. Then given a family of generalized marked pairs, $((\cY,\cD),p_i,\mu)\rightarrow S$, we have an induced marked period map to $\Hom(\Pic(Y_0),\mathbb{G}_m)$.
\begin{lem}\label{noaut} 
There are no non-trivial automorphisms of generalized marked pairs $((Y,D),p_i,\mu)$.
\end{lem}

\begin{proof}
Let $((Y,D),p_i,\mu)$ be a generalized marked pair with minimal resolution $((X,D),p_i,\tilde{\mu})$. Suppose $f\in \Aut(Y,D)$ fixes $D$ pointwise and $f^*$ induces the identity on $I(Y)$ which extends to an isometry on $\Pic(X)$ and consider a lift to an automorphism of $(X,D)$ given by $\tilde{f}$. Consider the following exact sequence
\[ 1\rightarrow \Hom(N'\rightarrow \mathbb{G}_m)\rightarrow \Aut(X,D)\rightarrow H_{X}/W_{X}\rightarrow 1,\]
where $H_X\subset \Aut(\Pic(X))$ is the subgroup that fixes the period point $\varphi_X$. If $\tilde{f}$ maps to an element in $H_X$ that is not contained in $W_X$, then it will not fix $R(Y)$ and therefore $I(Y)$, in contradiction with that $I(Y)$ is fixed. 
 Therefore, the automorphism $\tilde{f}$ must lie in $\Hom(N'\rightarrow\mathbb{G}_m)$, which is identified with
\[K:= \operatorname{ker}[\Aut(X,D)\rightarrow\Aut(\Pic(X))].\]
By Proposition $2.6$ in \cite{moduli}, the natural map $K\rightarrow \Aut^0(D)$ given by restriction is an injection. Then in order for $D$ to be fixed pointwise, $f$ must be the identity.

\end{proof}

\begin{prop} 
$T_{Y_0}^{*}=\Hom(\Pic(Y_0),\mathbb{G}_m)$ is the fine moduli space of generalized marked pairs $(Y,D,p_i,\mu)$ with a marking of $D$ and a marking of the Picard group by $\Pic(Y_0)$.
\end{prop}
\begin{proof}
Consider an almost-universal family $\cY\rightarrow T_{Y_0}^{*}$ of Looijenga pairs $(Y,D)$ with a marking of the Picard group by $\Pic(Y_0)$ as well as a marking of the boundary $D$. Given $s\in T^{*}_{Y_0}$, because $(\cY_s,\cD_s)$ is positive, all internal curves can be simultaneously contracted. This gives rise to a family, $(\check{\cY},\cD,p_i)\rightarrow T_{Y_0}^{*}$ where the fibers $\check{\cY}_s$ has at worst, du Val singularities. Furthermore, the original marking by $\Pic(Y_0)$ contains the data for a marking on the generalized family and the boundary marking descends as well. Thus, we get a family of generalized marked pairs $(\check{\cY},\cD,p_i,\mu)\rightarrow T_{Y_0}^{*}$. Furthermore, by construction, the period of the fiber over $\varphi\in T_{Y_0}^{*}$ is $\varphi$. Note that the choice of the almost-universal family does not affect the construction. The fact that the moduli space is fine follows from Lemma \ref{noaut}.
\end{proof}


\section{Review of Mirror Symmetry for Log Calabi-Yau Surfaces} 
We review notions and results related to the construction of the mirror family in \cite{GHK}. Let $(Y,D)$ be a positive Looijenga pair and denote $U=Y\setminus D$. To each pair, we will associate to it an integral affine manifold with singularities.
\begin{defn}
An \textit{integral affine manifold}  is a real manifold such that the transition functions are integral affine. 
\end{defn}
The spaces we consider do not have a global integral affine structure, so we consider integral affine manifolds with singularities. These are real manifolds $B$ such that there is an open subset $B_0\subset B$ that is an integral affine manifold such that the singular locus of $B$ (where the affine structure is not defined) given by $B\setminus B_0$ is a locally finite union of closed submanifolds of codimension at least two. There is a local system $\Lambda_B$ on $B_0$ given by flat integral vector fields. Given any simply connected subset $\tau \subset B_0$, we write $\Lambda_{\tau}$ for the stalk of $\Lambda$ at any point of $\tau$ (any two such stalks are canonically identified via parallel transport). We also use $\lambda_{\tau}$ to denote the local system of flat integral vector fields on $\tau$ since what we mean will be clear from the context. 
We  denote  by $\check{\Lambda}_B$ the  local system dual to $\Lambda$. 

Now, to a fixed Looijenga pair $(Y,D)$, we will associate an integral affine manifold, $B_{(Y,D)}$. Fix a lattice $M\simeq \Z^n$ and denote $M_{\R}=M\otimes_{\Z}\R$. Let $D= D_1+..+D_n$. For each node $p_{i,i+1}=D_i\cap D_{i+1}$, we take a rank $2$ lattice $M_{i,i+1}$ with a basis $v_i,v_{i+1}$. Then we take the cone $\sigma_{i,i+1}\subset {M_{i,i+1}}_{\R}$ generated by the two basis vectors. We then form $B_{(Y,D)}$ by gluing the $\sigma_{i,i+1}$ to $\sigma_{i-1,i}$ along the ray $\rho_i$. Note that $B_{(Y,D)}$ is homeomorphic to $\mathbb{R}^2$. Now we define the affine structure on $B_{(Y,D)}\setminus \{0\}$. Let $U_i$ be the interior of $\sigma_{i,i+1}\cup\sigma_{i-1,i}$. Then we define a chart $\psi_i:U_i\rightarrow M_{\R}$ given by:
\[ \psi_i(v_{i-1})=(1,0), \psi_i(v_i)=(0,1),\psi_i(v_{i+1})=(-1,-D_i^2).\]

\begin{remark} \label{rmk:irr-blo}
Note that the above construction even makes sense in the case where $n=1$, i.e., when the anticanonical cycle $D$ is an irreducible nodal curve. In this case, we only have one cone $\sigma_{1,1}$ whose two sides are identified. The subtle point is how we define the integral affine charts. Instead of using $D^{2}$, we use $D^{2} -2$, the degree of the normal bundle of the map from the normalization of $D$ to $Y$. However, this special case can complicate the arguments in the proofs in later sections of the paper. To bypass this technical complication, in this case, we usually pass to the blowup of $D$ at the node and replace $D$ with the reduced inverse image of $D$ under the blowup.  By Lemma 1.6 of \cite{GHK}, toric blowups will not change the underlying integral affine structure. Thus, in this way, we will obtain a unifying description for the integral affine manifolds associated with Looijenga pairs. 
\end{remark}

When $(Y,D)$ is clear from context, we will often drop the subscript and just denote the associated manifold $B$ instead of $B_{(Y,D)}$. Note that if the pair $(Y,D)$ is toric, the affine structure extends over the origin (see Lemma $1.8$ from \cite{GHK}). 

Denote by $B(\Z)$ the set of integer points on $B$. Then, we also have the identification $B(\Z) \simeq U^\mathrm{trop}(\Z)$. However, we will not use this perspective until Section \ref{sec:keel-yu}. \\

Let $P=\NE(Y)$ and $R=\bC[P]$ and $(Y,D)$ be positive.


The following notions from section $2.1$ of \cite{GHK} will be used and we reproduce them here.

\begin{defn}
A $\Sigma$-piecewise linear multivalued function on $B$ is a collection $\varphi=\{\varphi_i\}$ where $\varphi_i$ is a $\Sigma$-piecewise linear function on $U_i$ with values in $P_{\R}^{\operatorname{gp}}$. This is equivalent to giving a collection of $\rho^{-1}\Sigma$ piecewise linear functions $\varphi_i:\Lambda_{\R,\rho_i}\rightarrow P_{\R}^{\operatorname{gp}}$. Furthermore, two such functions $\varphi$ and $\varphi'$ are considered equivalent if $\varphi_i-\varphi'_i$ is linear for each $i$.
\end{defn}

From a multivalued function, $\varphi=\{\varphi_i\}$, we can form a principal $P_{\R}^{\operatorname{gp}}$ bundle $\pi:\mathbb{P}_0\rightarrow B_0$ with a section $\varphi:B_0\rightarrow \mathbb{P}_0$. This is constructed by gluing $U_i\times P_{\R}^{\operatorname{gp}}$ to $U_{i+1}\times P_{\R}^{\operatorname{gp}}$ along $U_i\cap U_{i+1}\times P_{\R}^{\operatorname{gp}}$ by
\[ (x,p)\mapsto (x,p+\varphi_{i+1}(x)-\varphi_i(x).\]
The section $\varphi$ is formed by local sections $x\mapsto (x,\varphi_i(x))$. Then we define a local system $\mathcal{P}$ on $B_0$ by
\[ \mathcal{P}:=\pi^*\Lambda_{\mathbb{P}_0}\]
which fits into the short exact sequence of local systems given by
\[ 0\rightarrow \underline{P}^{\operatorname{gp}}\rightarrow \mathcal{P}\xrightarrow{r} \Lambda\rightarrow 0\]
where $r$ is the derivative of $\pi$. 

In the case where $P=\NE(Y)$, there is a unique (up to a linear function) $\Sigma$-piecewise linear multivalued function
\[ \varphi: |\Sigma|\rightarrow P_{\R}^{\text{gp}}\]
where the bending parameter (as defined in Definition $1.11$ of \cite{GHK}) at $\rho_i$, denoted by $\kappa_{\rho_i}$, is given by $[D_{\rho_i}]$.

The canonical scattering diagram in \cite{GHK} is given by the following. Fix a ray $\mathfrak{d}\subset B$ with a rational slope with the origin at its endpoint. If $\mathfrak{d}$ is not a ray of $\Sigma$, then refine $\Sigma$ by adding $\mathfrak{d}$ and also other rays such that each maximal cell of $\Sigma'$ is affine isomorphic to the first quadrant in $\mathbb{R}^2$. This corresponds to a toric blowup $\pi:\tilde{Y}\rightarrow Y$. Now set $C$ to be the irreducible component corresponding to $\mathfrak{d}$. Let $\tau_{\mathfrak{d}}$ be the smallest cone containing $\mathfrak{d}$ and $m_{\mathfrak{d}}$ the primitive generator of the tangent space to $\mathfrak{d}$, pointing away from the origin. Then define
\[ f_{\mathfrak{d}}:=\exp \left( \sum_{\beta} k_{\beta}N_{\beta}z^{\pi_{*}(\beta)-\varphi_{\tau_{\mathfrak{d}}}(k_{\beta}m_{\mathfrak{d}})}\right)\]
where the sum is made over $\beta\in A_1(\tilde{Y})$ such that
\[ \beta\cdot \tilde{D}_i= 
\begin{cases} 
      k_{\beta}\leq 0 \text{ for } \tilde{D}_i=C\\
      0 \text{ for } \tilde{D}_i\ne C
        \end{cases}
\]
and $k_{\beta}>0$.  Heuristically, $N_{\beta}$ is the number of maps from affine lines to $\tilde{Y}\setminus \tilde{D}$ whose closures represent the class $\beta$. A class $\beta\in A_i(\tilde{Y})$ is called an $\mathbb{A}^1$ class if $N_{\beta}\ne 0$. The canonical scattering diagram is given by
\[ \mathfrak{D}^{\text{can}}:=\{ (\mathfrak{d},f_{\mathfrak{d}}):\mathfrak{d}\subset B, \text{ a ray of rational slope}\}.\]

\begin{defn}
A broken line $\gamma$ in $(B,\Sigma)$ for $q\in B_0(\Z)$ with endpoint $Q\in B_0$ is a proper continuous piecewise integral affine map $\gamma:(-\infty,0]\rightarrow B_0$ with finite domains of linearity, $L_1,...,L_n$ such that no segments are contained in a ray of $\mathfrak{D}$. Each $L_i$ is decorated by a monomial $m_i=c_iz^{C_i}w^{v_i}\in \Q[\Lambda_{L_i}\oplus P]$ such that $\gamma$ can bend only when it crosses a ray of $\mathfrak{D}$, each $L_i$ is compact expect for $L_1$, which goes to infinity and is parallel to $\rho_q$, $m_1=z^q$, $L_n$ ends at $s$ and each $v_i$ is non-zero and parallel to $L_i$ and points in the opposite direction to $L_i$. The other directions are determined by this. The monomial $m_{i+1}$ is required to be one of the monomial terms in $m_i\cdot f^{\langle n,v_i\rangle}$ where $f$ is the function attached to $\rho$, the ray that $L_i\cap L_{i+1}$ is on and $n$ is the unique primitive element vanishing on the tangent space to $\rho$ that is positive on $v_i$.
\end{defn}

Let $v(\gamma)=v_n$ and $c(\gamma)=c_nz^{c_n}$. 
\begin{thm}{(Theorem 2.34 and Corollary 6.11, \cite{GHK})} \label{thm:fiber-mirror}
Let $(Y,D)$ be positive. For $p_1,p_2,r \in B(\Z)$, there can only be finite many pairs of broken lines $(\gamma_1,\gamma_2)$ such that $\operatorname{Limits}(\gamma_i)=(p_i,s)$ and $v(\gamma_1)+v(\gamma_2)=s$. Define:
\[ \alpha(p_1,p_2,s)=\sum c(\gamma_1)c(\gamma_2)\]
Then we can define a multiplication rule given by
\[ \theta_{p_1}\theta_{p_2}=\sum_{r\in B(\Z)} \alpha(p_1,p_2,r)\theta_r.\]
This gives a commutative and associative $R$ algebra structure on $A$. The induced map 
\[ f:\Spec(A)\rightarrow \Spec(R)\]
is a flat affine family of Gorenstein semi-log canonical surfaces with smooth generic fiber.
\end{thm}

For the rest of the paper, we assume that the pair $(Y,D)$ is positive. As mentioned before, this means that the intersection matrix for $D$ is not negative semi-definite. Lemma $6.9$ in \cite{GHK} states that the following are equivalent:
\begin{enumerate}
\item $(Y,D)$ is positive.
\item $Y\setminus D$ is the minimal resolution of an affine surface with, at worst, du Val singularities.
\item There exist positive integers $a_1,...,a_n$ such that for all $j$, $\left(\sum a_i D_i\right)\cdot D_j>0$.\end{enumerate}
There are several important implications of this assumption, also from the same Lemma. These include:
\begin{enumerate}
\item The cone $\NE(Y)_{\R}$ is rational polyhedral, generated by only finitely any classes of rational curves. Furthermore, every nef line bundle on $Y$ is semi-ample.
\item The subgroup $G\subset \Aut(\Pic(Y))$ that fixes $[D_i]$ is finite.
\item The union $R\subset Y$ of all curves disjoint from $D$ is contractible.
\end{enumerate}

In the language of \cite{GHK}, this means that we can set $P=\NE(Y)$ and the ideal $J$ is zero. The key takeaway is that the mirror construction gives an actual algebraic family over $\Spec (\bC[\NE(Y)])$ with smooth generic fibers.

\begin{defn}\label{def:relative-torus}
Let $(Y,D)$ be a pair with $D=D_1+...+D_n$ and $\mathbb{A}^{D} = \mathbb{A}^{n}$ be the affine space with one coordinate for each component $D_i$. The relative torus $T^D$ is the diagonal torus acting on $\mathbb{A}^{n}$ whose character group is the free module with basis $e_{D_1},...,e_{D_n}$. 
\end{defn}

\begin{defn}\label{def:weight-map}
There are two types of weights that will come into play associated with the action of $T^{D}$ on the canonical algebra generated by theta functions. One comes from the map $w:A_1(Y)\rightarrow \chi(T^D)$ given by
\[ C\mapsto \sum_i (C\cdot D_i) e_{D_i}\]
and the other comes from $w:B\rightarrow \chi(T^D)\otimes \R$, where $w$ is the unique piecewise linear map with $w(0)=0$ and $w(v_i)=e_{D_i}$ for $v_i$ the primitive generator of $\rho_{i}$. Together, we get a weight map 
\begin{align*}
w: B(\mathbb{Z}) \times \mathrm{NE}(Y) & \rightarrow \chi(T^{D}) \\
(q,C) & \mapsto w(q) + w(C)
\end{align*}
\end{defn}

The following theorem is a special case of Theorem 5.2 in \cite{GHK} when $(Y,D)$ is positive:
\begin{thm}[Theorem 5.2, \cite{GHK}] \label{thm:rel-equi}The relative torus $T^{D}$ acts equivariantly on the mirror family $\mathcal{X}=\mathrm {Spec}(A) \rightarrow \mathrm{Spec}(\mathbb{C}[\mathrm{NE}(Y)])$. Furthermore, each theta function $\theta_{q}$, $q \in B(\mathbb{Z})$, is an eigenfunction of the action of $T^{D}$ with character $w(q)$. 
\end{thm}


\section{Compactifications via Polygons} \label{sec:compa}
\subsection{Generalized Half Spaces and Polygons in $B$} 
In this section, the convex geometry of $B$ is developed. Namely, the notion of a generalized half space is defined and we show how convex polytopes come from intersections of generalized half spaces. Most importantly, we construct the polytopes that will be used to compactify the mirror family. First, we set up the notion of convex polygons in $B$.
\begin{defn}
\label{def:polygons} A closed, connected subset $F$ in $B$
with $0$ in the interior is a \emph{polygon} in $B$
if for each maximal cell $\sigma$ in $\Sigma$, $F\cap\sigma$ is a polygon. Given a polygon $F$ on $B$, we say $F$ is\emph{
convex} if any immersed line segment on $B_0$ with endpoints in $F$ is contained in $F$. We say a polygon is rational if for every vertex $v$, we have $v \in B(\mathbb{Q})$ and each edge has rational slope.  We say a convex rational polygon is \emph{nonsingular} if for any vertex $v$ as the intersection of two edges $E_1$ and $E_2$, the primitive generators of $E_1$ and $E_2$ form a basis of $\Lambda_{v}$. 
\end{defn}

\begin{remark}
We emphasize that by our definition, a polygon in $B$ will always have $0$ in its interior.
\end{remark}

Recall the notion of the developing map on $B_0$. 
\begin{defn}
Let $q:\tilde{B}\rightarrow B_0$ be the universal cover. Then the universal cover inherits an integral affine structure pulled back from $B_0$. If we patch together integral affine charts on $B_{0}$, then up to integral linear functions, we obtain a canonical map called the \emph{developing map}  $\delta:\tilde{B} \rightarrow \mathbb{R}^2$. The fan $\Sigma $ induces a decomposition of $\tilde{B}$ into cones, which we  denote by $\tilde{\Sigma}$. 
\end{defn} 

\begin{lem} \label{devel-map}
The developing map is one-to-one if $(Y,D)$ is negative semidefinite. In this case, the closure of the image is a convex cone and is strictly convex if $(Y,D)$ is negative definite. If the developing map is not one-to-one, then it surjects onto $\mathbb{R}^{2}\setminus \{0\}$.
\end{lem}
\begin{proof}
Since developing maps are unique up to ${\rm SL}_{2}(\mathbb{Z})$-transformations,
it suffices to prove the lemma for any particular choice of the developing
map. We can assume that $D$ has no $-1$ curves by contracting irreducible $(-1)$ curves in $D$. Then $(Y,D)$ is negative semi-definite
if $D_{i}^{2}\leq-2$ for all $i$ and negative definite if, in addition,
$D_{i}^{2}<-2$ for at least one $i$. If $D_{i}^{2}=-2$ for all
$i$, the developing map that takes a lift of $\sigma_{i,i+1}$ to
the cone generated by $(i+1,1)$ and $(i,1)$ identifies $\tilde{B}$
with the upper half plane of $\mathbb{R}^{2}$. If $D_{i}^{2}<-2$
for some $i$, it follows that the developing map identifies
$\tilde{B}$ with a strictly convex cone in $\mathbb{R}^{2}$. 

If $(Y,D)$ is positive and contains no
$(-1)$-curves, $D_{i}^{2}\geq0$ for some $i$. Then the image under
the developing map, restricted to a lift of $\sigma_{i-1,i}\cup\sigma_{i,i+1}$,
will contain a half space in $\mathbb{R}^{2}$. It follows that the
developing map cannot be injective and is surjective onto $\mathbb{R}^{2}\setminus\{0\}$. 
\end{proof}

\begin{cor} \label{ray-escape}
A Looijenga pair $(Y,D)$ is positive if and only if every immersed,
directed ray on $B$ not passing through $0$ goes to infinity. 
\end{cor}

\begin{proof}
We use the developing map $\delta$. Let $R$ be an immersed directed
ray on $B$ and choose a lift $\tilde{R}$ of $R$ to the universal
cover $\tilde{B}$. By Lemma \ref{devel-map}, if $(Y,D)$ is positive, then
$\delta$ is surjective and therefore there exists a cone $\sigma$
in $\Sigma$ such that a connected component of $\delta^{-1}(\sigma)$
traps in which $\tilde{R}$ goes to infinity. Conversely, if
$(Y,D)$ is negative semidefinite, again by Lemma \ref{devel-map}, $\delta$ is
injective and the closure of its image is a convex cone. Let $R^{'}$
be an immersed ray with the endpoint in the interior of a cone $\sigma^{'}$
but escaping the cone $\sigma^{'}$. Let $\tilde{R}^{'}$ be a
lifting of $R^{'}$ to $\tilde{B}$. Let $\tilde{\Sigma}$ be the
decomposition of $\tilde{B} \setminus {0}$ into cones induced
by $\Sigma$. Since the image of $\delta$ is a convex cone and $\delta$
is injective, $\tilde{R}^{'}$ intersects rays in $\tilde{\Sigma}$
infinitely many times. Therefore, $R^{'}$ wraps around $0$ infinitely
many times and does not escape to infinity. 
\end{proof}

Let $\bar{L}\subset B_0$ be a line segment with rational slope and integral unit tangent vector $v$. For $p\in B_0$, we can view the vector $\overrightarrow{0p}$ as an element in $T_{p,B_0}$ and if $p\in B(\Z)$, then we can view it furthermore as an element of $\Lambda_p$. It is important to note that $\overrightarrow{0p} \wedge v$ is the same as long as we choose $p\in L$. If $\overrightarrow{0p}\wedge v=0$, then $\bar{L}$ must be on the ray spanned by $\overrightarrow{0p}$. Then, if $\overrightarrow{0p}\wedge v\ne 0$ for all $p \in \overline{L}$, it follows that $\bar{L}$ can be extended to an immersed line $L$ which does not pass through $0\in B$. The  distance of $L$ to $0$ is denoted by
\[ \operatorname{dist}(L,0):= |\overrightarrow{0p}\wedge v|\]
for any $p\in L$. Note that this is well-defined even when we cross charts since the transition functions are in $ \rm{SL}_2(\Z)$.  Given a line $L$, there are two sides to it. On one side, there is $0$ and on the other side, there is $\infty$. We always orient $L$ in a way such that $0$ will be on the left of $L$. In particular, the distance of $L$ from $0\in B$ will always be positive. 

Throughout the paper, assume that an immersed line has \emph{rational} slope. 
\begin{defn}
\label{def:Asymptotic directions of a line}
By Corollary \ref{ray-escape}, if $(Y,D)$ is positive, any immersed line $L$ not passing through $0$ has two asymptotic directions. Parametrize $L$ so that the orientation of $L$ agrees with that of $\mathbb{R}$ oriented from $-\infty$ to $+\infty$. We say the line \emph{ $L$
escapes or goes to infinity parallel} \emph{to $q$} (resp. \emph{comes from
infinity parallel to $q$}) if for a maximal cone $\sigma \ni q$ on $B$,
there exists $t_{\sigma}$ such that for any $t>t_{\sigma}$ (resp.
$t<t_{\sigma}$), $l(t)$ is contained in $\sigma$ and $l^{'}(t)=q$
(resp. $-l^{'}(t)=q$) via parallel transport within $\sigma$. We denote by $L_{\infty}$ (resp. $L_{-\infty}$)
the direction in which $L$ escapes to infinity (resp. comes from infinity).
\end{defn}

Given $x\in B_{0}$, consider all immersed straight segments $I_{px}$
with $p\in L$. Then the lemma below follows from Corollary \ref{ray-escape}:
\begin{lem}

\label{lem:minum of wedge exists}Suppose $(Y,D)$ is positive. Then given
a fixed $x\in B_{0}$, there are only finitely many values for $v\wedge\vec{px}$,
where $\vec{px}\in T_{p}B$ is the displacement vector along $I_{px}$ as we vary immersed straight segments $I_{px}$.
\end{lem}

\begin{defn}
\label{def: generalized half space}The \emph{generalized
half space} $Z\left(L\right)$, or simply the\emph{ half space, }of
an immersed line $L$ in $B$ with $d(L)>0$ is the closure of all $x$
in $B_{0}$ such that $v\wedge\vec{px}\geq0$ for all straight immersed line
segments $I_{px}$, as in Lemma \ref{lem:minum of wedge exists}.
\end{defn}

The definition of $Z(L)$ is meant to emulate regular half spaces in usual euclidean geometry. However, we now give an important example to depict why $Z(L)$ can be more complicated than initially imagined. 
\begin{example}\label{ex:dp1} 
Let $(Y,D)$ be the pair where $Y$ is a del Pezzo surface of degree $1$ and $D\in |-K_Y|$ is an irreducible nodal curve where $D^2=1$. We show what $Z(L)$ looks like in $B_{(Y,D)}$. Let $L$ be a line coming from infinity parallel to $\rho_1$. We work on a toric blowup of $(Y,D)$. Since a toric blowup $(Y',D')\rightarrow (Y,D)$ does not change the integral affine structure, we draw the half-space in $B_{(Y',D')}$. We perform one toric blowup to get a pair $(Y',D')$ where $D'=\tilde{D}+E$ where $\tilde{D}$ is the strict transform of $D$ and $E$ is the exceptional curve from the blowup. The self-intersection numbers are:
 $\tilde{D}^2=-3$,  $E^2=-1$. Below is the developing map where $\rho_1$ and $\rho_2$ and correspond to $\tilde{D}$ and $E$ respectively:

\begin{minipage}{\linewidth} \begin{center}  \begin{tikzpicture}[scale=1.5,rotate=90] 
\draw [->] (0,0) -- (0,-1); \node [label=right:\small$\rho_{1}^{1}$] at (0,-1) {}; \draw [->] (0,0) -- (1,0); \node [label=right:\small$\rho_{2}^{1}$] at (1,0) {}; 


\draw [->] (0,0) -- (1,1); \node [label=right:\small$\rho_{1}^{2}$] at (1,1) {}; \draw [->] (0,0) -- (2,3); \node [label=right:\small$\rho_{2}^{2}$] at (2,3) {}; 

\draw [->] (0,0) -- (1,2); \node [label=left:\small$\rho_{1}^{3}$] at (1,1.9) {}; \draw [->] (0,0) -- (1,3); \node [label=left:\small$\rho_{2}^{3}$] at (1,2.9) {}; 


\draw [->] (0,0) -- (0,1); \node [label=left:\small$\rho_{1}^{4}$] at (0,1) {}; \draw [->,bu-red] (0.6,-1) -- (0.6,3); \draw [bu-red](0.6,3)--(0.6,4); \node [label=left:\small$L$] at (0.6,3.9){}; 
\end{tikzpicture} \captionsetup{font=footnotesize,justification=raggedright,format=plain} \captionof{figure}{The developing map for the degree $1$ del Pezzo} \label{fig:developing-map}
\end{center}
\end{minipage}\\

Let $L$ be an immersed line coming from infinity parallel to  $\rho_{1}$. As we can see in Figure \ref{fig:developing-map}, $L$ also goes to infinity parallel to $\rho_{1}$. As shown in Figure \ref{fig:bounded-half-space}, $L$ intersects itself twice and $
Z(L)$ is a bounded polygon in $B_{(Y',D')}=B_{(Y,D)}$.\\

\begin{minipage}{\linewidth} \begin{center}  \begin{tikzpicture}[scale=0.3]
\draw[domain=3*pi/2:2*pi,samples=700] plot ({deg(-(\x))}:{1.2-3.5*cos(\x r))});
\draw[domain=2*pi:3*pi,samples=700] plot ({deg(-(\x))}:{0.7-3*cos(\x r))}); \draw[domain=3*pi:3.5*pi,samples=700] plot ({deg(-(\x))}:{3.5-0.2*cos(\x r))}); \draw (0,3.5) -- (6,3.5);
\draw[domain=pi/2:pi,samples=700] plot ({deg(\x)}:{1.2-4*cos(\x r))});
\draw[domain=pi:1.5*pi,samples=700] plot ({deg(\x)}:{4.6-0.6*cos(\x r))});
\draw [->](0,-4.6) -- (6,-4.6); \node [label=right:\small$\rho_{1}$] at (3.9,0) {}; \draw [dashed](-1.2,0) -- (4,0); \end{tikzpicture} \captionsetup{font=footnotesize,justification=raggedright,format=plain} \captionof{figure}{The half space $Z(L)$ is a bounded polygon.} \label{fig:bounded-half-space} \end{center} \end{minipage}\\\\
\end{example}

The following construction will be used to prove facts about half spaces and polygons. 

\begin{cons}\label{def:construction of fan of cones meeting an immersed line}

Suppose that $(Y,D)$ is positive. Let $L\subset B$ be a line that does not pass through zero with the orientation such that $0$ is on the left. For $L$, there are two (not necessarily distinct) asymptotic rays,  $\rho_{\infty}$ and $\rho_{-\infty}$, parallel to $L_{\infty}$ and $L_{-\infty}$ respectively. We refine $\Sigma$ by adding these asymptotic rays if they are not already there. Now let $\bar{U}\subset B$ be the closure of the cones in   $\Sigma$ that meet $L$. Denote the distance from $0$ to $L$ by $d(L)$. Denote by $\sigma$ the cone containing the $-\infty$ asymptotic end of $L$. Note that one of the generating rays of $\sigma$ is $\rho_{-\infty}$. Denote the other generating ray of $\sigma$ by $\rho$. Now consider the developing map $\delta$ that takes  a lift of $\rho_{-\infty}$ and $\rho$ to the negative $y$-axis and the positive $x$-axis respectively. Then $\delta$ gives a canonical, multi-valued, orientation-preserving immersion
\[ h:\bar{U}\rightarrow \R^2.\]
which takes $L$ to the line $x=d(L)$ with upwards orientation and $p$ to $(d(L),0)$. Here $p$ is the point when $L$ first intersects $\rho$. We can see that:
\[ h(\bar{U})=\{x:x\ge 0\}.\]
We denote $\bar{H}:=h(\bar{U})$ and define
\[H:=\{x:x>0\}.\]
  Then when restricted to $H$, the inverse of $h$ is a single-valued immersion, we call it $g:H\rightarrow B_0$. Under $g$, $L$ is the image of an affine immersion $l:\R\rightarrow B_0$. Then we can get a lift $\tilde{l}:\R\rightarrow \tilde{B}$ by choosing a lift of $p$ on $\tilde{B}$.  Then $\delta\circ\tilde{l}:\R\rightarrow \tilde{B}$ gives a parametrized line. Now let $\tilde{H}$ be the union of cones in $\tilde{\Sigma}$ that meet the image of $\tilde{l}$. Then we get a linear isomorphism $\delta:\tilde{H}\rightarrow H$. Now if we pullback $\Sigma$ along $g$, we get a fan $\bar{\Sigma}$ whose support is $\bar{H}$. Then for each maximal cone $\bar{\sigma}\in \bar{\Sigma}$, $g:\bar{\sigma}\rightarrow\sigma$ is an isomorphism onto a cone of $\Sigma$.

\end{cons}

\begin{lem} \label {lem:half-space-convex}
If $(Y,D)$ is positive, then given an immersed line $L\subset B_0$, the generalized half space $Z(L)$ is a convex polygon.
\end{lem}

\begin{proof}
 Given $x\notin Z(L)$, we want to show that the ray $\rho_x$ (generated by $\overrightarrow{0x}$) intersects $L$. Let $I_{px}$ be a line segment from $p\in L$ to $x$ that points to the right of $L$. Then $d(I_{px},0)$ is a non-constant affine function as we vary $p\in L$. That means there exists $p$ such that $d(I_{px},0)=0$, meaning that the immersed line spanned by $I_{px}$ goes through $0$. However, the immersed ray starting at $p$ in the $\overrightarrow{px}$ direction cannot contain zero since it would imply that $x$ is on the left-hand side of $L$. Thus, we see that $p$ is on the ray $\rho_x$. \\
For the rest of the proof, we use Construction \ref{def:construction of fan of cones meeting an immersed line} and the notation there. Now we wish to show that:

\begin{equation}
Z\left(L\right)^{c}=g\left(\left\{ x\mid x>d\left(L\right)\right\} \right).\label{eq:cl-half}
\end{equation}

As we have shown, if $x \notin Z(L)$, then $x$ is in the image of $g$. Let $I_{px}$ be a line segment that points to the right of $L$. Then because $I_{px}$ is in the image of $g$, we can lift it to a line segment $\tilde{I}_{px}\subset \bar{H}$ from a point $\bar{p}\in \{x=d(L)\}$ to a point $\bar{x}$ such that $\overrightarrow{\bar{p}\bar{x}}$ points to the right of $\{x=d(L)\}$. Thus, we get that $Z(L)^c\subset g(\{x:x>d(L)\})$. The other inclusion is immediate. \

Now given an immersed line segment $I$ with endpoints $x$ and $y$ in $Z(L)$, suppose that $I$ is not contained in $Z(L)$. Then $I$ intersects $L$ at a point $p$ such that either $v\wedge \overrightarrow{px} <0 $ or  $v\wedge \overrightarrow{py} <0 $, in contradiction with the assumption that both $x,y$ are in $Z(L)$. Hence the assumption is false and every immersed line segment with endpoints in $Z(L)$ is contained in $Z(L)$, i.e., $Z(L)$ is convex.   
\end{proof}

For the rest of this subsection, we assume that $(Y,D)$ is positive.

\begin{lem}\label{lem:polygon}
Let $F\subset B$ be a convex polygon. If $E$ is an edge of $F$ and $L$ is the immersed line generated by $E$, then $F\subset Z(L)$ and $F\cap L=E$. 
\end{lem}

\begin{proof}
By our definition of polygons, $0$ is in the interior of $F$. Let $g:H\rightarrow B_0$ be as defined in Construction \ref{def:construction of fan of cones meeting an immersed line}. Define:
\[ \bar{L}=\{x \mid x=d(L)\}.\]
We claim that $g^{-1}(F)$ is convex. Indeed, since $F$ is convex, $g^{-1}(F)$ is locally convex at any vertex $v$. Therefore $g^{-1}(F)$ is convex. Since $0$ is contained in the interior of $F$,  we have:
\[ g^{-1}(F)\subset \{x \mid 0\le x\le d(L)\}.\]
By Equation \ref{eq:cl-half}, 
\[ F\cap Z(L)^c = g ( g^{-1}(F) \cap \{ x \mid x>d(L) \}) = \emptyset. \] Hence, $F \subset Z(L)$.

To show that $F\cap L=E$, it remains to show that $F\cap L\subset E$. Let $q\in F\cap L$ and $p\in E$. Let $I_{pq}\subset L$ be the immersed line segment from $p$ to $q$. We want to show that  $I_{pq}$ is contained in $E$. If not, then $I_{pq}$ will cross an edge, $E'$, of $F$ at a vertex of $F$. Now let $L'$ be the immersed line generated by $E'$. We know that $F\subset Z(L')$ and that $I_{pq}\subset Z(L')$. Notice though that $I_{pq}$ must cross $E'$ from left to right, implying that $q\notin Z(L')$, which is in contradiction with that $I_{pq}\subset Z(L')$. Thus, it must be that $I_{pq}\subset E$. This implies that $F\cap L=E$, as desired.
\end{proof}

\begin{prop} \label{prop:edge-halfspace}
Let $F\subset B$ be a convex polygon. Then $F$ is equal to the intersection of the half spaces spanned by the edges. 
\end{prop}

\begin{proof}
Let $E\subset F$ be an edge and $Z_E$ the corresponding half space obtained by the corresponding immersed line. Then define $P:=\bigcap_E Z_E$. By Lemma \ref{lem:polygon}, $F\subset P$ and $\partial(F)\subset \partial (P)$ which implies that $F=P$.
\end{proof}

By Lemma  \ref{lem:half-space-convex} and Proposition \ref{prop:edge-halfspace}, we have the following corollary:

\begin{cor}
A polygon $F$ is convex  if and only if it is the intersection of generalized half spaces. 

\end{cor}

\subsection{Passing to toric blowups and blowdowns.} \label{ssec:tor-bl}
It will often be convenient to pass to a toric blowup or blowdown of the pair $(Y,D)$. We show here that our family is essentially unaffected by such operations. 
Let $\pi:(Y,D) \rightarrow (Y',D')$ be a toric blowdown of the pair $(Y,D)$ and $L\subset \Pic(Y)$ the subgroup generated by the $\pi$-exceptional divisors. By Lemma $1.6$ of \cite{GHK}, $\Sigma_{(Y,D)}$ is a refinement of $\Sigma_{(Y',D')}$ and $B_{(Y',D')}$ and $B_{(Y,D)}$ are isomorphic as integral affine manifolds. Let $P=\mathrm{NE}(Y)$ and $P'=\mathrm{NE}(Y')$. Then, the surjection $\pi_{*}: A_1(Y) \rightarrow A_1(Y')$, $C \mapsto \pi_{*}(C)$, induces a surjection $P+L \twoheadrightarrow P'$, which further induces a closed embedding $\mathrm{Spec}(\mathbb{C}[P']) \hookrightarrow \mathrm{Spec}(\mathbb{C}[P+L])$. Composing this closed embedding with the open embedding $\mathrm{Spec}(\mathbb{C}[P+L]) \hookrightarrow \mathrm{Spec}(\mathbb{C}[P])$, we obtain a locally closed embedding 
\[
\iota_\pi: \mathrm{Spec}(\mathbb{C}[P']) \hookrightarrow \mathrm{Spec}(\mathbb{C}[P]).
\]

Let $A,\,A'$ be the mirror algebra constructed from the canonical scattering diagram on $B_{(Y,D)}$, $B_{(Y',D')}$ respectively. Let $A_{L}$ be the localization of the  mirror algebra $A$ at $L$. By Proposition 3.10 of \cite{GHK}, there is a surjection $A_L\rightarrow A'$ given by $\theta_p\mapsto \theta_p$ and $z^{[C]}\mapsto z^{[\pi_* C]}$. Thus, we conclude that 
\[ \cX |_{\iota_\pi (\mathrm{Spec}(\mathbb{C}[P']))}=\cX' \]
where $\cX =  \mathrm{Spec}(A)$ and $\cX' = \mathrm{Spec}(A')$.

Moreover, consider the subtorus $T^L\subset T^D$ whose characters are generated by $e_{D_i}$ where $D_i$'s are $\pi$-exceptional divisors. The composition of the weight map $w: B(\mathbb{Z}) \times \NE(Y) \mapsto \chi(T^{D})$ with the projection map $e_{D_i} \mapsto e_{D_i}$ if $D_i$ is a $\pi$-exceptional divisor  and $e_{D_i} \mapsto 0$ otherwise, gives an action of $T^{L}$ on the mirror family. Then, using the identification above, we get the following diagram, where the horizontal arrows are isomorphisms.
\[
\begin{tikzcd}
T^L\times \cX' \arrow[r] \arrow[d]
& \cX|_{\mathrm{Spec}(\mathbb{C}[P+L])} \arrow[d ] \\
T^L\times \Spec(\bC[P']) \arrow[r]
& \mathrm{Spec}(\mathbb{C}[P+L])
\end{tikzcd}
\]

As we mentioned in Remark \ref{rmk:irr-blo}, for the cases where $D$ is an irreducible curve, we will pass to a toric blowup of $D$ at the node, without explicitly mentioning it each time. Thus, every argument we use in the proofs of this section also works for these special cases.\\

At the end of this subsection, we recall the description of the coordinate rings of the fiber $\mathbb{V}_n$ over the zero stratum $0$ of $\Spec(\bC[\NE(Y)])$ in the case where $n=1,2$ since we need them later in this section. 

When $n=1$, we cut $B$ along the unique ray $\rho =\rho_1 \in \Sigma$. Then, the set of linear coordinates $\psi$ identify $B\setminus \{\rho\}$ with a strictly convex cone $\sigma$ in $\mathbb{R}^{2}$. Let $w,w'$ be the direction of the two boundary rays of the closure of $\sigma$ in $\mathbb{R}^2$. Abusing the notation, we use $w+w'$ and  $2w+w'$ to also denote the two distinct directions in $B(\mathbb{Z})$ corresponding to the inverse image of $w+w'$ and $2w+w'$ in $B \setminus \{\rho\}$. Let $\nu$ be the primitive generator of $\rho_1$. Then, by the proof in Subsection of 6.2 in \cite{GHK}, $\theta_{\nu}$, $\theta_{w+w'}$ and $\theta_{2w+w'}$ generate the coordinate ring of $\mathbb{V}_1$ together with the relation
\begin{align}\label{eq:eq-V_1}
   \theta_{2w+w'} \cdot \theta_{\nu} \cdot \theta_{w+w'}     = \theta^{2}_{2w+w'} + \theta^{3}_{w+w'}
\end{align}
which we recognize as the homogeneous coordinate ring of an irreducible nodal cubic embedded in the weighted projective space $\mathbb{P}(3,1,2)$. 

When $n=2$, we can cut $B$ along the ray $\rho_1 \in \Sigma$. Then, the set of linear coordinates $\psi$ identify $B\setminus \{\rho_1\}$ with the union of two strictly convex cones $\sigma_1, \sigma_2$ in $\mathbb{R}$. Let $u,u'$ be the direction of the two boundary rays of the closure of $\sigma \cup \sigma_2$ in $\mathbb{R}^2$. Then, $\theta_{\nu_1}$, $\theta_{\nu_2}$, and $\theta_{u}$ generate the coordinate ring of $\mathbb{V}_2$ together with the relation:
\begin{align}\label{eq:eq-V_2}
    \theta_{\nu_1} \cdot \theta_{\nu_2} \cdot \theta_u = 
    \theta_{u}^{2}.
\end{align}
After coordinate changes, we can identify the coordinate ring of $\mathbb{V}_2$ with $\mathbb{C}[x_1,x_2,y]/(y^2 = x_{1}^{2}x^{2}_{2})$.

\subsection{Compactified Families} \label{ssec:compa-fam}
In toric geometry, an integral convex polytope in the character lattice of an algebraic torus gives rise to a polarized toric variety, as a compactification of the algebraic torus. In this section, we exhibit how polygons in $B_{(Y,D)}$ can be used to compactify the family $\cX\rightarrow\Spec(\bC[\NE(Y)])$.

Given a bounded, rational, convex polygon, $F\subset B_{(Y,D)}$, first consider a toric blowup $p:(\tilde{Y},\tilde{D})\rightarrow (Y,D)$ such that every vertex of $F$ is on a ray of the fan $\tilde{\Sigma}\subset B_{(\tilde{Y},\tilde{D})}$. For each ray $\tilde{\rho}_i \in \tilde{\Sigma}$, let $\tilde{\nu}_i$ be the primitive generator of $\tilde{\rho}_i$. 

Let $\tilde{V}\subset \tilde{A}[T]$ be the free $\bC[\NE(\tilde{Y})]$-submodule with basis $\theta_q T^m$ where $q\in mF\cap B(\Z)$. Let $L \subset \mathrm{Pic}(\tilde{Y})$ be the subgroup generated by $p$-exceptional divisors. Let $\tilde{A}_L$ be the localization of $\tilde{A}$ at $L$ and $\tilde{V}_L \subset \tilde{A}_L[T]$ the free  $\bC[\NE(\tilde{Y})+L]$-submodule with the same basis as $\tilde{V}$. Define the free $\bC[\NE(Y)]$-submodule $V \subset A[T]$ analogously to $\tilde{V}$. Then, $V$ is the image of $\tilde{V}_L$ under the natural surjection $\tilde{A}_L \rightarrow A$, induced by $p_{*}$. 

Since everything we want to claim  about $V$ can be deduced from $\tilde{V}$, for the rest of this subsection, we just assume that we start with a polygon $F$ such that every vertex is on a ray of the fan $\Sigma$, without explicitly mentioning passing to a toric blowup. \\

Recall from Definition \ref{def:weight-map} the weight map $w:B(\Z)\times \NE(Y)\rightarrow \chi(T^D)$.
By rescaling $F$ if necessary, we can assume that $F$ satisfies the following conditions: 
\begin{align}
    \mathrm{For \,each\, ray\,}  \, 
    \rho_i \in \Sigma,\,\rho_i \cap F = \frac{1}{a_i}\nu_i (a_i \in \Z_{>0}). \label{ass-poly}
\end{align}

Consider the associated Weil divisor $W_F =\sum_i a_i D_i$. Let $T^{W_F} \subset T^{D}$ be the 1-parameter subgroup corresponding to the homomorphism $\mathbb{Z}^{D} \rightarrow \mathbb Z$, $e_{D_i} \mapsto a_i$. 
We define the height function 
\[
g:B(\Z)\times \NE(Y) \rightarrow \chi (T^{W_{F}}) \simeq \mathbb{Z} 
\]
by composing the weight function with the homomorphism $\chi(T^{D}) \rightarrow \chi(T^{W_{F}})$ induced by the 1-parameter subgroup $T^{W_F} \subset T^{D}$. 

\begin{lem}\label{nefpoly}
The Weil divisor $W_F$ is nef. Moreover, we have $W_F \cdot D_k > 0$ if and only if $\rho_k$ passes through a vertex of $F$. 
\end{lem} 

\begin{proof}
We assume that the number of rays in $\Sigma$ is $\geq2$ since the statement holds trivially when $\Sigma$ has only one ray. Consider the points $p_i:=\frac{1}{a_i}\nu_i \in \rho_i$ and take straight line segments, $E_{i-1,i}$, that connect $p_{i-1}$ to $p_i$. We claimed that the angle formed at $p_k$ is convex (viewed from $0\in B_{(Y,D)}$) if and only if $(\sum_i a_i D_i)\cdot D_k\geq 0$, with a strict inequality if and only if the angle is strictly convex, which occurs when $\rho_k$ passes a vertex of $F$. To see this, work in the affine chart given by $\sigma_{k-1,k}\cup \sigma_{k,k+1}$ so that $p_{k-1}=\left(\frac{1}{a_{k-1}},0\right)$, $p_k=\left(0,\frac{1}{a_k}\right)$ and $p_{k+1}=\left(\frac{-1}{a_{k+1}},\frac{-D_k^2}{a_{k+1}}\right)$. Then, convexity at $p_k$ means the angle between $E_{k-1,k}$ and $E_{k,k+1}$ is less than or equal to $\pi$. This is equivalent to checking if
\[ \frac{1}{a_k}\frac{1}{a_{k+1}}+\frac{1}{a_{k-1}}\left(\frac{D_k^2}{a_{k+1}}+\frac{1}{a_k}\right)\ge0\]
which gives
\[ a_{k-1}+a_k D_i^2+a_{k+1}=(\sum_i a_i D_i)\cdot D_k \ge0.\]

\begin{minipage}{\linewidth} \begin{center}  \begin{tikzpicture} \fill [pattern= north west lines, pattern color= gray,opacity=0.2] (0,0) rectangle (2,4); \fill [pattern= north west lines, pattern color= gray,opacity=0.2] (0,0) -- (0,4)-- (-2,4) -- (-2,-2) -- (0,0);  \draw (0,0) -- (0,4); \draw (0,0) -- (-2,-2); \draw (0,0) -- (2,0); \draw [->] [blue] (0,1.5) coordinate (a) -- (-1.5,-1.5) coordinate (b); \draw [->] [blue](0,1.5)--(1.5,0)coordinate (c);
\draw (0,0) -- (0,4); \draw (0,0) -- (-2,-2); \draw (0,0) -- (2,0); \draw [->] [blue] (0,1.5)--(-1.5,-1.5); \draw [->] [blue](0,1.5)--(1.5,0); \draw [dotted,blue] (0,1.5)--(-1.5,3); \draw (0.5,-0.2) node {\tiny$\rho _{i-1}$}; \draw (2,-0.3) node {\tiny$\frac{1}{a_{i-1}}(1,0)$}; \draw (0.2,3) node {\tiny$\rho _{i}$}; \draw (0.5,1.7) node {\tiny$\frac{1}{a_{i}}(0,1)$}; \draw (-0.7,-1.7) node {\tiny$\frac{1}{a_{i+1}}(-1,-D^{2}_{i})$}; \draw (-0.3,-0.7)node {\tiny$\rho _{i+1}$};
\draw [dotted,blue] (0,1.5)--(-1.5,3); \draw (0.5,-0.2) node {\tiny$\rho _{i-1}$};
\pic ["$\theta$", draw=gray, <->, angle eccentricity=1.2, angle radius=0.4cm]     {angle=b--a--c}; \end{tikzpicture} \end{center} \end{minipage}\\ \\
Moreover, we obtain the strict inequality if and only if when the angle between $E_{k-1,k}$ and $E_{k,k+1}$ is less than $\pi$, i.e. when $p_k$ is a vertex of the polygon $F$. 
\end{proof}

\begin{lem}\label{lem:integralpolygonpoint}
 Let $g$ be as above. Then $q\in mF\cap B(\Z)$ if and only $g(q)\le m$. 
 \end{lem} 
  \begin{proof}
 Consider $\sigma_{i,i+1}$, a maximal cone in $\Sigma_{(Y,D)}$. Then, $\sigma_{i,i+1}\cap mF$ is a convex triangle with vertices $0$, $\frac{m}{a_i} v_i$ and $\frac{m}{a_{i+1}} v_{i+1}$. Extending $g$ linearly on $\sigma_{i,i+1}$, the function $g$ has values
 \[ g(0)=0,\]
 \[g\left( \frac{m}{a_i} v_i \right)=g\left(\frac{m}{a_{i+1}}v_{i+1} \right)=m.\]
 Furthermore, $g$ is linear when restricted to a maximal cone. Thus, $g$ is bounded above by $m$ when restricted to $\sigma_{i,i+1}\cap mF$. Similarly, $g|_{(mF)^{c}}>m$.
 \end{proof}

\begin{lem}\label{lem:eq-hw}
Given $p,q \in B(\mathbb{Z})$, consider 
\[
\theta_p \cdot \theta_q = \sum_{r\in B(\mathbb{Z})} \alpha(p,q,r)\theta_r.
\]
If $z^{[C]} \theta_r$ is contained in the expansion of  $\theta_p \cdot \theta_q$, then we have 
\[
w(r,C) = w(p) + w(q) \,\,\mathrm{and} \,\, g(r,C) = g(p) + g(q).
\]
\end{lem}
\begin{proof}
The equality between weights follows from Theorem \ref{thm:rel-equi}. The equality between heights follows from  the equality between weights. 
\end{proof}


\begin{lem}
The submodule  $V\subset A[T]$ is a subalgebra. 
\end{lem}

\begin{proof}
Let $p\in nF$ and $q\in mF$ which correspond to $\theta_p$ and $\theta_q$. If $z^{[C]}\theta_r$ is in the expansion of $\theta_p\cdot \theta_q$, then we want to show that $r\in (n+m)F$. Equivalently, by Lemma \ref{lem:integralpolygonpoint},  we want to show that $g(r)\le n+m$. By Lemma \ref{lem:eq-hw}, we have $g(r,C)=g(p)+g(q)$. By Lemma \ref{nefpoly}, we have
\[ g(r)\le g(r) + W_F \cdot C = g(r) + g(C) = g(r,C)=g(p)+g(q)\le n+m.\]
Here, the last inequality follows from Lemma \ref{lem:integralpolygonpoint}.
\end{proof}

\begin{prop}\label{prop:fin-gen}
V is a finitely generated $\bC[\NE(Y)]$-algebra.
\end{prop}
\begin{proof}
The proof is the same as the proof for Proposition 6.6 of \cite{GHK}. We include it here for the convenience of the reader. Let $J= \NE(Y)\setminus \{0\}$ and consider $V/JV$. Then, $\Proj(V/JV)$ is given by the union of toric varieties corresponding to $\underline{F}$, the polyhedral decomposition of $F$ induced by that of $B$ given by $\Sigma \bigcap F$. It follows that $V/JV$ must be finitely generated, implying that there are a finite number of $\theta_qT^m$ in $V$ that generate $V/JV$. Denote the largest $m$ in this generating set by $N$. Next, we show that the finite set $G:=\{\theta_qT^m| m\le N\}$ actually generates $V$. 

Let $V'$ be the subalgebra generated by $G$.  Since the structure coefficients for the mirror algebra are deformation invariant by Lemma 3.9 of \cite{GHK},  we can choose $(Y,D)$ in its deformation class such that there are no internal curves. Then there exists, by the positivity assumption, $W=\sum a_iD_i$, $a_i>0$ such that $W$ is ample on $Y$. Fix a weight $a \in \chi(T^{D})$ of the relative torus $T^D$. For fixed $m$, define the set
\[ V_{m,a}:=\{z^{[C]}\theta_rT^m:z^{[C]}\theta_r\text{ is of weight } a,\,r\in mF\}.\]
We claim that only a finite number of $z^{[C]} \theta_q$ have the weight $a$, i.e., the weight map 
\begin{align*}
    w:B(\Z)\times \NE(Y) & \rightarrow \chi(T^D)
\end{align*}
has finite fibers. Let $\sigma\in \Sigma$ be a maximal cone. Consider the restriction of $w$ to $\sigma(\Z)\times \NE(Y)$. After the restriction, $w$ is linear and the domain is the set of integral points on a cone. It suffices to show that $\mathrm{ker}(w)\cap (\sigma(\mathbb{Z})\times \NE(Y))=0$. Let $(q,C)$ be such that $w(q,C)=0$. We write $q=bv_i+cv_{i+1}$. Then,
\[ w(q,C)=be_{D_i}+ce_{D_{i+1}}+\sum_j (C\cdot D_j)e_{D_j}=0.\] If
$C \in J$, then since $b,c\ge 0$, we must have that $C\cdot D_j\le 0$ for every $j$.  Then $C\cdot D_j\le 0$ implies that $C\cdot W\le 0$ which contradicts the fact that $W$ is ample. It follows that $(q,C)=0$. This shows that $w$ has finite fibers. 

Define the order $\ord([C])$ for $[C]\in \mathrm{NE}(Y)$ to be the maximal integer $k$ such that $[C]=[C_1]+...+[C_k]$ where $[C_i]\in J$. We show that $V_{m,a}\subset V'$ by a decreasing induction on the order of $[C]$. Since $V_{m,a}$ is finite, there is an upper bound on the possible orders.  Therefore, the claim is vacuously true for large orders. Consider $z^{[C]}\theta_r T^m\in V_{m,a}$. Since $G$ generates $V$ modulo $JV$, we can write
\[ \theta_rT^m=\alpha T^m +\beta T^m,\]
where $\alpha T^{m}$ is in $V'$ and  $\beta T^{m}$ is in $JV$. Then,
\[ z^{[C]}\theta_rT^m=z^{[C]}\alpha T^{m} +z^{[C]}\beta T^{m}\]
where the weight of $z^{[C]}\alpha$ is also $a$ and the weight of $z^{[C]}\beta$ is $a$ or $z^{[C]}\beta=0$. The latter implies that $z^{[C]}\theta_rT^m\in V'$. In the former case, $z^{[C]}\beta$ is the sum of terms $z^{[C']}\theta_q$ of weight $a$ where  $\mathrm{ord}([C']) > \mathrm{ord}([C])$ since $\beta \in JV$. By the induction hypothesis, $z^{[C]}\beta$ is in $V'$. Hence, $z^{[C]}\theta_r T^{m}$ is in $V'$. 
\end{proof}
 \begin{remark}
 The equivariant action of $T^{D}$ on $A$ naturally extends to $V$.
 \end{remark}
Thus, given any bounded, convex, rational polygon $F$, there is a canonical homogeneous coordinate ring associated with it. Indeed, a scaling of $F$ will satisfy assumption \ref{ass-poly}. Then, we can construct the corresponding homogeneous coordinate ring $V$ as defined above. Moreover, it is straightforward to see that the isomorphism class of $V$ as a finitely generated $\mathbb{C}[\mathrm{NE(Y)}]$-algebra does not depend on the particular scaling of $F$ we choose. \\

Let $\pi:(Y,D) \rightarrow (Y',D')$ be a toric blowup. We follow the notation used in Subsection \ref{ssec:tor-bl}. Given a polygon $F'$ in $B_{(Y',D')}$, we could also view it as a polygon in $B_{(Y,D)}$. Let $\bar{\cX}$ and $\bar{\cX'}$ be the compactified mirror family we get using $F$, viewed as a polygon in $B_{(Y,D)}$ and $B_{(Y',D')}$ respectively. Then, everything we said about $\cX$ and $\cX'$ in Subsection \ref{ssec:tor-bl} extends to $\bar{\cX}$ and $\bar{\cX'}$. In particular, we obtain the following lemma: 
\begin{lem} \label{lem:tor-bl-compa}
Following the notation used in Subsection \ref{ssec:tor-bl}, we have the  identification
\[ \bar{\cX} |_{\iota_\pi (\mathrm{Spec}(\mathbb{C}[P']))}= \bar{\cX'}  \]
and the following diagram commutes:
\[
\begin{tikzcd}
T^L\times \bar{\cX'} \arrow[r] \arrow[d]
& \bar{\cX}|_{\mathrm{Spec}(\mathbb{C}[P+L])} \arrow[d ] \\
T^L\times \Spec(\bC[P']) \arrow[r]
& \mathrm{Spec}(\mathbb{C}[P+L])
\end{tikzcd}
\]
\end{lem}

Now, we want to study the family $\bar{\cX}:=\Proj(V)\rightarrow\Spec(\bC[\NE(Y)])$ and show that it is a compactification of $\cX$. Moreover, we are specifically interested in the restriction of the family over the algebraic torus $T_Y:=\Pic(Y)\otimes \mathbb{G}_m\subset \Spec(\bC[\NE(Y)])$. \\

\begin{defn}
The \textit{boundary} of $\bar{\cX}$ is the subscheme $\cD$ defined by the equation $T=0$.
\end{defn}


\begin{prop}\label{trivialboundary} 
There is a natural identification $\bar{\cX}\setminus \cD=\cX$.  Suppose $F$ is bounded, convex and rational. Then $\cD|_{T_Y}\rightarrow T_Y$ is a trivial family of nodal elliptic curves with one irreducible component for each face of $F$. 
\end{prop} 

\begin{proof} 
Let $V_T$ be the localization of $V$ with respect to $T$. On the affine scheme $\mathrm{Spec}(V_T)$, the subring of degree $0$ elements of $V_T$ is equal to $A$, which immediately gives $\bar{\cX}\setminus \cD=\cX$. We proceed to show to prove the triviality of $\cD$ over the torus $T_Y$. Consider the quotient ring $V/(T\cdot V)$. By Lemma \ref{lem:integralpolygonpoint}, $\theta_qT^m\in V$ if and only if $g(q)\le m$. Similarly, $\theta_q T^m\in T\cdot V$ if and only if $g(q)<m$. Then $V/(T\cdot V)$ is generated by $\theta_qT^{g(q)}$.  

Let $\Sigma_F$ be the fan over faces of $F$. This fan is a coarsening of the fan $\Sigma$. Given $p,q \in B(\mathbb{Z})$, consider 
\begin{equation}\label{eq:expan}
       \theta_p T^{g(p)}\cdot \theta_q T^{g(q)}=\sum_{s\in B(\Z)}\alpha(p,q,s)\theta_s T^{g(p)+g(q)} \quad \mod T\cdot V. 
\end{equation}
 We make the following claim:
\begin{claim} \label{mult-str} If two pairs of broken lines cross a ray passing through a vertex of $F$ or bend at a ray in the scattering diagram, then they do not contribute to the expansion \ref{eq:expan}. In particular,
if $p,q$ are not in the same $2$-cell of $\Sigma_F$, then
\[ \theta_p T^{g(p)}\cdot \theta_qT^{g(q)} \equiv 0 \mod T\cdot V \]
and if $p,q$ are in the same 2-cell of $\Sigma_F$, only straight broken lines can contribute to the expansion \ref{eq:expan}. 
\end{claim}
Let us prove this claim. Suppose that $p$ and $q$ are not in the same 2-cell of $\Sigma_F$. Then, for a non-zero $\alpha(p,q,s)$, it must be of the form $cz^{([D_{i_0}]+\beta)}$ where ${D_{i_0}}$ is the bending parameter for a ray passing through a vertex of $F$ and $\beta$ is in $\mathrm{NE}(Y)$. By Lemma \ref{nefpoly}, we have 
\[
(D_{i_0} + \beta) \cdot W_F \geq D_{i_0} \cdot W_F > 0. 
\]
By Lemma \ref{lem:eq-hw}, we have 
\[
g(s,D_{i_0}+C) = g(s) + (D_{i_0} + \beta) \cdot W_F = g(p) + g(p).
\]
Therefore, $g(s)<g(p) + g(q)$ and 
\[
z^{([D_{i_0}]+\beta)} \theta_s T^{g(p)+g(q)} \equiv 0 \mod T\cdot V.
\]
Now, suppose  $p,q$ are in the same 2-cell of $\Sigma_F$ and $cz^{[C]}\theta_s T^{g(p)+g(q)}$ is in the expansion \ref{eq:expan}.
If the two broken lines contributing $cz^{[C]}\theta_s$ bend at a wall, $[C]$ must be of the form $[C]= [C']+\beta$ where $[C']$ is a $\mathbb{A}^1$-class and $\beta \in \mathrm{NE}(Y)$.  
But $\beta \cdot W_F>0$ for all $\mathbb{A}^1$-classes $\beta$. Again, by Lemma \ref{lem:eq-hw}, 
\[
g(p)+g(q) = g(s,C) = g(s) + C \cdot W_F \geq g(s) + C' \cdot W_F > g(s).
\]
Therefore, $cz^{[C]}\theta_s T^{g(p)+g(q)}$ is trivial modulo $T\cdot V$. Thus, for $p,q$ in the same 2-cell of $\Sigma_F$, we see that the only broken lines that contribute to the expansion \ref{eq:expan} are the straight ones.  

Now, we are ready to prove the statement of the proposition. First, if $F$ is a 1-gon, by the same computation done in Subsection 6.2 of \cite{GHK}, to get the description of the coordinate ring of $\mathbb{V}_1$, we see that $(V/T \cdot V)$ is isomorphic to the homogeneous coordinate ring of an irreducible nodal rational curve embedded in the weighted projective plane $\mathbb{P}(3,1,2)$.
If $F$ is not a 1-gon, then Claim \ref{mult-str} implies that by restricting to each cone in $\Sigma_F$, modulo $T\cdot V$, we get the homogeneous coordinate ring for the trivial family of smooth rational curves over $T_Y$. Thus, $\cD\mid_{T_Y}$ is a trivial family of nodal elliptic curves with one component for each face of $F$. 
\end{proof}

\begin{lem}\label{lem:cartier} 
If $F$ is bounded, convex, rational, then, $\bar{\cX}\mid_{T_Y} \rightarrow T_Y$ is smooth along $(\cD \setminus \mathcal{N}) \mid_{T_Y}$ where $\mathcal{N} \subset \cD$ is the union of nodes  of $\cD$.
\end{lem}
\begin{proof}
We can assume that $F$ satisfies the condition \ref{ass-poly}. Let $W_F$ be the Weil divisor associated to $F$ and $T^{W_F} \subset T^{D}$ be the 1-parameter subgroup corresponding to $W_F$. By Lemma \ref{nefpoly}, for any effective class $C \in \mathrm{NE}(Y)$, the weight of $z^{[C]}$ is non-negative. Therefore, given any $s \in T_Y$, the limit of $s$ under the action of $T^{W_F}$ exists in $S=\mathrm{Spec}(\mathbb{C}[\mathrm{NE}(Y)])$ and we denote it by $s_0$.

Recall that $\Sigma_F$ is the fan over faces of $F$. Denote by $\delta_s: \mathbb{A}^{1} \rightarrow S$ the morphism given by $T^{W_F} \cdot s$. Denote $\mathcal{D} \setminus \mathcal{N}$ and $\bar{\cX} \setminus \mathcal{N}$ by $\mathcal{D}^{\circ}$and $\bar{\cX}^{\circ} $ respectively. First, we assume that $F$ is not a 1-gon. We claim that the fiber $\bar{\cX}_{s_0}$ over $s_0$ is isomorphic to the union of toric varieties with a polarization given by the polyhedral decomposition $\Sigma_F \cap F$. Indeed, since any $\mathbb{A}^1$-class has a positive intersection with $W_F$, there is no scattering restricted to $s_0 \in S$. Then, the claim follows from that the statement $W_F \cdot D_k >0$ if and only if $\rho_k$ passes through a vertex of $F$ in Lemma \ref{nefpoly}. In particular, we conclude that $\mathcal{D}^{\circ}_{s_0}$ is a Cartier divisor on $\bar{\cX}^{\circ}_{s_0}$ and $\cD^{\circ} \times_S \mathbb{A}^{1} \rightarrow \mathbb{A}^{1}$ is the trivial family. Notice that these conclusions still hold even when $F$ is a 1-gon.

By Proposition \ref{prop:fin-gen}, $f: \bar{\cX}^{\circ} \times_{\mathbb{A}^1}  S \rightarrow \mathbb{A}^1$ is of finite type. Since  $\mathbb{A}^1$ is Noetherian, by \cite[\href{https://stacks.math.columbia.edu/tag/01TX}{Tag 01TX}]{Sta}, $f$ is of finite presentation. Since  $\bar{\cX}^{\circ} \times_S  {\mathbb{A}^1} \rightarrow \mathbb{A}^1$ is flat and $\cD^{\circ} \times_S \mathbb{A}^1 \rightarrow \mathbb{A}^{1}$ is the trivial family, by \cite[\href{https://stacks.math.columbia.edu/tag/062Y}{Tag 062Y}]{Sta}, for each $x \in \mathcal{D}^{\circ}_{s_0}$, there exists an open neighborhood $U$ with $ x \in U \subset \bar{\cX}^{\circ} \times_S {\mathbb{A}^1}$ and a relative effective Cartier divisor $\cD' \subset U$ such that $ \cD^{\circ} \cap U = \cD'$ and $\cD'_{s_0} = \cD^{\circ}_{s_0} \cap U$. 
 
 Since $T^{W_F}$ can push $s$ to be arbitrarily close to $s_0$, the action of $T^{D}$ on $\bar{\cX}$ is equivariant, and $\cD^{\circ} \times_S \mathbb{A}^1 \rightarrow \mathbb{A}^{1}$ is the trivial family, we conclude that for each $x \in \cD^{\circ}_s$, $\cD^{\circ}_s \subset \bar{\cX}^{\circ}$ is Cartier near $x$. Hence, we conclude that $\bar{\cX}^{\circ}_s$ is smooth along $\cD^{\circ}_s$. Since $s$ is an arbitrary point in $x\in T_Y$, the statement of the lemma follows.
\end{proof}

Now, we are ready to prove one of the main results in this section, which is the following proposition.


\begin{prop} \label{compa-fiber} 
Fix a bounded, nonsingular, rational convex polygon, $F\subset B_{(Y,D)}$ and let $(\bar{\cX},\cD)\rightarrow S = \Spec(\bC[\NE(Y)])$ be the corresponding family. Then, the generic fiber over $T_Y$ is smooth and the restriction $\bar{\cX}\mid_{T_Y}\rightarrow T_Y$ is smooth along $\cD \mid_{T_Y}$. Let $(\bar{\cX}_t,\cD_t)$ be the fiber of $(\bar{\cX},\cD)\rightarrow T_Y$ over $t\in T_Y$. Then $\cD_t\subset \bar{\cX}_t$ is an anti-canonical cycle of rational curves contained in the smooth locus and $\bar{\cX}_t$ has, at worst, du Val singularities. In particular, $(\bar{\cX}_t, \cD_t)$ is a generalized pair in the sense of Definition \ref{generalized pair}. 
\end{prop}

\begin{proof} 

Recall that in Proposition \ref{trivialboundary}, we proved that $\cD\mid_{T_Y}$ is a trivial family of nodal elliptic curves with one component for each face of $F$. By Lemma \ref{lem:cartier},  to show that $\bar{\cX}\mid_{T_Y}$ is smooth along $\cD \mid_{T_Y}$,   it remains to check that $\bar{\cX}\mid_{T_Y}$ is smooth near every node of $\cD \mid_{T_Y}$. 

Fix a node of $\cD \mid_{T_Y}$, which corresponds to a vertex of the polygon, say $v_i$. Then, since every vertex lies on a ray by our assumption, $v_i$ also corresponds to a boundary divisor $D_i\subset D$. By replacing $(Y,D)$ with a toric blowup as in Lemma \ref{lem:tor-bl-compa}, we could assume that $D_i$ is contractible (i.e. with negative self-intersection) and  the number, $n$, of irreducible components in $D$ is $\geq 3$. We emphasize that by this procedure we do not change the polygon $F$. The original compactified mirror family now embeds into the compactified mirror family after the toric blowup as in Lemma \ref{lem:tor-bl-compa}, but for the cleanliness of the notation, we only implicitly acknowledge this.  

Now, let $\pi:(Y,D)\rightarrow (Y',D')$ be the (weighted) toric blowdown that contracts $D_i$ and $Q\subset \NE(Y)$ be the corresponding toric monomial ideal generated by all curve classes $z^{[C]}$ for $C$ not blowdown to a point under $\pi$. Thus, the only monomials not in $Q$ are all of the form $z^{[mD_i]}$. Now, consider  $f:\operatorname{Proj}(V/Q)\rightarrow \operatorname{Spec}(\bC[\NE(Y)]/Q)$. This family is a purely Mumford degeneration since all the $\mathbb{A}^1$-classes that can appear are contained in $Q$.  In  the family given by $f$, the double curve corresponding to the ray $v_i$ is smoothed out by  the construction of the Mumford degeneration. Since the polygon is smooth, the family $f$ is smooth near the node in consideration. 

To further get the smoothness near the node in consideration for $\bar{\cX} \mid_{T_Y} \rightarrow T_Y$, we use the equivariant action of $T^{D}$ on the restricted compactified family. Let $H =\sum_j a_j D_j \,(a_j > 0)$ be a $D$-ample divisor. By rescaling $H$ if necessary, we can assume that $\pi_{*}H$ is Cartier. Let $H' = \pi^{*} \pi_{*}H = \sum_j a'_jD_j$. Consider the one-parameter subgroup $T^{H'} \subset T^{D}$ given by $\chi(T^{D}) \rightarrow \mathbb{Z}$, $e_{D_j} \mapsto a'_j$. We claim that $H' \cdot D_i =0$ and $H' \cdot D_k >0$ for $k \neq i$. Indeed, notice that $a_j = a_j'$ for $j \neq i$. The claim holds when $k \neq i-1,i,i+1$ just by our assumption that $n \geq 3$. The fact that $H' \cdot D_i =0$ follows from the fact that $\pi^{*}(a_j\pi_{*}D_{j}) \cdot D_i =0$ for $j=i-1,i+1$. Write $ \pi^{*}\pi_{*}(a_{i-1}D_{i-1}) = a_{i-1}D_{i-1} + aD_i $ and  $ \pi^{*}\pi_{*}(a_{i+1} D_{i+1}) = a_{i+1}D_{i+1} + b{D_i}$. Then, it follows from  $(\pi^{*}\pi_{*}D_{j}) \cdot D_i =0$ for $j=i-1,i+1$, that $H' \cdot D_i =0$, $a_{i-1} + aD_{i}^{2} =0$ and $a_{i-1} + bD_{i}^{2} =0$. Since 
\[
H\cdot D_i = a_i D_{i}^{2} + a_{i-1} + a_{i+1} >0 
\]
and $D^{2}_i <0$, we conclude that $a+b > a_i$. Therefore, 
\[
H' \cdot D_{i-1}   = a_{i-1}D_{i-1}^{2} + (a+b) + a_{i-2} >  a_{i-1}D_{i-1}^{2} + a_i + a_{i-2} = H \cdot D_{i-1} >0.
\]
Similarly, we conclude that $H' \cdot D_{i+1}$ >0. Thus, the claim holds. Hence, for any point in $T_Y$, we can use the 1-parameter subgroup $T^{H'}$ to push it near the boundary stratum $\mathrm{Spec}(\mathbb{C}[\mathrm{NE}(Y)]/Q)$. Since the condition of being smooth near the node is open in the base, we conclude that the entire family $\bar{\cX}\mid_{T_Y} \rightarrow T_Y$ is smooth near the node in consideration. 

By repeating the previous arguments for each node of $\cD \mid_{T_Y}$, we conclude that $\bar{\cX}\mid_{T_Y} \rightarrow T_Y$ is smooth along $\cD \mid_{T_Y}$. Thus, combining this result with Corollary $6.11$ of \cite{GHK}, we obtain that each fiber $(\bar{\cX}_t,\cD_t)$ in  $\bar{\cX}\mid_{T_Y} \rightarrow T_Y$ is Gorenstein and semi-log canonical and the generic fiber is smooth. 

 Recall that relative sheaf $\omega_{\cX /S}$ is trivial by Proposition 2.31 \cite{GHK}. Let $0$ be the unique fixed toric point of $S$ and $(X_0,D_0)$ the fiber over $0$.  Consider the 1-parameter $T^{H} \subset T^{D}$ given by $\chi(T^{D}) \rightarrow \mathbb{Z},$ $e_{D_i} \mapsto a_i$. Then, given any $t \in T_Y$, we can push $t$ arbitrarily close to $0$ using $T^H$. Notice that $\omega_{X_0}(D_0)$ is trivial. By Lemma 8.42 of \cite{GHKK}, $\omega_{\bar{\cX}_t}(\cD_t)$ is trivial and $H^1(\bar{\cX}_t,\cO_{\bar{\cX}_t}) =0$ since $(X_0,D_0)$ satisfies these conditions. 

Next, we show that $\cD_t$ is anti-canonical and $\bar{\cX}_t$ has at worst du Val singularities. We first show that $\bar{\cX}_t$ has isolated singularities.  Taking the closure of the orbit $T^{H}\cdot t$, we obtain a morphism $\mathbb{A}^1 \rightarrow \Spec(\mathbb{C}[\NE(Y)])$. Then, the pullback of $\cX \rightarrow \Spec(\mathbb{C}[\NE(Y)])$ along $\mathbb{A}^1$ has the central fiber $\mathbb{V}_n$. We use an argument from the proof for Proposition 7.4 of \cite{HKY}. Suppose that the singular locus of $\cX_t$ is of codimension $1$. Then, the intersection of the  closure of the singular locus of $\cX \mid _{\mathbb{G}_m}$ (where $\mathbb{G}_m \subset \mathbb{A}^1$) with the central fiber has codimension at most $1$. Then, the total space $\cX \mid _{\mathbb{A}^1}$ is not normal at the generic point of some double curve in $\mathbb{V}_n$. However, if we base change to the completion of $\mathbb{C}[\NE(Y)]$ at the maximal monomial ideal $\mathfrak{m}$,  a neighborhood in $\cX$ of the interior double curve is isomorphic to the toric Mumford degeneration with the explicit equation given by Equation 0.5 of \cite{GHK}, which is clearly normal. Hence, a contradiction occurs and $\cX_t$ and therefore $\bar{\cX}_t$ has isolated singularities. Now, it follows from Proposition 7.2 of \cite{HKY} that $\bar{\cX}_t$ have, at worst, du Val singularities. 

Let $\tilde{X}_t$ be the minimal resolution of $\bar{\cX}_t$. Since $\bar{\cX}_t$ has at worst du Val singularities and is smooth along $\cD_t$, the strict transform $\tilde{D}_t$ of $\cD_t$ in $\tilde{X}_t$ is anti-canonical and $H^1(\tilde{X}_t,\cO_{\tilde{X}_t}) = H^1(\bar{\cX}_t,\cO_{\bar{\cX}_t})=0$. By the classification of surfaces, $\tilde{X_t}$ is rational. Therefore, $(\bar{\cX}_t,\cD_t)$ is a generalized pair in the sense of Definition \ref{generalized pair}.

\end{proof}

\begin{defn}
Since $(Y,D)$ is positive, by Corollary \ref{ray-escape}, for any line segment on $B$ that can be extended to an immersed line $L$ not passing through $0$, there exists a ray with endpoint $0\in B$ which is generated by the direction parallel to that which $L$ escapes to $\infty$. We call this ray the \emph{escape ray} of the line segment.  
\end{defn}

\begin{prop}\label{prop:paralleledgeiso} 
Let  $F,F' \subset B$ be two bounded, rational, convex polygons. Suppose two edges $E \subset F$ and $E'  \subset F'$ have the same escape ray and $mE$, $m'E'\,(m,m'\in \mathbb{N})$ share two adjacent  points $p_1,p_2 \in  B(\mathbb{Z})$. We further require $p_1,p_2$ to be interior if one of $F,F'$ is a 1-gon. Let $(\bar{\cX},\cD)$ and $(\bar{\cX}',\cD')$  be the compactified mirror families we get using $F$ and $F'$ respectively. Let $\cD^{c}_E$ be the sum of boundary components other than $\cD_E$ if $F$ is not a 1-gon. Otherwise,  let $\cD^{c}_E$ be the complement of the node of $\cD$. Define  $\cD'^{c}_{E'}$ similarly. Then,  the rational map $\bar{\cX} \dashrightarrow \bar{\cX'}$ restricts to an isomorphism
\[ (\bar{\cX} \setminus \cD_E^c)|_{T_Y} \rightarrow (\bar{\cX}' \setminus \cD'^{c}_{E'})|_{T_Y}.\]
\end{prop}

\begin{proof} For the convenience of notation, we assume that $F$ and $F'$ satisfy assumption \ref{ass-poly}. Let $V,V'$ be the homogeneous coordinate rings we get from $F$ and $F'$. Consider $X_{p_i} = \vartheta_{p_i} T^{m} \in V $ and  $X'_{p_i} = \vartheta_{p_i} T^{m'} \in V'(i=1,2)$. Then, $X_{p_1}\cdot X_{p_2}$ and $X'_{p_1}\cdot X'_{p_2}$ vanishes on $\cD^{c}_{E}$ and $\cD'^{c}_{E'}$ respectively by Claim \ref{mult-str}. Let $U \subset (\bar{\cX} \setminus \cD^{c}_{E})|_{T_Y}$ and $U' \subset (\bar{\cX}' \setminus \cD'^c_{E'})|_{T_Y}$ be the affine charts given by $X_{p_1}\cdot X_{p_2}\neq 0$ and $X'_{p_1}\cdot X'_{p_2}\neq 0$ respectively. By Proposition \ref{trivialboundary}, $(\bar{\cX} \setminus \cD)|_{T_Y}$ and $(\bar{\cX}' \setminus \cD')|_{T_Y}$ are isomorphic to $\cX|_{T_Y}$. Moreover, $U$ and $U'$ contains the boundary of $(\bar{\cX} \setminus \cD^{c}_{E})|_{T_Y}$ and $(\bar{\cX}' \setminus \cD'^{c}_{E'})|_{T_Y}$ respectively. Thus, to prove the proposition, it suffices to show that the rational map $U \dashrightarrow U'$ is an isomorphism. Notice that $\vartheta_{p_1} / \vartheta_{p_2}$ are regular on  the boundaries  of both $(\bar{\cX} \setminus \cD_E^c)|_{T_Y}$ and $(\bar{\cX}' \setminus \cD'^{c}_{E'})|_{T_Y}$. Therefore, the rational map $U \dashrightarrow U'$ induces an isomorphism between the boundaries of $(\bar{\cX} \setminus \cD_E^c)|_{T_Y}$ and $(\bar{\cX}' \setminus \cD'^{c}_{E'})|_{T_Y}$. Since $U \dashrightarrow U'$ is also an isomorphism away from the boundaries, it is actually an isomorphism. 
\end{proof}

\subsection{The Canonical Compactification}\label{ssec:cano-poly} 
We are now ready to construct the polygons that will be used to compactify the mirror family. 

\begin{defn}  \label{def:para-poly}
Let $W=\sum_i a_{i}D_{i}$ be a Weil divisor with $a_{i}>0$ such that $W \cdot D_i \geq 0$ for each $i$. Let $L_i$ be the immersed line going to infinity which is parallel to $\rho_i$ and lattice distance $a_{i}$ from $\rho_{i}$. Then, we define the polygon associated with $W$ to be:
\begin{align}
    P(W):=\bigcap_{i} Z(L_{i}). \label{inter-hal}
\end{align}
We call such polygon $P(W)$ the \emph{parallel polygon} associated with the Weil divisor $W$. 
\end{defn}

Let us first consider the cases where $(Y,D)$ has a toric model. For each $i$, let $k_i$ be the number of non-toric blowups along $\bar{D}_i$ for the given toric model $p:(Y,D)\rightarrow (\bar{Y},\bar{D})$.  Denote the toric fan for $\bar{Y}$ by $\bar{\Sigma}$. The Weil divisor, $\bar{W}=\sum a_{i}\bar{D}_{i}$, determines a polytope
\[ P(\bar{W}):=\{m\in \bar{M}:\bar{v}_i \wedge m\ge -a_i \text{ for all } i\}\]
where $\bar{v}_i$ is the first lattice point on the ray $\bar{\rho_i}$. Then, for every $\bar{\rho}_i$, there is an edge $\bar{F}_i\subset P(\bar{W})$ such that $\bar{F}_i$ is parallel to $\bar{\rho}_i$ and is at distance $a_i$ away from it on the right. Furthermore, by standard toric geometry, the length $l(\bar{F}_i)$ of the edge $\bar{F}_{i}$ is $\bar{W}\cdot \bar{D}_i$. We have the following lemma immediately:

\begin{lem} \label{lem:edge-length}
$W\cdot D_{j}>0$ if and only if $l(\bar{F}_{j})>k_{j}a_{j}$,
where $k_{j}$ is the number of non-toric blow up along $\bar{D}_{j}$.
\end{lem}

\begin{proof}
This can be directly computed:
\begin{align*}
W\cdot D_{j} & =a_{j+1}D_{j+1}\cdot D_{j}+a_{j}D_{j}^{2}+a_{j-1}D_{j-1}D_{j}\\
 & =a_{j+1}\bar{D}_{j+1}\cdot\bar{D}_{j}+a_{j}(\bar{D}_{j}^{2}-k_{j})+a_{j-1}\bar{D}_{j-1}\bar{D}_{j}\\
 & =\bar{W}\cdot\bar{D}_{j}-k_{j}a_{j}
\end{align*}
where $\bar{W}=\sum a_i\bar{D}_i$.
\end{proof}

\begin{defn}
Recall that $(Y,D)$ is positive if and only if there exists $a_i >0$ $(1\leq i \leq n)$ such that $W \cdot D_j >0$ for each $j$ where $W=\sum_{i=1}^{n} a_iD_i $. We call such Weil divisor $W$ a \emph{$D$-ample Weil divisor}. 
\end{defn}

\begin{lem}\label{lem:const-par-poly}
Suppose $(Y,D)$ has a toric model and $W$ is a $D$-ample Weil divisor. Then, $P(W)$ is a bounded, convex, nonsingular, integral polygon with $0$ in its interior. Moreover, for each $i$, $Z(L_i)$ is a supporting half space of $P(W)$. 
\end{lem}

\begin{proof} 
We argue  that $P(W)$ is as claimed indirectly by applying a construction of Symington \cite{Sym}. By Lemma \ref{lem:edge-length}, we can draw a triangle whose base has length $k_j a_j$ and is on the edge $\bar{F}_{j}$ such that the base of the triangle avoids the vertices of $\bar{F}_j$. For each triangle, cut out the cone spanned by the interior of the triangle. Then, we glue the two rays of each cone and thus get an integral affine manifold $B'$ with the singularity at the origin $0$. Moreover, if we denote what is remaining of $P(\bar{W})$ after cutting and gluing by $P(W)'$, then $P(W)'$ is a bounded, convex, nonsingular, integral polygon on $B'$ with $0$ in its interior and each edge corresponding to an edge in $P(\bar{W})$. 

The integral affine structure of $B'$ is defined by the requirement that each edge in $P(W)'$ remains straight. Consider a triangle that has its base on an edge $\bar{F_j}$ with endpoints $a$ and $a'$. Then, $a$ and $a'$ are identified in $B'$ and a piecewise linear function, $f$, on $B'$ is linear if and only if:
\[ f(a+v)+f(a-v)=2f(a)\]
for any vector $v$ that is parallel to $\bar{F}_j$. 

Let $F'_i$ be the edge
of $P(W)'$ corresponding to the edge $\bar{F}_i$ in $\bar{P}(W)$. For each $i$, let $R_{i}$ be the forward ray spanned by $F'_i$, i.e., the ray containing $F'_i$ whose endpoint is the intersection point of $\bar{F}_i$ and $\bar{F}_
{i-1}$ and whose orientation agrees with the canonical orientation of the immersed line spanned by $F'_i$.   Then, we parallel transport $R_i$ along $F'_{i+1}$. As we do this, the distance between $R_i$ and $0$ decreases, and at some point, the distance will be $0$, which means we get a ray $\rho'_i\in B'$ emanating from $0 \in B'$ and parallel to $F'_i$. Then the collection of rays $\rho'_i$ gives a fan $\Sigma'$ in $B'$. Notice that $\rho'_i$ is parallel to $F'_i$. 

Note that there is a natural orientation preserving bijection $\delta$ between  $\Sigma$ and $\Sigma'$ which sends $\rho_i$ to $\rho_i'$ and is linear on each cone in $\Sigma^{'}$. We show that this bijection is an isomorphism between integral affine manifolds. Denote by $v_{i}$ the vertex shared by
$F_{i}^{'}$ and $F_{i+1}^{'}$. Let $w_{i},w_{i+1},w_{i-1}$ be the
primitive vectors in $T_{v_{i}}B^{'}$ parallel to $\rho_{i}^{'}$,
$\rho_{i+1}^{'}$ and $\rho_{i-1}^{'}$ respectively. It follows immediately
from our definition of the integral affine linear structure of $B^{'}$
that 
\[
D_{i}^{2}w_{i}+w_{i-1}+w_{i+1}=0.
\]
Hence, $\delta$ is an isomorphism of integral affine linear manifolds. It is clear that under this identification, $P(W)'$ is identified with $P(W)$ and therefore the statement of the lemma  holds for $P(W)$ since it holds for $P(W)'$. 
\end{proof}

\begin{prop}\label{constructpoly}
For any positive Looijenga pair $(Y,D)$ and a $D$-ample Weil divisor  $W$ on $Y$, the  polygon   $P(W)$ is bounded, convex, nonsingular, integral, and with $0$ in its interior. Moreover, each half space in the intersection \ref{inter-hal} is a supporting half space for $P(W)$.  
\end{prop}
\begin{proof}
If $(Y,D)$ has a toric model, then the statement of this Proposition follows immediately from Lemma \ref{lem:const-par-poly}. 
For the general case, let $\pi:(Y',D') \rightarrow (Y,D)$ be a toric blowup such that $(Y',D')$ has a toric model $p:(Y',D')\rightarrow (\bar{Y}', \bar{D}')$. Let $W' = \pi^{*}W$ and $D'_i$ be the strict transform of $D_i$ for each $i$. Furthermore, let us assume that $\pi$ is minimal, i.e., $\pi$ does not factor through a non-trivial toric blowdown $(Y',D') \rightarrow (Y'', D'')$ such that $(Y'',D'')$ has a toric model.  Then, the Weil divisor $\bar{W}' = p_*(W')$ is a $\bar{D'}$-ample divisor. 

By Lemma 1.6 of \cite{GHK}, $\Sigma_{(Y',D')}$ is a refinement of $\Sigma_{(Y,D)}$ and $B_{(Y,D)}$ and $B_{(Y',D')}$ are isomorphic as integral affine manifolds. Let $P(W')$ be the parallel polygon on $B_{(Y',D')}$ corresponding to $W'$. Consider the toric parallel polygon $P(\bar{W'})$ in $\mathbb{R}^{2}$. If we apply Symington's construction \cite{Sym} as in Lemma \ref{lem:const-par-poly}, then for each $j$, since $W' \cdot D'_j = W\cdot D_j >0$, we draw a triangle whose base is in the interior of the edge parallel to the ray corresponding to $p_*(D'_j)$.  If $E$ is an exceptional divisor of $\pi$, we have $W'\cdot E=0$. Then, we draw a triangle whose base is the entire edge parallel to the ray corresponding to $p_*(E)$. 
Then, after deleting the cones spanned by the interior of the triangles, we draw and identify the sides of each cone. As shown in Lemma \ref{lem:const-par-poly}, we get an integral affine manifold $B'$ isomorphic to $B_{(Y',D')}$. The remaining part of $P(\bar{W}')$ on $B'$ is the parallel polygon $P(W')$ and the supporting half spaces of $P(W')$ are precisely half spaces corresponding to original rays in $\Sigma_{(Y,D)}$. Since $B_{(Y',D')}$ has the same integral affine structure as $B_{(Y,D)}$, $P(W)=P(W')$. The statements that $P(W)$ is bounded, integral, and contains $0$ in its interior is straightforward. By Lemma \ref{lem:const-par-poly} and the fact that $\Sigma_{(Y',D')}$ is a refinement of $\Sigma_{(Y,D)}$, near each vertex of $P(W)$, the chopping down of $P(W)$ gives a convex polygon. Therefore, $P(W)$ itself is convex. That $P(W)$ is nonsingular follows from the parallel construction of the polygon and the fact that wedge products between integral vector fields are preserved under parallel transport.
\end{proof}

Thus, we have shown that for each  $D$-ample Weil divisor $W$, we get a bounded, nonsingular, integral, convex polygon $P(W)$ which then gives a compactification $\bar{\cX}(P(W)) \rightarrow \mathrm{Spec}(\mathbb{C}[\mathrm{NE}(Y)])$ of the mirror family $\cX \rightarrow \mathrm{Spec}(\mathbb{C}[\mathrm{NE}(Y)])$. As we will show in Proposition \ref{weilnomatter}, the compactification is actually canonical, in the sense that it is not affected by the choice of the $D$-ample Weil divisor $W$. 

\begin{lem}\label{lem:tor-bl-poly} 
Let $(Y,D)$ be a positive Looijenga pair and $W$ a $D$-ample Weil divisor on $Y$. Let $\pi:(Y,D) \rightarrow (Y',D')$ be a toric blowdown and $W' = \pi_{*}(W)$. If $D_j$ is not an exceptional divisor of $\pi$, then $L_j$ is the supporting half space of both $P(W)$ and $P(W')$. Moreover, in this case, if we let $F_j$ and $F'_j$ be the edges lying on $L_j$ in $P(W)$ and $P(W')$ respectively, then $F_j$ is contained in $F'_j$.
\end{lem}
\begin{proof}
 Write $\pi^{*}(W')=\sum_i a'_iD_i$. We want to show that $a_i = a'_i$ if $D_i$ is not an exceptional divisor and $a'_i > a_i$ if $D_i$ is. It suffices to prove this for the case where $\pi$ has only one exceptional divisor $D_k$. In this case, $a_i = a'_i$ for $i\neq k$ is trivial. Since \[
\pi^{*}(W') \cdot D_k = -a'_k + a_{k-1} +a_{k+1} = 0 < W \cdot D_k = -a_k + a_{k-1} +a_{k+1},
\]
we have $a_k < a'_k$.

By Proposition \ref{constructpoly}, for each $j$, $L_j$ is a supporting half space of $P(W)$ and it is also a supporting half space for $P(W')$ if and only if $D_j$ is \emph{not} an exceptional divisor of $\pi$. If this is the case, by the numerical condition we just proved in the previous paragraph, it is easy to see that $F_j$ is contained in $F'_j$. 
\end{proof}

Let $E\subset P(W)$ be an edge and denote the corresponding boundary divisor by $\cD_E$. Let $\cD_E^{\circ} \subset \cD_E$ be the complement of the nodes of $\cD$.
\begin{prop}\label{thetamarking}
Let $p,q\in E\subset P(W)$ be consecutive lattice  points on the same edge. If $n=1$, we require these two lattice points are in the interior of the edge. Then, the rational function
$\vartheta_{p}/\vartheta_{q}$ is regular generically on $\mathcal{D}_{E}$ and restricts to a generator of the ring of regular function for each fiber in $\cD_E^{\circ}\mid_{T_Y}\rightarrow T_Y$.
\end{prop}

\begin{proof}
When $n\geq 2$, by Proposition \ref{trivialboundary},  $\mathcal{D}_E \rightarrow \Spec(\bC[\NE(Y)])$ is a purely Mumford degeneration as in Section 1.3 of \cite{GHK}. Then, the statement of the proposition follows.

When $n=1$, again, by Proposition \ref{trivialboundary},  $\mathcal{D}_E^{\circ} \rightarrow \Spec(\bC[\NE(Y)])$ is a Mumford degeneration and statement of the proposition still holds for this case.
\end{proof}

\begin{prop}\label{weilnomatter}
Let $W$ and $W'$ be two $D$-ample Weil divisors. Then the compactified families $\bar{\cX}(P(W))$ and $\bar{\cX}(P(W'))$ are isomorphic over the torus $T_Y$.
\end{prop}

\begin{proof}
When $n=1$, the statement of the proposition follows trivially since rescaling the polygon does not change the homogeneous coordinate ring we get. Assume then, that $n\geq 2$. Fix a ray $\rho_i$. By the construction of parallel polygons, there are edges $E\subset P(W)$ and $E'\subset P(W')$ that have the same escape ray as $\rho_i$ and satisfy the conditions of Proposition \ref{prop:paralleledgeiso}. Since this is true for every $\rho_i$, the result follows from Proposition \ref{prop:paralleledgeiso}.
\end{proof}

\begin{defn}\label{def:cano-camp}
We call the compactification of the mirror family using the parallel polygon associated with a $D$-ample divisor the \emph{canonical compactification} and use $(\bar{\cX},\cD)$ specifically to denote the canonical compactification  for the rest of this paper. The canonical compactification is not unique over the entire base $\Spec(\mathbb{C}[\NE (Y)])$. But over the torus $T_Y \subset \Spec(\mathbb{C}[\NE (Y)])$, the locus we mainly care about in this paper, the canonical compactification is indeed unique by Proposition \ref{weilnomatter}. 
\end{defn}

\begin{prop}\label{can-comp-fiber}
Consider the restriction of the canonical compactification $(\bar{\cX}, \cD)\mid _{T_Y} \rightarrow T_Y$. Then, the statements in Proposition \ref{trivialboundary} and Proposition \ref{compa-fiber} hold for the restricted, canonically compactified mirror family. In particular, for each $t\in T_Y$, $(\bar{\cX}_t,\cD_t)$ is a generalized marked pair. 
\end{prop}
\begin{proof}
By Proposition \ref{weilnomatter}, we see those parallel polygons satisfy the conditions in Proposition \ref{trivialboundary} and Proposition \ref{compa-fiber} and their statements hold for the restricted, canonically compactified mirror family.
\end{proof}

The following lemma shows that the canonical compactifications are compatible with toric blowups or blowdowns. 
\begin{lem} \label{lem:compa-para-tor-bl}
Let $(Y,D)$ be a positive Looijenga pair and $W$ a $D$-ample Weil divisor on $Y$.  Let $\pi:(Y,D) \rightarrow (Y',D')$ be a toric blowdown. Let $\bar{\cX}$ and $\bar{\cX}'$ be canonical compactifications of the mirror families $\cX$ and $\cX'$ respectively. Consider the locally closed embedding $\iota_{\pi}: \mathrm{Spec}([\mathrm{NE}(Y')]) \hookrightarrow \mathrm{Spec}([\mathrm{NE}(Y)])$ as in Lemma \ref{lem:tor-bl-compa}. Then, given any $t \in T_{Y'}$, $\bar{\cX}_{\iota_{\pi}(t)}$ is a toric blowup of $\bar{\cX}'_t$ with exceptional divisors in bijection with edges in $P(W)$ parallel to rays $\rho_i$'s such that $D_i$ is an exceptional divisor of $\pi$. 
\end{lem}
\begin{proof}
Let $W'=\pi_{*}W$. Denote  by $\bar{\cX}^{W'}$ the compactification of $\cX$ using $P(W')$. Then, by Lemma \ref{lem:tor-bl-compa}, we have the identification $\bar{\cX}^{W'}_{\iota_{\pi}(t)} \simeq \bar{\cX}'_t$. Thus, to prove the lemma, it suffices to show that $\bar{\cX}_{\iota_{\pi}(t)}$ is a toric blowup of $\bar{\cX}^{W'}_{\iota_{\pi}(t)}$ with exceptional divisors in bijection with edges in $P(W)$ parallel to rays $\rho_i$'s such that $D_i$ is an exceptional divisor of $\pi$. That $\bar{\cX}_{\iota_{\pi}(t)}$ is a toric blowup of $\bar{\cX}^{W'}_{\iota_{\pi}(t)}$ follows from Proposition \ref{trivialboundary}, Proposition \ref{compa-fiber}, and Proposition \ref{prop:paralleledgeiso}. The statement about exceptional divisors of $\bar{\cX}_{\iota_{\pi}(t)} \rightarrow \bar{\cX}^{W'}_{\iota_{\pi}(t)}$ follows from Lemma \ref{lem:tor-bl-poly}.
\end{proof}



\subsubsection{Parallel configurations.}
Now, let us introduce the notion of parallel configurations. 
\begin{defn}
Given a convex polygon $0\in F \subset B_{(Y,D)}$ and an edge $L$ of $F$, the \emph{forward ray} $R_L$ of $L$ is the ray we get by extending $L$ in the forward direction (given by the canonical orientation of $L$).
 We say $F$ has a \emph{correct corner} in $\sigma_{i,i+1}$ if there is a unique vertex $v\in F$ such that $v$ is in the interior of $\sigma_{i,i+1}$, $F$ has no vertex on $\rho_i$ or $\rho_{i+1}$, the forward ray of one edge at $v$ goes to infinity in $\sigma_{i+1,i+2}$ parallel to $\rho_{i+2}$, and the forward ray of the other edge at $v$ goes to infinity in $\sigma_{i,i+1}$ parallel to $\rho_{i+1}$. 
\end{defn}

\begin{defn}
Given a parallel polygon $0\in P(W) \subset B_{(Y,D)}$, we say $P(W)$ is a \emph{parallel configuration} for $(Y,D)$ if $P(W)$ has a correct corner in $\sigma_{i,i+1}$ for each $1\leq i\leq n$. 
\end{defn}

\begin{prop}\label{existenceofparallelconfig} 
Let $(Y,D)$ be positive, $n\ge 2$, $D_i^2\ge 0$ for some $i$. Then, $(Y,D)$ has a parallel configuration. 
\end{prop}



\begin{proof}
Without loss of generality, suppose $D_1^2\ge 0$. Then, using the development map, we see that any immersed line $L$ goes to infinity parallel to $\rho_2$ with positive lattice distance $a_2$ from $\rho_2$ is actually an embedded line on $B$ contained in the union $\sigma_{1,2} \cup \sigma_{n,1}$. We set $a_2 =1$. Suppose we have defined $a_i$, then we set $a_{i+1} =1$ if $-D_i^{2} \leq 0$ and $a_{i+1} = (-D_i^{2})a_i+1$ if $-D_i^{2} >0$. All our indices are modulo $n$. Then, $W=\sum_i a_iD_i$ is a $D$-ample divisor and using the standard chart for $\sigma_{i-1,i} \cup \sigma_{i,i+1}$, it is straightforward to see that $P(W)$ has a correct corner in $\sigma_{i-1,i}$ for each $i$. Thus, $P(W)$ is a parallel configuration for $(Y,D)$.
\end{proof}

\begin{prop}\label{existenceofparallelconfigtwo}
Let $(Y,D)$ be positive with $n= 1$. Then there exists a parallel configuration for $(Y,D)$ if and only if $D^2\ge 2$.
\end{prop}

\begin{proof}
A polygon with a parallel configuration for $(Y,D)$ is a $1$-gon so it must be the half space $Z(L)$ for some $L$. Let $L$ be an immersed line with escape ray $\rho_1$. Then by considering the developing map, $Z(L)$ has the correct corner if and only if $D_1^2\ge 2$.
\end{proof}

\begin{lem}\label{lem:int_numbers}
If $(Y,D)$ is positive and $D_i^2\ne -1$ for all $i$, then $D_j^2\ge 0$ for some $j$.
\end{lem}

\begin{proof}
Suppose that $(Y,D)$ is positive and $D_i^2\le -2$ for all $i$. By the positivity, there exists positive integers $b_i$ such that $\left(\sum  b_i D_i\right)\cdot D_j>0$ for all $j$. Then for every $j$, we have $a_{j+1}+a_{j-1}>-a_j D_j^2$. Summing up all such inequalities, we get
\[ 2\sum a_j>\sum a_j |D_j|^2.\]
However, since $|D_j|^2\ge 2$ for all $j$, no such collection of $a_i$'s can exist.  
\end{proof}

\begin{lem}\label{blowdown}
Let $(Y,D)$ be a positive pair. Then, there exists a toric blow-down $(Y,D)\rightarrow (Y',D')$ such that either $(Y',D')$ has a parallel configuration or $(Y',D')$ is a del Pezzo surface of degree $1$ together with an irreducible anti-canonical cycle.
\end{lem}

\begin{proof}
Suppose $n\ge 2$. If there exists $i$ such that $D_i^2\ge 0$, then by Proposition \ref{existenceofparallelconfig}, there exists a parallel configuration. Otherwise, $D_j^2<0$ for all $j$ and by Lemma \ref{lem:int_numbers}, there exists $i$ such that $D_{i}^{2}=-1$. We then blowdown $D_i$ to obtain a new boundary, $D'$. If $D'$ has  only one component, then $D'^2>0$ by positivity. By Proposition \ref{existenceofparallelconfigtwo}, either $(Y',D')$ has a parallel configuration or is the case of the degree $1$ del Pezzo surface together with an irreducible anti-canonical cycle. If there is more than one component in $D'$, then again, there is either a component such that $D'^{2}_{i} \ge 0$, in which case there is a parallel configuration, or there is a component we can blow down. We continue to blow down until the pair we get either has a parallel configuration or is the degree $1$ del Pezzo surface together with an irreducible anti-canonical cycle.

\end{proof}

While many pairs in consideration do have a parallel configuration, some well-known pairs do not. For example, a cubic surface obtained by blowing up toric $\mathbb{P}^2$ twice on each boundary gives a pair $(Y,D)$ where $D$ has $3$ components, each with self-intersection $-1$. There is no parallel configuration for this pair. To get a sense of what these polygons then look like, consider the divisor $W=D_1+D_2+D_3$. Then, $P(W)$ is a convex polygon where the only vertices are given by the primitive lattice point on the $\rho_i$ and the edges connect these vertices. Another example is when $Y$ is a degree $1$ del Pezzo and $D$ is an irreducible nodal curve with self-intersection $1$ as seen in Figure \ref{fig:bounded-half-space}.


\subsubsection{Self-intersection numbers of boundary divisors of the canonical compactification.} \label{ssec:bou-int-nu}

Fix a $D$-ample Weil divisor $W = \sum_i a_iD_i$. And in this subsection, we use $P(W)$ to get the canonical compactification. Now, we want to show the canonical compactification of a smooth fiber in $\cX \rightarrow \mathrm{Spec}(\mathbb{C}[\mathrm{NE}(Y)])$  looks like the original pair $(Y,D)$ by showing that the self-intersection numbers of the boundary divisors of the compactification match with those in $D$. In Section \ref{sec:per-int}, we will show that a compactified smooth fiber is deformation equivalent to the original pair $(Y,D)$. 

Here, we state a theorem that we will use frequently to compute the self-intersection numbers of boundary divisors in the compactification. 
\begin{thm}\label{thm:cst-eu}
Suppose $f:X \rightarrow Y$ is a projective morphism and $\mathcal{F}$ is a coherent sheaf on $X$, flat over $Y$. Suppose $Y$ is locally Noetherian, then the Euler characteristic of $\mathcal{F}$ is constant in the fibers of $f$. 
\end{thm}

Let $T^W \subset T^D$ be the 1-parameter subgroup given by the map $\chi(T^{D}) \rightarrow \mathbb{Z}$, $e_{D_i} \mapsto a_i$. Let $F$ be a polygon satisfying the conditions as in Proposition \ref{compa-fiber}. Denote by $\bar{\cX}(F)$ the compactification of the mirror family $\cX \rightarrow S$ using $F$. Recall that $S=\mathrm{Spec}(\mathbb{C}[\mathrm{NE}(Y)])$. For each edge $E$, let $\cD_E \subset \bar{\cX}(F)$ be the Weil divisor corresponding to $E$.   Denote by $X_0$ the central fiber of $\bar{\cX}(F)$ over the unique toric fixed point $0 \in S$.

\begin{lem}\label{lem:rel-car}
Given a point $s \in T_Y$, the limit $s_0$ of $s$ under the action of the 1-parameter subgroup exists in $S$ and $\bar{\cX}(F)_{s_0}$ is isomorphic to $X_0$. Denote by $\delta_s: \mathbb{A}^{1} \rightarrow S$ the morphism given by $T^W \cdot s$. Suppose that $\bar{\cX}(F)_s$ is smooth and $\cD_{E,0}$ is a Cartier divisor on $X_0$. Then, $\cD_E \times_S \mathbb{A}^{1}$ is a relatively Cartier divisor on $\bar{\cX}(F) \times_S \mathbb{A}^{1}$ for each edge $E$ of $F$. Moreover, $\chi(\bar{\cX}(F)_s, \cO(\cD_{E,s})) = \chi(X_0, \cO(\cD_{E,0}))$.
\end{lem}

\begin{proof}
We use an argument similar to the one we used in the proof of Lemma \ref{lem:cartier}. Given any $C\in \mathrm{NE}(Y)$, the weight of $z^{[C]}$ under the action of $T^W$ is non-negative. Hence, the limit $s_0$ exists. We claim that $\cX_{s_0}$ is isomorphic to $\mathbb{V}_n$ and therefore $\bar{\cX}(F)_{s_0}$ is isomorphic to $X_0$. This follows from that  the intersections of all $\mathbb{A}^1$-classes with $W$ is positive and that $W \cdot D_i >0$ for all $i$. In particular, $\cD_{E,s_0}$ is a Cartier divisor on $\bar{\cX}(F)_{s_0}$ since $\cD_{E,0}$ is one on $X_0$.

Since $\bar{\cX}(F)_s$ is smooth and the action of $T^{D}$ is equivariant on $\bar{\cX}(F) \rightarrow S$,  $\bar{\cX}(F)_{\delta_s(x)}$ is smooth for any $x \in \mathbb{A}^1 \setminus\left\{0\right\}$. Thus,  $\cD_{E, \delta_s(x)} $ is a Cartier divisor on $\bar{\cX}(F)_{\delta_s(x)}$ for $x \in \mathbb{A}^1 \setminus \left\{0\right\}$. By Proposition \ref{prop:fin-gen}, $f: \bar{\cX}(F) \times_{\mathbb{A}^1}  S \rightarrow \mathbb{A}^1$ is of finite type. Since  $\mathbb{A}^1$ is Noetherian, by \cite[\href{https://stacks.math.columbia.edu/tag/01TX}{Tag 01TX}]{Sta}, $f$ is of finite presentation. Since  $\bar{\cX}(F) \times_S  {\mathbb{A}^1} \rightarrow \mathbb{A}^1$ is also flat and $\cD_{E,\delta_s(x)}$ is an effective Cartier divisor on $\bar{\cX}(F)_{\delta_s(x)}$ for each $x \in \mathbb{A}^1$, by \cite[\href{https://stacks.math.columbia.edu/tag/062Y}{Tag 062Y}]{Sta}, $\cD_E \times_S \mathbb{A}^1$ is a relative effective Cartier divisor on $\bar{\cX}(F) \times_S {\mathbb{A}^1}$. By Proposition \ref{prop:fin-gen}, $f$ is projective. Therefore, since we have the isomorphism $\bar{\cX}(F)_{s_0} \simeq X_0$, $\chi(\bar{\cX}(F)_s, \cO(\cD_{E,s})) = \chi(X_0, \cO(\cD_{E,0}))$ by Theorem \ref{thm:cst-eu}.
\end{proof}

\begin{lem}\label{boundaryselfint} 
If $\bar{\cX}_t$ is the fiber of $\bar{\cX}$ over $t \in T_Y$ such that $\bar{\cX}_t$ is smooth, then $\cD_{i,t}^2=D_i^2$.
\end{lem}

\begin{proof} 
Let us first consider the case where $P(W)$ is a parallel configuration. It is straightforward to see that $\cD_{i,0}$ is a Cartier divisor on $X_0$, the compactified fiber over $0\in \Spec(\bC[\NE(Y)])$, for each $i$. Then, $\chi(\bar{\cX}_t, \cO(\cD_{i,t})) = \chi(X_0, \cO(\cD_{i,0}))$ by Lemma \ref{lem:rel-car}.  We break up the argument into two sub-cases: one for $n\geq 2$ and the other where $n=1$. Recall that $n$ is the number of irreducible components in $D$.

\emph{Case 1: $P(W)$ is a parallel configuration and $n\geq 2$.} Each irreducible component in $X_0$ is a toric ruled surface. Fix the edge $F_i$ which is parallel to $\rho_i$ and lies in two  maximal cells, $\sigma_{i-2,i-1}$ and $\sigma_{i-1,i}$. For $j=i-1,i$, denote $P(W) \cap\sigma_{j-1,j}$ by $Q_{j-1,j}$ and the corresponding toric variety by $X(Q_{j-1,j})$. Denote $R_i=F_i\cap Q_{i-1,i}$ and $R'_i=F_i\cap Q_{i-2,i-1}$.

Let $Y_1 = X(Q_{i-1,i}) $ and $Y_2 = X(Q_{i-2,i-1})$ and $Y=Y_1 \cup Y_2$. The 1-cell $R_i$ corresponds to a section $C$ of the ruled toric surface $Y_1$ and has self-intersection number $D_i^2$. By Riemann-Roch,
\[ \chi \left( Y_1,\cO(\cD_{i,0})\mid_{Y_1}\right)=1+\frac{1}{2}C\cdot(C-K_{Y_1})=1+\frac{1}{2}(2C^2+2)=2+D_i^2.\]
In $Y_2$, $R_i'$ corresponds to a fiber of the ruled toric surface $Y_2$ and thus has self-intersection number $0$. Again, by Riemann-Roch,
\[
\chi \left( Y_2,\cO(\cD_{i,0})\mid_{Y_2}\right) = 2.
\]
Denote the 1-cell $Q_{i-2,i-1} \cap Q_{i-1,i}$ by $\varrho_{i-1}$. Then, $X(\varrho_{i-1}) = Y_1 \cap Y_2$ is a rational curve and we have
\[
\cO(\cD_{i,0}) \mid_{X(\varrho_{i-1})} \simeq \cO_{\mathbb{P}^{1}}(1).
\]
Thus, $\chi(X(\varrho_{i-1}), \cO(\cD_{i,0})\mid_{X(\varrho_{i-1})})= 2$ by Riemann-Roch.

Recall that if $X$ is the union of two closed subvarieties $X_1, X_2$ and $\mathcal{F}$ is a locally free sheaf on $X$, then we have the following short exact sequence
\[
0 \rightarrow \mathcal{F} \otimes \cO_{X}  \rightarrow (\mathcal{F}\otimes \cO_{X_1}) \oplus (\mathcal{F} \otimes \cO_{X_2}) \rightarrow \mathcal{F} \otimes \cO_{X_1 \cap X_2} \rightarrow 0.
\]
Taking the Euler characteristic, we get 
\begin{equation}
    \chi(X, \mathcal{F}) = \chi(X_1, \mathcal{F}\mid_{X_1}) + \chi (X_2, \mathcal{F}\mid_{X_2}) - \chi(X_1 \cap X_2, \mathcal{F}\mid_{X_1 \cap X_2}).  
\end{equation} \label{eq:ses}
By Equation \ref{eq:ses}, we get 
\begin{align*}
   & \chi(Y, \cO(\cD_{i,0})\mid_{Y_1 \cup Y_2}) \\ 
  = & \chi(Y_1,\cO(\cD_{i,0})\mid_{Y_1}) + \chi(Y_2,\cO(\cD_{i,0})\mid_{Y_2}) - \chi(X(\varrho_{i}), \cO(\cD_{i,0})\mid_{X(\varrho_{i-1})}) \\
  = & D_i^{2}+2. 
\end{align*}
Let $Z$ be the union of irreducible components in $X_0$ that are not $Y_1$ or $Y_2$. Then, $\cO(\cD_{i,0})\mid_{Z}$ is the trivial line bundle. In particular, we have
\[
\chi(Z, \cO(\cD_{i,0}) \mid_{Z}) = 1 \,\, \mathrm{and} \,\, \chi(Y \cap Z, \cO(\cD_{i,0}) \mid_{Y \cap Z})=1.
\]
Applying Equation \ref{eq:ses} again, we obtain that 
\begin{align*}
& \chi(X_0, \cO(\cD_{i,0})) \\
= & \chi(Y, \cO(\cD_{i,0}) \mid_{Y}) + \chi(Z, \cO(\cD_{i,0}) \mid_{Z}) - \chi(Y \cap Z, \cO(\cD_{i,0}) \mid_{Y \cap Z})\\
= & D_i^2+2.
\end{align*}

\begin{figure}
\begin{minipage}{\linewidth} \begin{center} 
\vspace{1 cm}
\begin{tikzpicture}[scale=0.7] 
  \draw [-stealth](0,0) -- (4,0) node[anchor=west] {$\rho_1$};
  \draw [-stealth](0,0)-- (0,4) node[anchor=south,yshift=.1cm] {$\rho_2$};
  \draw [-stealth] (0,0) --(-4,4) node[anchor=south] {$\rho_1$};
    \draw [red](2,0) --(0,2);
     \draw [red](0,2) --(-2,4);
       \draw [red](-2,2) --(-2,4);
      \filldraw [black] (2,0) circle (2pt);
  \filldraw [black] (-2,2) circle (2pt);
  \end{tikzpicture}
 \caption{$P(W)$ for $(Y,D)$ where $n=1$ and $D^2 =3$.}\label{fig:para-poly}
\end{center}\end{minipage}
\end{figure}  
Now, fix a generic $t\in T_Y$ such that $\bar{\cX}_t$ is smooth. By Riemann-Roch, on $\bar{\cX}_t$, $\chi \left(\bar{\cX}_t, \cO(\cD_{i,t})\right)=\cD_{i,t}^2+2$. Then, since 
\[
\chi \left(\bar{\cX}_t, \cO(\cD_{i,t})\right) = \cD_{i,t}^2+2 = \chi(X_0, \cO(\cD_{i,0})) =  D_i^2+2
\]
we conclude that $\cD_{i,t}^2 = D_i^2$.

\emph{Case 2: $P(W)$ is a parallel configuration and $n =1$.} By Proposition \ref{existenceofparallelconfigtwo}, we have $D^2 \geq 2$.  We  pass to a toric blowup at the node of $D$ and draw $P(W)$ in $B_{(Y',D')}$ where $(Y',D')$ is the toric blow-up at the node of $D$. Let $\bar{\cX}'(P(W))$ be the compactification of the mirror family $\cX'$ associated to $(Y',D')$ using $P(W)$. By Lemma \ref{lem:tor-bl-compa}, we can embed $\bar{\cX}$ into $\bar{\cX}'(P(W))$. The fiber over $0\in \mathrm{Spec}(\bC[\mathrm{NE}(Y')])$ consists of two irreducible components, $X_1$ and $X_2$, whose moment map is contained in $\sigma_{2,1}$ and $\sigma_{1,2}$ respectively. Denote by $C_i$, the irreducible component of $X_1 \cap X_2$, whose moment map is contained in $\rho_i$. Denote the central fiber by $X_0 = X_1 \cup X_2$. Let $\cD_0 = \cD \mid_{X_0}$ and $H_i = \cD_0 \cap X_i$. 

If $D^2 \geq 3$, $X_1$ is isomorphic to a smooth ruled toric surface and $X_2$ isomorphic to $\mathbb{P}^2$. The intersection of the only edge $F_1$ of $P(W)$ with $\sigma_{2,1}$ consists of two segments $R_1$ and $R'_1$ with $R_1$ parallel to $\rho_1$.  $X(R_1)$ is a toric section of $X_1$ with self-intersection $D^2 -3$ and  $X(R'_1)$ is a fiber of $X_1$. Thus,
$H^{2}_{1} = (X(R_1)+X(R'_1))^2 = D^2-1$. If $D^2 = 2$, $X_1$ and $X_2$ are both isomorphic to $\mathbb{P}^2$, $H_1$ is an irreducible toric boundary divisor on $X_1$ and $H^2_{1} = D^2 -1$ still holds.     
By Riemann-Roch, 
\begin{align*}
    \chi (X_1, \cO(\cD_0)\mid_{X_1}) & = 1+\frac{1}{2} H_1 (H_1 - K_{X_1}) 
    = 1+\frac{1}{2} (2H^{2}_{1} +2 ) 
     = D^2+1
\end{align*}
 and 
 \[
  \chi (X_2, \cO(\cD_0) \mid_{X_2}) = 1+\frac{1}{2}H_2 (H_2 - K_{X_2}) = 3.
 \]

Consider the irreducible components  $C_i$ of $X_1 \cap X_2$. On $C_i$, $ \cO(\cD_0)\mid_{C_i} \simeq \cO_{\mathbb{P}^1}(1)$. By Riemann-Roch, 
$\chi (C_i, \mathcal{O}(\cD_0)\mid_{C_i}) =2.$
On $C_1 \cap C_2$, $\cO(\cD_0)\mid_{C_1 \cap C_2} $ is the trivial line bundle over a point.
By Equation \ref{eq:ses} and Riemann-Roch,
\begin{align*}
   & \chi(X_1 \cap X_2, \mathcal{O}(\cD_0)) \\
     = & \chi (C_1, \mathcal{O}(\cD_0)\mid_{C_1}) + \chi (C_2, \mathcal{O}(\cD_0)\mid_{C_2}) - \chi(C_1 \cap C_2,  \mathcal{O}(\cD_0)\mid_{C_1 \cap C_2})\\  =& 2+2-1=3 . 
\end{align*}
Applying Equation \ref{eq:ses} again, 
\begin{align*}
   & \chi(X_0, \mathcal{O}(\cD_0)) \\
     = & \chi (X_1, \mathcal{O}(\cD_0)\mid_{X_1}) + \chi (X_2, \mathcal{O}(\cD_0)\mid_{X_2}) - \chi(X_1 \cap X_2,  \mathcal{O}(\cD_0)\mid_{X_1 \cap X_2})\\  =& D^2 +1 . 
\end{align*}
Now, fix a generic $t\in T_Y$ such that $\bar{\cX}_t$ is smooth. By Riemann-Roch, on $\bar{\cX}_t$, $\chi \left(\bar{\cX}_t, \cO(\cD_{i,t})\right)=\cD_{i,t}^2+1$. Then, since 
\[
\chi \left(\bar{\cX}_t, \cO(\cD_{i,t})\right) = \cD_{i,t}^2+1 = \chi(X_0, \cO(\cD_{i,0})) =  D_i^2+1
\]
we conclude that $\cD_{i,t}^2 = D_i^2$.

\begin{figure}
\begin{minipage}{\linewidth} \begin{center} 
\vspace{1 cm}
\begin{tikzpicture}[scale=0.7] 
  \draw [-stealth](0,0) -- (4,0) node[anchor=west] {$\rho_1$};
  \draw [-stealth](0,0)-- (0,4) node[anchor=south,yshift=.1cm] {$\rho_2$};
  \draw [-stealth] (0,0) --(-4,4) node[anchor=south] {$\rho_1$};
    \draw [red](2,0) --(-2,2);
      \filldraw [black] (2,0) circle (2pt);
  \filldraw [black] (-2,2) circle (2pt);
  \end{tikzpicture}
 \caption{$P(W)$ for the degree $1$ del Pezzo.}\label{fig:dP1-poly}
\end{center}\end{minipage}
\end{figure}

Now, suppose $\pi: (Y,D) \rightarrow (Y',D')$ is a toric blowdown and that we know that the statement of the lemma holds for smooth fibers in $\bar{\cX}'\rightarrow T_{Y'}$. Let $\iota_{\pi}: \Spec(\bC[\NE(Y')]) \hookrightarrow \Spec(\bC[\NE(Y)])$ be the locally closed embedding as in Lemma \ref{lem:tor-bl-compa}. Then, given $t\in T_{Y'}$ such that $\bar{\cX}'_t$ is smooth, by Lemma \ref{lem:compa-para-tor-bl}, the statement of the lemma also holds for $\bar{\cX}_{\iota_{\pi}(t)}$ since it holds for $\bar{\cX}'_t$.  Using the equivariant action of $T^{D}$ on $\bar{\cX} \mid_{T_Y} \rightarrow T_Y$, we know the statement of Lemma has to hold for any smooth fiber in $\bar{\cX} \mid_{T_Y} \rightarrow T_Y$. In particular, if  $(Y',D')$ has a parallel configuration, then by our previous computation, the statement holds for smooth fibers $\bar{\cX}'\rightarrow T_{Y'}$ and therefore also holds for smooth fibers in $\bar{\cX} \mid_{T_Y} \rightarrow T_Y$. 

If $(Y,D)$ does not blow-down to a pair that has a parallel configuration, it must blow-down to the case of the degree $1$ del Pezzo surface together with an irreducible anti-canonical divisor. By our discussion in the previous paragraph, it suffices to prove that the statement of the lemma holds for smooth fibers in the canonically compactified mirror family associated with the degree $1$ del Pezzo surface. 

Let $(Y,D)$ be a degree $1$ del Pezzo surface with an irreducible nodal anticanonical divisor. Consider $W=D$ and the associated polytope, $P(W)$. Again, we pass to a toric blowup $(Y',D') \rightarrow (Y,D)$ at the node of $D$ as we did in the parallel case where $n=1$ and follow notation we used there. The fiber $X_0$ over $0\in \mathrm{Spec}(\bC[\mathrm{NE}(Y')])$ consists of two irreducible components $X_1$ and $X_2$, each of which is a copy of $\mathbb{P}(1,1,2)$.  

\begin{figure}
\begin{minipage}{\linewidth} \begin{center} 
\vspace{1 cm}
\begin{tikzpicture} [scale = 1.5]
  \draw [-stealth](0,0) -- (1,0);
  \draw [-stealth](0,0)-- (0,1) ;
  \draw [-stealth] (0,0) --(-1,-2) ;
  \node [label= left: $\tau_1$] at (0.7, -0.3) {};
   \node [label= right: $\tau_2$] at (-0.7, 0.2) {};
    \end{tikzpicture}
 \caption{The toric fan for $\mathbb{P}(1,1,2)$.}
\end{center}\end{minipage}
\end{figure}

 Note that due to the quotient singularity in $\mathbb{P}(1,1,2)$, $\cD_0$ is not Cartier but $2\cD_0$ is. By Lemma \ref{lem:rel-car}, it suffices to compute $\chi(X_0, \mathcal{O}(2\cD_0))$.

First, notice that $\mathcal{O}(2\cD_0)\mid_{X_i} = \mathcal{O}(2\cD_0 \cap X_i)$ is generated by its global sections since its support function is convex. Thus, higher cohomology groups of $\mathcal{O}(2\cD_0)\mid_{X_i}$ vanish and we have 
\[
\chi(X_i, \mathcal{O}(2\cD_0)\mid_{X_i}) = H^{0}(X_i, \mathcal{O}(2\cD_0)\mid_{X_i}) =4
\]
which is equal to the number of lattice points in the convex polytope associated with $\mathcal{O}(2\cD_0) \mid_{X_i}$ as shown in Figure \ref{fig:pol-lin}.

\begin{figure}
\begin{minipage}{\linewidth} \begin{center} 
\vspace{1 cm}
\begin{tikzpicture} 
  \draw (0,0) -- (4,0);
  \draw (0,0)-- (0,2) ;
  \draw  (4,0) --(0,2) ;
\filldraw [black] (0,0) node [anchor=north] {(0,0)} circle (2pt); 
\filldraw [black] (2,0) node [anchor=north] {(1,0)} circle (2pt); 
\filldraw [black] (4,0) node [anchor=north] {(2,0)} circle (2pt); 
\filldraw [black] (0,2) node [anchor=east] {(0,1)} circle (2pt); 

    \end{tikzpicture}
 \caption{The convex polytope of the line bundle $\mathcal{O}(2\cD_0) \mid_{X_i}$.} \label{fig:pol-lin}
\end{center}\end{minipage}
\end{figure}

Next, consider the two irreducible curves $C_1, C_2$ of $X_1 \cap X_2$.  We claim that 
\[
 \mathcal{O}(2\cD_0) \mid_{C_2} \simeq \mathcal{O}_{\mathbb{P}^{1}}(1)
\]
and 
\[
 \mathcal{O}(2\cD_0) \mid_{C_1} \simeq \mathcal{O}_{\mathbb{P}^{1}}(2).
\]
Indeed, denote the intersection point of $D_2$ with $C_1$ and $C_2$ by $r$ and ${q}$ respectively. Denote by $U_{\tau_i} (i=1,2)$ the affine chart corresponding to the cone $\tau_i$ in the fan for $X_2 \simeq \mathbb{P}(1,1,2)$. Then, on the affine chart $U_{\tau_1}$, $\mathcal{O}(-2\cD_0) \mid_{U_{\tau_1}}$ is generated by the toric monomial $x=z^{(0,-1)}$. Since $C_2 \cap U_{\tau_1} = \mathrm{Spec}(\mathbb{C}[x])$, we conclude that $\mathcal{O}(-2\cD_{0})\otimes O_{C_2}$ is the ideal sheaf for $q$. Let $y=z^{(-1,0)}$. On the affine chart $U_{\tau_2}$,  $\mathcal{O}(-2\cD_0) \mid_{U_{\tau_2}}$ is generated by the toric monomial $y^2=z^{(-2,0)}$. Since $C_1 \cap U_{\tau_2} =  \mathrm{Spec}(\mathbb{C}[y])$, we conclude that   $\mathcal{O}(-2\cD_{0})\otimes O_{C_1}$ is the ideal sheaf for $2r$.

Now, let us compute $\chi (X_1 \cap X_2, \mathcal{O}(2 \cD_0)\mid_{X_1 \cap X_2})$. We have $X_1 \cap X_2 =C_1 \cup C_2$. Applying Equation \ref{eq:ses} gives
\begin{align*}
    & \chi (X_1 \cap X_2, \mathcal{O}(2 \cD_0)\mid_{X_1 \cap X_2}) \\
 = & \chi(C_1, \mathcal{O}(2 \cD_0)\mid_{C_1}) +  \chi(C_1, \mathcal{O}(2 \cD_0)\mid_{C_2}) - \chi(C_1 \cap C_2, \mathcal{O}(2 \cD_0)\mid_{C_1 \cap C_2}) \\ = & 2+3-1=4. 
\end{align*}
Apply Equation \ref{eq:ses} again, we get 
\begin{align*}
    & \chi(X_0, \mathcal{O}(2\cD_0)) \\
     = & \chi (X_1, \mathcal{O}(2\cD_0)\mid_{X_1}) + \chi (X_2, \mathcal{O}(2\cD_0)\mid_{X_2}) - \chi(X_1 \cap X_2,  \mathcal{O}(2\cD_0)\mid_{X_1 \cap X_2})\\  =& 4+4-4=4. 
\end{align*}
Thus, by Riemann-Roch, we have
\begin{align*}
    \chi (\bar{\cX}_t, O(2\cD_t)) & = \frac{1}{2}2\cD_t \cdot(2\cD_t +\cD_t) +1 = 3\cD_t^2 +1  =\chi(X_0, \mathcal{O}(2\cD_{0})) = 4
\end{align*}
and we obtain that $\cD_t^2 =1$ as desired. 
\end{proof}

\section{The Boundary and the Marked
Period}\label{sec:marking-boundary}

The global theta functions provide a way to consistently give a marking on the boundary divisor of the compactified family. Later, in Section \ref{sec:per-int}, we will use this marking on the boundary to compute the marked periods of the fibers in the compactified mirror family. 

\begin{cons}[Marking the boundary via the canonical coordinates]
\label{cons:marking} Given a non-self-intersecting line segment $\bar{L}$ in $B\setminus\left\{0\right\}$, we give an order on the set of lattice points contained in  $\bar{L}$ by letting $a \succ b$ if $a$ succeeds $b$, following the canonical orientation of the immersed line containing $\bar{L}$.

Fix a parallel polygon $F$ associated with a $D$-ample Weil divisor $W$. Denote by $F_i$ the edge parallel to the ray $\rho_i$. Rescaling the polygon $F$ by a positive integral multiple if necessary, we can assume that for each edge $F_i$, there exists a maximal cone $\sigma$ in $\Sigma$ that contains two consecutive lattice points $a_i \succ b_i$ on $F_i$. If $F$ is a 1-gon, we further require the two consecutive lattice points to be contained in the interior of $\sigma$. Let $F_{i,\sigma} = \sigma \cap F_i$. We can extend $F_{i,\sigma}$ in the forward direction (given by the canonical orientation of $F_i$) to a ray $R$ that escapes in the cone $\sigma_{i-1,i}$, which is parallel to $\rho_i$. Given a ray $\rho_j$ in $\Sigma$ that intersects $R$, let $\delta_j = 
\left\langle n_{\rho_j} \,,\,\overrightarrow{b_ia_i} \right\rangle $ where $n_{\rho_j}\in \check{\Lambda}_{\rho_j}$ is the unique primitive element annihilating $\rho_j$ that is positive on the parallel transport of $\overrightarrow{b_ia_i}$. Let 
\begin{align}
   f_{\cD_i} = z^{\sum_{\rho_j} \delta_j \cdot D_j} \cdot \frac{\theta_{a_i}}{\theta_{b_i}} \label{eq:diff-mar}
\end{align}
 where the sum is over all rays, $\rho_j,$ that intersect $R$. By Proposition \ref{thetamarking}, $f_{\cD_i}$ restricts to a coordinate on $\cD^{\circ}_i \mid_{T_Y}$ where $\cD^{\circ}_i \mid_{T_Y}$ is the complement of the nodes.  Then, we can define a section $p_i$ of $\cD^{\circ}_i \mid_{T_Y} \rightarrow T_Y$ by letting $ f_{\cD_i} = -1 $. These choices of sections ${p_i}$'s give us a marking of $\cD\mid_{T_Y}$. 

By the definition of the $\Sigma$-piecewise linear function $\varphi$ and Proposition \ref{prop:paralleledgeiso}, $f_{\cD_i}$ and the marking $p_i$ are independent of the particular choice of the parallel polygon $F$ and the consecutive lattice points $a \succ b$. By Lemma \ref{extendsection}, $p_i$ extends to a section on $\cD_i$. We call the coordinate $f_{\cD_i}$ the \emph{canonical coordinate} on the boundary $\cD_i$. 
\end{cons}
\begin{lem}\label{lem:rel-tor-mar}
The relative torus $T^{D}$ acts on $(\bar{\cX}\mid_{T_Y},\cD\mid_{T_Y})$ by changing the canonical marking $p_i$ $(1\leq i \leq n)$ as defined in Construction \ref{cons:marking}.
\end{lem}
\begin{proof}
The forward ray $R_{F_i}$ of $F_i$ escapes in the cone $\sigma_{i-1,i}$, parallel to $\rho_i$ on the right. Take two consecutive lattice points $p \succ q$ on $R_{F_i} \cap \sigma_{i-1,i}$. Then, by Proposition \ref{prop:paralleledgeiso}, restricted to $\cD^{\circ}_i \mid_{T_Y}$, $f_{\cD_i} = \frac{\theta_p}{\theta_q}$. It follows from the definition of the weight function and Theorem \ref{thm:rel-equi} that restricted to $\cD^{\circ}_i \mid_{T_Y}$, $f_{\cD_i}$ is an eigenfunction of $T^{D}$ with weight $e_{D_i}$. 
\end{proof}

\begin{lem}\label{extendsection}
Recall that $P=\NE(Y)$ and $S=\Spec(\mathbb{C}[\NE(Y)])$. Let $(\bar{\cX},\cD)\rightarrow S$ be the compactification of the mirror family using a parallel polygon $F$. Consider the degeneration $f_i: \cD_i \rightarrow S$. Let $p \succ q$ be two consecutive integer points in the same 1-cell $\varrho$ in the polyhedral decomposition of the edge $F_i$ given by $\Sigma \cap F_i$. Consider the section $s:T_Y\rightarrow \cD_i \mid_{T_Y}$ of $\cD_i \mid_{T_Y} \rightarrow {T_Y}$ given by $\theta_p/\theta_q=- 1$. Then $s$ extends to a section over the full base $S$. Furthermore, at the unique toric fixed point $0\in S$, it meets an interior point of the boundary component corresponding to $\varrho$.
\end{lem}

\begin{figure}
\begin{minipage}{\linewidth} \begin{center} 
\vspace{1 cm}
\begin{tikzpicture} 
\draw (-4,-2) node[anchor=north] {}
  -- (2,4) node[anchor=north] {}
  -- (2,-2) node[anchor=south] {}
  -- cycle;
  \draw (0,0) -- (2,0) node[anchor=west] {$\rho_1$};
  \draw (0,0)-- (0,2) node[anchor=south,yshift=.1cm] {$\rho_2$};
  \draw (0,0) --(-2,-2) node[anchor=north] {$\rho_3$};
  \filldraw [black] (2,1) circle (2pt);
  \filldraw [black] (2,2) circle (2pt);
    \filldraw [black] (0,-2) circle (2pt);
  \filldraw [black] (1,-2) circle (2pt);
    \filldraw [black] (-3,-1) circle (2pt);
  \filldraw [black] (-2,0) circle (2pt);
\end{tikzpicture}
 \caption{An example of the marking on the boundary in the case  where $Y$ is $\mathbb{P}^2$ and $D$ its toric boundary.}\label{fig:marking}
\end{center}\end{minipage}
\end{figure}

\begin{proof}
The lemma follows from Claim \ref{mult-str} and the fact that restricted to $\varrho$, $\varphi$ does not bend and is affine.
\end{proof}

\begin{lem}\label{lem:toricperiod} 
For $i\in \{1,...,n\}$ with $n\geq1,$  let $(S_i,p_i,G_i,G'_i)$ be an ordered collection of smooth complete toric surfaces $S_i$ with a choice of two adjacent toric boundary components $G_i$ and $G'_i$ that intersect at the $0$-stratum $p_i$. The indices are modulo $n$. Gluing $G_i\subset S_i$ to $G'_{i+1}\subset S_{i+1}$ by identifying the structure tori so that $p_i$ is glued to $p_{i+1}$, we obtain a  surface $S$ with normal crossing boundary cycle $D$  given by
\[ D:= \bigcup_i \partial (S_i)\setminus (G_i\cup G'_i).\]
Each irreducible component $Z\subset D$ has a structure torus $\mathbb{G}_m\subset Z$ so there is a distinguished point $-1_Z\in \mathbb{G}\subset Z$. Then, given  $L\in \Pic(S)$, 
\[ L|_D=\cO\Bigl(\sum_{Z\subset D}(L\cdot Z)(-1_Z)\Bigr).\]
\end{lem}


\begin{proof}
Let $\pi: \amalg_i S_i \rightarrow S$ be the normalization map. It  follows from Proposition $3.1$ of \cite{FS} that the pullback map  $\pi^{*}: \Pic(S)\rightarrow \oplus_i \Pic(S_i)$ is injective. Then we can write $L=\cO_S(\sum_i C_i)$ such that $\tilde{C}_i = \pi^{*}{C_i}$ is a Cartier divisor  on $S_i$. 

Let $L_i$ be the pullback of $L$ to $S_i$. Then, $L_i = O_{S_i}(\tilde{C}_i)$. We can write $\tilde{C}_i = A_i-B_i$ where $A_i$ and $B_i$ are ample divisors, so it will be enough to prove the lemma for the case where  $L_i$ is ample for each $i$. Then, the polarized pair $(S_i,L_i)$ is given by a lattice polytope $Q_i$. Take a linear combination $f$ of monomials given by lattice points on the boundary $\partial Q_i$ such that the restriction of $f$ to a boundary component $Z_{\rho} \subset D$ with $\rho$ a boundary of $Q_i$, has the form $(X+Y)^{L_i \cdot Z_{\rho}}$ where $X,Y$ are homogeneous coordinates on $Z_{\rho}$. Then, the  zero scheme given by $f=0$ on $Z_{\rho}$ is $(L_i \cdot Z_{\rho}) (-1_{Z_{\rho}})$ where $-1_{Z_{\rho}} \in \mathbb{G}_m \subset Z_{\rho}$ is the point on defined by $X+Y=0$. Let $C= \sum_i C_i$. Then, we conclude that
\[L|_D=\cO_D(C|_D)=\cO\Bigl(\sum_{Z\subset D} \left(L\cdot Z\right)\left(-1_Z\right)\Bigr).\]
 \end{proof}

\begin{lem}\label{lem:num-tri}
Let $((\bar{\cX} \mid_{T_Y},\cD \mid_{T_Y}),p_i)\rightarrow T_Y$  be the family where the marking $p_i$ is given as in Construction \ref{cons:marking}. Fix a component, $\cD_i\subset \cD$, and let $\cD_{i,t} = \cD_i \mid_{\bar{\cX}_t}$. When $n\ge 2$, we define:
\[ \psi_i(t):=(\mathcal{O}(\cD_i)|_{\cD_t})^{-1}\otimes \cO(p_{i-1}(t)+D_{i}^2p_i(t)+p_{i+1}(t))\]
When $n=1$, we define
\[ \psi_1(t)=\psi(t):=(\cO(\cD)|_{\cD_t})^{-1}\otimes \cO( D^2 p(t)) \]
Then, $\psi_i$ defines an element in $\Pic^0(\cD_t)$ for $t\in T_Y$. 
\end{lem}
\begin{proof}
Since $\bar{\cX}\mid_{T_Y}$ is smooth along $\cD \mid_{T_Y}$ by Proposition \ref{compa-fiber}, $\cD_i \mid_{T_Y}$ is a Cartier divisor on $\bar{\cX}\mid_{T_Y}$ by \cite[\href{https://stacks.math.columbia.edu/tag/062Y}{Tag 062Y}]{Sta}. If $\bar{\cX}_t$ is smooth, then ${\cD}^{2}_{i,t} = D^{2}_i$ by Lemma \ref{boundaryselfint}.  If $\bar{\cX}_t$ is singular, since $\bar{\cX}_t$ is still smooth along $\cD_t$, we still have ${\cD}^{2}_{i,t} = D^{2}_i$ after passing through a minimal resolution of $\bar{\cX}_t$. Therefore, $\psi_i$ defines an element in $\Pic^0(\cD_t)$ for $t\in T_Y$. 
\end{proof}

\begin{prop}\label{prop:boundaryperiod-para} 
Following the notations and definitions of Lemma 
\ref{lem:num-tri}. Suppose that $(Y,D)$ has a parallel configuration. Then $\psi_i(t)=z^{D_i}(t)$ if $n \geq 2$ and $\psi(t)=z^{D}(t)$ if $n=1$.
\end{prop}

\begin{figure}
\begin{minipage}{\linewidth} \begin{center} 
\vspace{1 cm}
\begin{tikzpicture} 
\draw (-4,-2) node[anchor=north]{}
  -- (2,4) node[anchor=north]{}
  -- (2,-2) node[anchor=south]{}
  -- cycle;
  \draw (0,0) -- (2,0) node[anchor=west] {$\rho_1$};
  \draw (0,0)-- (0,2) node[anchor=south,yshift=.1cm] {$\rho_2$};
  \draw (0,0) --(-2,-2) node[anchor=north] {$\rho_3$};
  \filldraw [black] (2,1) circle (2pt);
  \filldraw [black] (2,2) circle (2pt);
    \filldraw [black] (0,-2) circle (2pt);
  \filldraw [black] (1,-2) circle (2pt);
    \filldraw [black] (-3,-1) circle (2pt);
  \filldraw [black] (-2,0) circle (2pt);
      \filldraw [red] (0,2) circle (2pt);
  \filldraw [red] (1,3) circle (2pt);
\end{tikzpicture}
\caption{The red nodes signify the alternative marking on the boundary divisor $D_2$ of $\mathbb{P}^2$ as in the proof of Proposition \ref{prop:boundaryperiod-para}}\label{fig:altmarking}
\end{center}\end{minipage}
\end{figure}

\begin{proof}
Fix the canonical identification $\Pic^0(\cD_t) \simeq \mathbb{G}_m$. Let $F$ be a parallel configuration for $(Y,D)$ and we use $F$ to compactify our mirror family. Then, $\cD_{i,t}$ is still a Cartier divisor on $\bar{\cX}_t$ for $t$ on the boundary strata $\Spec(\mathbb{C}[\NE(Y)])\setminus T_Y$. Therefore, $\cD_i$ is an Cartier divisor on $\bar{\cX}$ by \cite[\href{https://stacks.math.columbia.edu/tag/062Y}{Tag 062Y}]{Sta}. Let $0\in \Spec (\mathbb{C} \left[ \NE(Y) \right] )$ be the unique toric fixed point and $X_0$ the fiber of $\bar{\cX}$ over this point. 

 Let us first consider the case where $n\geq 2$. We now use notation from Construction \ref{cons:marking}.  Denote $R_i=F_i\cap Q_{i-1,i}$ and $R'_i=F_i\cap Q_{i-2,i-1}$.  The marking on $\cD_{i+1} \mid_{T_Y}$ is given by $f_{\cD_{i+1}} = \frac{\theta_{a_{i+1}}}{\theta_{b_{i+1}}} = -1$ where $a_{i+1} \succ b_{i+1}$ are two consecutive lattice points on $R_{i+1}$. Let $f'_{\cD_{i+1}} = \frac{\theta_{c_{i+1}}}{\theta_{d_{i+1}}}$ where $c_{i+1} \succ d_{i+1}$ are two consecutive lattice points on $R'_{i+1}$. Consider an alternative marking $p'_{i+1}$ on $\cD_{i+1} \mid_{T_Y}$ by letting $f'_{\cD_{i+1}}$ = -1 instead. By the definition of the $\Sigma$-piecewise linear function $\varphi$ and Proposition \ref{trivialboundary}, over $T_Y$ , we have
\[
f_{\cD_{i+1}} = z^{D_{i}} f'_{\cD_{i+1}} \quad \mathrm{on}\,\,\cD_{i+1}^{\circ} \mid T_Y.
\]
Using the alternative marking on $\cD_{i+1}\mid_{T_Y}$, we define a new function
\[\tilde{\psi_i}(t)=(\cO(\cD_i)|_{\cD_t})^{-1}\otimes \cO(p_{i-1}(t)+D_{i}^2p_i(t)+p'_{i+1}(t)) \in \Pic^{0}(\cD_t).\]
Then, it follows from the definition of the alternative marking on $\cD_{i+1}\mid_{T_Y}$ and Lemma \ref{lem:pic-mul} that $\tilde{\psi_i}(t)=\psi_i(t)\otimes z^{-D_i}(t)$ on $T_Y$.

Next, we show that $\tilde{\psi_i}$ extends to a nowhere vanishing regular function over the toric variety $\Spec (\mathbb{C} \left[ \NE(Y) \right] )$.  This is equivalent to showing that over $X_0$, $\tilde{\psi}_i$ still gives a numerically trivial line bundle. By Lemma \ref{extendsection}, the sections given by the marking extend over the full base.  On $X_0$, $\cD_{i}$ breaks into two components, $X(R_i)$ and $X(R'_i)$.  Furthermore, since $X(R_i)^2=D_i^2$ and $ X(R'_i)^2 =0$, we indeed get a numerically trivial line bundle on $X_0$. Thus, $\tilde{\psi_i}$ extends to a non-vanishing function over the whole base $\Spec (\mathbb{C} \left[ \NE(Y) \right] )$, which tells us that $\tilde{\psi_i}$ is constant over $\Spec (\mathbb{C} \left[ \NE(Y) \right] )$. On $X_0$,  using Lemma \ref{lem:toricperiod}, we can compute that $\tilde{\psi_i}(0)=1$.
Thus, $\tilde{\psi_i}\equiv 1$ over $T_Y$,  which implies that $\psi_i(t)=z^{D_i}(t)$ over $T_Y$. 

Finally, let us consider the case where $n=1$. Let $R_1 \subset F_1$ be the segment of the unique edge  $F_1$ before it crosses the unique ray $\rho_1$ and $R_2  \subset F_1$ the segment after $F_1$ crosses $\rho_1$. Consider the alternative marking $p'$ on $\cD|_{T_Y}$ given by two interior consecutive lattice points on  $c \succ d$ on $R_1$ instead. Define:
\[
\tilde{\psi}(t)=(\cO(\cD)|_{\cD_t})^{-1}\otimes \cO(p'(t)+(D^2-1)p(t)).
\]
Let $f:\tilde{X_0} \rightarrow X_0$ be the normalization map and $\tilde{X}(R_i):=f^{*}(X(R_i))(i=1,2)$.
Since $\tilde{X}(R_1)^2 =0$, $\tilde{X}(R_2)^2 = D^2 -2$, and $\tilde{X}(R_1) \cdot \tilde{X}(R_2) = 1$, by Lemma \ref{lem:toricperiod}, $\tilde{\psi}$ still defines  a numerically trivial line bundle on $X_0$ and is equal to $1$ at $0$. Then, with the same  argument we used for the $n\geq 2$ case, we can show that  $\psi = z^{D}$ on $T_Y$. 
\end{proof}

\begin{lem}\label{lem:eigen-wei}
Let $\psi_i$, for $1\leq i \leq n$, be defined as in Proposition \ref{prop:boundaryperiod-para}. Then, $\psi_i$ and $z^{D_i}$ are eigenfunctions of the $T^{D}$ action with the same weight. 
\end{lem}
\begin{proof}
Given $g\in T^{D}$, it follows from Lemma \ref{lem:rel-tor-mar} that
\begin{align*}
    f_{\cD_j}(g\cdot p_j(t)) & = z^{e_{D_j}}(g) \cdot f_{\cD_j}(p_j(t)) = -z^{e_{D_j}}(g).\\
\end{align*}
Since \begin{align*}
     f_{\cD_j}(p_j(g\cdot t)) & = -1,
\end{align*}
we have
\[
f_{\cD_j}(g\cdot p_j(t)) = z^{e_{D_j}}(g) \cdot f_{\cD_j}(p_j(g\cdot t)).
\]
If $n\geq2$, by Lemma \ref{lem:pic-mul},
\begin{align*}
    \psi_i(t) =& (\mathcal{O}(\cD_i)|_{D_{g\cdot t}})^{-1} \otimes \cO(g\cdot p_{i-1}(t)+D_{i}^{2}g\cdot p_i(t)+ g\cdot p_{i+1}(t))\\
     =& (\mathcal{O}(\cD_i)|_{D_{g\cdot t}})^{-1} \otimes \cdot \cO(p_{i-1}(g\cdot t) + D^{2}_ip_i(g\cdot t) + p_{i+1}(g\cdot t))\\
     &  \cdot z^{-e_{D_{i-1}}-D^{2}_{i}e_{D_i}-e_{D_{i+1}}}(g)\\
     = & \psi_i(g\cdot t) \cdot z^{-e_{D_{i-1}}-D^{2}_{i}e_{D_i}-e_{D_{i+1}}}(g).
\end{align*}
Therefore, both $\psi$ and $z^{D_i}$ are eigenfunctions of $T^{D}$ with weight \[
z^{e_{D_{i-1}}+D^{2}_{i}e_{D_i}+e_{D_{i+1}}}.
\]
The proof for $n=1$ cases is the same. There, $\psi_1 =\psi$ and $z^{D}$ are both eigenfunctions of $T^{D}$ with weight $z^{D^2 e_D}$.
\end{proof}

\begin{prop}\label{prop:boundaryperiod} 
The conclusion of Proposition \ref{prop:boundaryperiod-para} holds if $(Y,D)$ is not a nontrivial toric blowup of the del Pezzo surface of degree 1 together with an irreducible nodal anti-canonical cycle.
\end{prop}
\begin{proof}
Again, by Lemma \ref{boundaryselfint}, $\cD_{i,t}^2 = D_i^2$. It remains to prove the case where $(Y,D)$ does not have a parallel configuration. We will follow the notation used in the proof of Proposition \ref{prop:boundaryperiod-para}.

First, we consider the case where  $(Y,D)$ does not have a toric blowdown to the del Pezzo surface of degree 1. Then, there is a toric blowdown $\pi:(Y,D)\rightarrow (Y',D')$ such that $(Y',D')$ has a parallel configuration $F'$. Let $D' = \sum_{j=1}^{n'}D'_j$. Denote cones in $\Sigma'=\Sigma_{(Y',D')}$ by $\sigma'_{j,j+1}$'s. By rescaling $F'$ if necessary, there exists a parallel polygon $F$, viewed as a polygon on $B'=B_{(Y',D')}$, such that if there exists a chain of $\pi$-exceptional divisors over some node $D'_j \cap D'_{j+1}$, $F \cap \sigma'_{j-1,j}$ is the chopping of $F' \cap \sigma'_{j-1,j}$, with one new edge in the interior of  $F' \cap \sigma'_{j-1,j}$ for each exceptional divisor between $D'_j$ and $D'_{j+1}$. We use a parallel polygon  $F$ like this for expository convenience.

Let $\bar{\cX}$ be the compactification of the mirror family using $F$. Then, since $\cD_{i,t}$ is still a Cartier divisor on $\bar{\cX}_t$ for $t$ on the toric boundary strata of $\Spec(\mathbb{C}[\NE(Y)])$, $\cD_i$ is a Cartier divisor on $\bar{\cX}$ by \cite[\href{https://stacks.math.columbia.edu/tag/062Y}{Tag 062Y}]{Sta}. Let $S'=\Spec(\mathbb{C} \left[ \NE(Y') \right])$. Let us first restrict to the locus $\iota_{\pi} (S')  \subset \Spec(\mathbb{C} \left[ \NE(Y) \right])$ where $\iota_{\pi}$ is the locally closed embedding defined in Subsection \ref{ssec:tor-bl} and compute the marked periods of boundary divisors of fibers of $\bar{\cX}$ over this locus. Notice that for any $\pi$-exceptional divisor $E$, $z^{[E]} =1$ over $\iota_{\pi} (S')$. Therefore, it suffices to consider the polyhedral decomposition of $F$ given by $\Sigma'$ instead of $\Sigma$ over $\iota_{\pi} (S')$. 

Let $0'$ be the image under $\iota_{\pi}$ of the unique toric fixed point of $S'$. Let $X_{0'}$ be the fiber of $\bar{\cX}$ over $0'$. Now, if $D_i$ is a $\pi$-exceptional divisor over a node $D'_{j} \cap D'_{j-1}$, we define $p'_{i}$ to be the alternative marking of $\cD_i \mid _{T_Y}$ given by two consecutive lattice points $a'_{i} \succ b'_i$ on the edge $F_i$ parallel to $\rho_i$ inside $F \cap \sigma'_{j-1,j}$. 

\emph{Claim: Over $\iota_{\pi}(T'_Y)$, $\psi_i(t)=z^{D_i}(t)$. }

\begin{figure}
\begin{minipage}{\linewidth} \begin{center} 
\vspace{1 cm}
\begin{tikzpicture}[scale=1] 
  \draw [-stealth](0,0) -- (4,0) node[anchor=west] {$\rho'_{j-1}$};
  \draw [-stealth](0,0)-- (0,4) node[anchor=south,yshift=.1cm] {$\rho'_{j}$};
  \draw [-stealth] (0,0) --(-4,4);
  \draw  [-stealth] (0,0) -- (-6, 3);
   \draw  [-stealth] (0,0) -- (-3, 6);
  \draw [-stealth] (0,0) -- (-4,0) node[anchor=east] {$\rho'_{j+1}$};
  \draw [red] (-3.5,0) -- (-3.5,3.5) -- (1.5,3.5) -- (2.5,3) -- (3,2.5) -- (3.5,1.5) -- (3.5,0);

\filldraw [black] (-3.5,3.5) circle (1pt);
\filldraw [black] (1.5,3.5) circle (1pt);
\filldraw [black] (2.5,3) circle (1pt);
\filldraw [black] (3,2.5) circle (1pt);
\filldraw [black] (3.5,1.5) circle (1pt);

  \end{tikzpicture}
 \caption{Blowing up between $D'_j$ and $D'_{j+1}$ corresponds to chopping down the corner in $F' \cap \sigma'_{j-1,j}$.}\label{fig:chopping down}
\end{center}\end{minipage}
\end{figure} 
If $D_i$ is not a $\pi$-exceptional divisor, then by the computations done in Proposition \ref{prop:boundaryperiod-para}, over $\iota_{\pi}(T'_Y)$, the marked period of $\cO(\pi^{*}\pi_{*}(\cD_i)\mid_{\cD_t})$ is equal to $z^{D_i}(t) = z^{\pi^{*}\pi_{*}(D_i)}(t)$. Therefore, to prove the claim, it remains to consider the case where $D_i$ is a $\pi$-exceptional divisor over a node $D'_j \cap D'_{j+1}$. Then, we define
$$
 \tilde{\psi}_i(t) =\begin{cases}
 (\cO(\cD_i)|_{\cD_t})^{-1}\otimes \cO( p_{i-1}(t) + D_i^{2} p'_{i}(t) + p'_{i+1}(t)) \,\,\,\\ \text{if $D_{i-1}$ is not an exceptional divisor, i.e., $\pi_{*}D_{i-1} = D'_j$}, \\
  (\cO(\cD_i)|_{\cD_t})^{-1}\otimes \cO( p'_{i-1}(t) +D_i^{2}p'_{i}(t) + p'_{i+1}(t)) \,\,\, \\\text{if $D_{i-1}$ is an exceptional divisor}, \\
  \end{cases}
$$
Here, $p'_{i+1}$ is the alternative marking on $\cD_{i+1}$ given by two consecutive lattice points on the segment $F_{i+1} \cap \sigma'_{j-1,j}$. In either case, $\tilde{\psi}_i(0')$ defines a numerically trivial line bundle on $X_{0'}$.  We can compute that $\tilde{\psi}_i(0') \equiv 1$ on $X_{0'}$ using Lemma \ref{lem:toricperiod}. Therefore, $\tilde{\psi}_i(0') \equiv 1$ on $\iota_{\pi}(S')$. Let $R_{\epsilon}$ be the forward ray $F_{\epsilon}$ for $\epsilon = i-1, i,i+1 \mod n$. Then, the only ray in $\Sigma_{(Y',D')}$ that $R_{\epsilon}$ can intersect is $\rho'_j$. Let $\xi_{\epsilon}$ be the direction of $F_{\epsilon}$ immediately after it crosses $\rho'_j$. By standard toric geometry, we have 
\[
\xi_{i-1} + \xi_{i+1} = -D_i^{2} \xi_{i}. 
\]
Thus, it follows that the contribution from the kinks on $\rho'_j$ due to the use of alternative markings is canceled. Therefore,  $\tilde{\psi}_i(t)= \psi_i(t)  \equiv 1 $ over $\iota_{\pi}(T_{Y'})$ and the claim is proved.
Now, consider the compactified mirror family over the entire base $\Spec(\mathbb{C}[\NE(Y)])$. Let $T^{L}$ be defined as in Subsection \ref{ssec:tor-bl}. By Lemma \ref{lem:eigen-wei}, both $\psi_i(t)$ and $z^{D_i}$ are  eigenfunctions for the $T^L$ action and have the same weight. Together with Lemma \ref{lem:tor-bl-compa}, we conclude that  $\psi_i(t)=z^{D_i}(t)$ over $T_Y$. 

Finally, let us deal with the case where $(Y,D)$ is the del Pezzo surface of degree 1 together with an irreducible nodal anticanonical cycle. Let $F$ be a parallel polygon with at least two interior consecutive lattice points on the unique edge $F_1$ and  $p'$ be the alternative marking given by them. Let 
\[
\tilde{\psi}(t) = (\cO(\cD)\mid_{\cD_t})^{-1}\otimes \cO(D^2p'(t)). 
\]
Let $(X_0,D_0)$ be the fiber over $0\in \Spec(\bC[\NE(Y)])$ and $f: \tilde{X_0} \rightarrow {X_0}$ be the normalization map. Then,  $\tilde{X_0}$ is isomorphic to $\mathbb{P}^2$ and $(f^{*}(D_0))^2 =1$. Again, $\tilde{\psi}$ defines a numerically trivial line bundle on $X_0$ and is equal to $1$ at $0\in \Spec(\bC[\NE(Y)])$ by Lemma \ref{lem:toricperiod}. Thus, we conclude that $\tilde{\psi}\equiv 1$  over $\Spec(\bC[\NE(Y)])$. Since $\tilde{\psi} = \psi \otimes z^{-D}$ over $T_Y$, we conclude that $\psi = z^{D}$ over $T_Y$. 
\end{proof}

\begin{remark}
The caveat here is that we are unable to compute the marked periods of all the boundary components for the compactified mirror family associated with the non-trivial toric blowup of the degree 1 del Pezzo surface. The method we have used breaks because of the singularities in irreducible components of the central fiber in this case. However, the computation we have done for the case of the degree 1 del Pezzo surface is sufficient for us to prove Proposition \ref{identity} later in Section \ref{sec:per-int}.
\end{remark}

\section{Tropical 1-cycles and period integrals}\label{sec:per-int}
Throughout this section, we assume that our positive Looijenga pair $(Y,D)$ has a  toric model $p:(Y,D) \rightarrow (\bar{Y},\bar{D})$ that we fix. The other cases can be deduced from the toric model case. 

 Recall that in Section \ref{sec:compa}, we gave a canonical compactification $(\bar{\cX}\mid_{T_Y}, \cD\mid_{T_Y})$ of the mirror family restricted to $T_Y \subset \Spec(\mathbb{C}[\NE(Y)])$ and showed that for each $t\in T_Y$, $(\bar{\cX}_t,\cD_t)$  is a generalized pair in Proposition \ref{can-comp-fiber}. In Section \ref{sec:marking-boundary}, we defined a canonical marking $p_i: T_Y \rightarrow \cD_i \mid_{T_Y}$ for each boundary component $\cD_i \mid_{T_Y}$. Now, we want to give a complete geometric interpretation of the family. In particular, we want to give a marking of the Picard group for fibers $\bar{\cX}\mid_{T_Y}$ by $\Pic(Y)$ and use it, together with the canonical marking on the boundary $\cD\mid_{T_Y}$, to identify $\bar{\cX}\mid_{T_Y}$ as the universal family of the moduli space of generalized marked pairs with a marking of the boundary and a marking of the Picard group by $\Pic(Y)$. 
 
 To achieve the goal  mentioned above, we will compute the marked period of each fiber of the family using the canonical marking $p_i$ $(1\leq i \leq n)$. This is done by a local analysis of the Gross-Siebert locus associated with a toric model $(Y,D)\rightarrow (\bar{Y},\bar{D})$. On this restriction, techniques from \cite{RS19} are used to construct tropical curves that correspond to exceptional curves of the toric model. In doing so, the deformation type of the family is determined, with the key observation being that the fibers are deformation equivalent to the original pair, $(Y,D)$. Note that in Proposition \ref{prop:boundaryperiod}, the part of the period map associated with $\langle D_i\rangle\subset \Pic(Y)$ is computed. The remainder of the period map is given by the exceptional curves of the toric model, which is carried out in this section.
\begin{remark} \label{rm:id-base}
Note the base of the mirror family is $T_Y = \Pic(Y)\otimes \mathbb{G}_m$ while the moduli space of generalized marked pairs is given by 
\[
T_{Y}^{*} = \Hom(\Pic(Y),\Pic^0(D))\simeq \Hom(\Pic(Y),\mathbb{G}_m).
\]
The isomorphism given by Poincare duality 
\begin{align*}
     \Pic(Y) & \rightarrow \Pic(Y)^{*} = \Hom(\Pic(Y),\mathbb{Z}) \\
     \mathcal{O}_{Y}(C) & \mapsto C
\end{align*}
between character lattices induces an isomorphism $T_Y \overset{\sim}{\rightarrow} T_{Y}^{*}$.  For the rest of the paper, we will implicitly use the identification between the two bases.
\end{remark}

\subsection{Gross-Siebert locus.}

The main results of \cite{RS19} give a way to construct singular cycles on smooth fibers near the central fiber of degenerations coming from \cite{GS11} and to also compute the periods of such cycles. These singular cycles come from tropical cycles on the corresponding integral affine manifold. In order to employ the techniques of \cite{RS19}, we need to align our constructions with one that would come from \cite{GS11}. To do this, we look the restriction of the  mirror family over certain loci in $\Spec(\mathbb{C}[\NE(Y)])$.

\begin{defn}[Gross-Siebert Locus] \label{def:GS-locus}
Let $p:(Y,D)\rightarrow (\bar{Y},\bar{D})$ be a toric model with exceptional curves $\{E_{ij}\}$ over $D_i$. Let $H$ be an ample divisor on $\bar{Y}$ such that $\NE(Y)\cap (p^*(H))^{\perp}$ is a face of $P=\NE(Y)$ generated by $\{E_{ij}\}$. Now, let $G$ be the prime monomial ideal that is generated by the complement of this face. In other words, $G$ is the prime monomial ideal generated by classes of curves that are \emph{not} contracted to a point by $p$. The open torus of $\Spec (\mathbb{C}[P]/G)$, denoted by $T^{GS}$, is called the \emph{Gross-Siebert locus} (GS locus for short) associated with the toric model given by $p$.
\end{defn}
Let $(\rho_i, f_{\rho_i})$ be the wall in the canonical scattering diagram supported on $\rho_i$. Let $X_i= z^{\varphi(v_i)}$ where $v_i$ is the primitive generator of $\rho_i$. Given a monomial ideal $I \subset R=\mathbb{C}[P]$, let $R_I = R/I$. Let $P_{\varphi}$ be the monoid generated by elements $(m,\varphi(m)+p) $ with $m \in B$ and $p\in P$. Recall the following lemma from \cite{GHK}:

\begin{lem} (Lemma 3.15, \cite{GHK})
If $f_{\rho_i}$ is viewed as an element of $\mathbb{C}[P_{\varphi}]\otimes R_I$ with $\sqrt{I}=\mathfrak{m}=P \setminus \{0\}$, then:
\[ f_{\rho_i}=g_{\rho_i}\prod_{j=1}^{l_i}(1+z^{[E_{ij}]}X_i^{-1})\]
where $g_{\rho_i}\equiv 1 \operatorname{ mod } G$. Here, $g_{\rho_i}$ is the product of contributions from all $\mathbb{A}^{1}$ classes other than $E_{ij}$'s.  
\end{lem}

For the rest of this section, we will let $I$ be an ideal such that $\sqrt{I}=G$. In order to apply the formulas in \cite{RS19}, we need to first restrict our mirror family to a locus transverse to the GS locus $T^{GS}$. More specifically, for the fixed toric model $p:(Y,D)\rightarrow (\bar{Y},\bar{D})$, let $E\subset P^{{\rm gp}}$ be the lattice generated by the classes
of exceptional curves and $P_{E}=P+E$ the localization of $P$ by inverting
elements in $E\cap P$. Let  $D^{\oplus}=\sum_{i}\mathbb{N}p^*([\bar{D}_{i}])$. The homomorphism of monoids 
\begin{align*}
P_{E}=P+E=D^{\oplus} \oplus E & \rightarrow\mathbb{N}\oplus E\\
\left(\sum_{i}a_{i}[p^*(\bar{D}_{i})],\sum_{i,j}n_{ij}[E_{ij}]\right) & \mapsto\left(\sum_{i}a_{i}[\bar{D}_{i}]\cdot H,\sum_{i,j}n_{ij}[E_{ij}]\right)
\end{align*}
 induces an embedding 
\[
\iota_{G}:{\rm Spec}(\mathbb{C}[E][t])\simeq {\rm Spec}(\mathbb{C}[E]) \times \mathbb{A}^{1}\hookrightarrow{\rm Spec}(\mathbb{C}[P_{E}]).
\]
Notice that the image of ${\rm Spec}(\mathbb{C}[E])\times \{0\}$ under
$\iota_{G}$ is the Gross-Siebert locus $T^{GS}$ (c.f. Definition \ref{def:GS-locus}) and for any $a\in{\rm Spec}\left(\mathbb{C}[E]\right)$,
$\{a\} \times \mathbb{A}^{1}$ is transverse to $T^{GS}$. Denote the
image of $\iota_{G}$ in ${\rm Spec}(\mathbb{C}[P_{E}])$ by $S^{GS}$ and denote $\iota_{G}(\{a\}\times \mathbb{A}^1) \subset S^{GS}$ by $T_a$.

Recall from Section $4$ of \cite{GHK} that $\mathcal{X}\mid_{T^{GS}} \simeq \mathbb{V}_{n}\times T^{GS}$ where, for $n\ge 3$, $\mathbb{V}_{n}$ is the $n$-vertex surface singularity defined as the $n$-cycle of coordinate planes in $\mathbb{A}^{n}$:
\[
\mathbb{V}_{n} := \mathbb{A}_{x_1, x_2} \cup \mathbb{A}_{x_2, x_3} \cup \cdots \cup \mathbb{A}_{x_n, x_1} \subset \mathbb{A}_{x_1, x_2, \cdots, x_n}. 
\]

Then, for any fixed $a \in \mathrm{Spec}\left(\mathbb{C}[E]\right)$, 
we can view $\mathcal{X}\mid_{T_a}$ as a 1-parameter degenerating family whose central fiber is $\mathbb{V}_n$. These 1-parameter families of degenerations are exactly the degenerations considered in the Gross-Siebert program and addressed in  \cite{RS19} which is why we consider the locus $S^{GS}$. Later, we will restrict our period computations to the mirror family over $S^{GS}$ and then extend to the entire family.

Fix a parallel polygon $F = P(W)$ associated with a $D$-ample divisor $W$. If $(Y,D)$ does not blow down to the del Pezzo surface of degree $1$, we require that $F$ is either a parallel configuration or the `chopping down'\ of one. As we showed in Proposition \ref{weilnomatter}, over $T_Y$, the family $(\bar{\cX} \mid_{T_Y},\cD \mid_{T_Y}$)  is independent of choice of the $D$-ample divisor $W$. 
For each $a\in {\rm Spec}(\mathbb{C}[E])$, the central fiber at $0\in\mathbb{A}^1$ of the 1-parameter family degeneration $\bar{\mathcal{X}}\mid_{T_a}$ is isomorphic to the union of toric varieties whose moment polytopes come from the polyhedral decomposition  $\underline{F}$ of $F$ given by $\Sigma \bigcap F$. \\

Now, we fix an $a\in {\rm Spec}(\mathbb{C}[E])$. We denote by $X_{a,0}$ the central fiber of the 1-parameter degeneration $\bar{\mathcal{X}}\mid_{T_a}$. 

\begin{defn}
Given a ray $\rho$ in $\Sigma$, let $X_{\rho}$ be the 1-dimensional toric stratum in $X_{a,0}$ corresponding to $\rho$. Let $L_{\rho} \subset X_{\rho}$ be the zero locus of $f_{\rho}$. 
We define the \emph{log singular locus} $L \subset X_{a,0}$ to be the union
$L= \bigcup_{\rho \in \Sigma^{(1)}} L_{\rho}$. We define the \emph{tropical amoeba} $\mathcal{A}$ to be the image of log singular locus $L$ under the moment map $\mu$. 
\end{defn}

After compactifying the mirror family, by Proposition 2.1 of \cite{RS19}, we get a degenerate moment map $\mu: X_{a,0} \rightarrow B'$ where $B'$ is an integral affine manifold with boundary.
As a topological space,  $B'$ is obtained from $F$ by factoring out the singularity at the origin $0\in F$. Instead of having a single singularity at the origin, singularities of the affine structure of $B'$ coincide with the tropical amoeba $\mathcal{A}$. 

By choosing a generic $a$, we can assume that $\mathcal{A}$ consists of distinct points, i.e., $L$ consists of distinct points, and their image under $\mu$ are also distinct as shown in Figure \ref{fig:focus}. Then, the integral affine structure of $B'$ can be viewed as obtained from that of $F$ by factoring the big singularity at the origin into $l_i$ simple focus-focus singularities $o_{ij}$ $(1 \leq j \leq l_i)$ on each ray $\rho_{i}$ where $l_i$ is the number of non-toric blowups on the corresponding divisor $\bar{D}_i$ in the toric model given by $p: (Y,D) \rightarrow (\bar{Y}, \bar{D})$. Moreover, $o_{ij}$ is the image of the point defined by $1+z^{[E_{ij}]}X_{i}^{-1}=0$ under the degenerate moment map. We call a connected component in the complement of the $o_{ij}$'s in $\rho_i$ a \emph{slab}.

It suffices to consider the cases where the $o_{ij}$'s have positions as shown in Figure \ref{fig:focus}. The descriptions and computations for all the other cases are exactly the same. Given slabs $\rho_{ij},\rho_{ij'}\subset\rho_{i}$, where $\rho_i$ is a ray, consider the affine parallel
transport $T$ along a path starting from a point $x$ in ${\rm Int}\rho_{ij}$
via ${\rm Int}\sigma_{+}$ to ${\rm Int}{\rho_{ij'}}$ and back to
$x$ through ${\rm Int}\sigma_{-}$ where $\sigma_{\pm}$ are  two maximal cells adjacent to $\rho_{i}$. We fix the convention that $\sigma_{+}$
is on the right side of the ray $\rho_{i}$.
Then, $T$ is of the form 
\begin{equation}
T(m)=m+\check{d}_{\rho_{i}}(m)\cdot m_{\rho_{ij}\rho_{ij'}},\quad m\in\Lambda_{x}\label{eq:monodromy}
\end{equation}
where $\check{d}_{\rho_{i}} \in \check{\Lambda}_{\rho_{i}}$ is a primitive
generator of $\rho_{i}^{\perp}\subset\check{\Lambda}_{\rho_{i}}$ such that
$\check{d}_{\rho_{i}}$ is non-negative on $\sigma_{+}$ and $m_{\rho_{ij}\rho_{ij'}}$
is parallel to $\rho_i$. Moreover, for adjacent slabs $\rho_{ij}$ and $\rho_{i (j+1)}$, we have $m_{\rho_{ij}\rho_{i
(j+1)}}= \nu_i$ where $\nu_i$ is the primitive generator of $\rho_i$. 


\begin{figure} 
\begin{minipage}{\linewidth} \begin{center}  \begin{tikzpicture} [circ/.style={shape=circle, inner sep=2pt, draw, node contents=},scale=1.2] \fill  (0,0) circle (2pt); \draw  node (x) at (-1,0) [circ]; \draw (0,0) -- (x); \draw  node (y) at (-2,0) [circ]; \draw (x) -- (y); \draw  (y) -- (-2.5,0);  \draw [dashed] (-2.5,0) -- (-4,0); \draw  node (z) at (-4.5,0) [circ]; \draw (-4,0) -- (z); \draw [->] (z) -- (-5.5,0);
\node [label=below:o] at (0,0){}; \node [label=below:$o_{i1}$] at (-1,0){}; \node [label=below:$o_{i2}$] at (-2,0){}; \node [label=below:$o_{i{l_{i}}}$] at (-4.5,0){};
\node [label=above:$\rho_{i0}$] at (-0.5,-0.1){}; \node [label=above:$\rho_{i1}$] at (-1.5,-0.1){};
\node [label=above:$\rho_{il_{i}}$] at (-5,-0.1){};
\node [label=left:$\rho_{i}$] at (-5.5,0){}; \end{tikzpicture}
\end{center} \end{minipage}\caption{Singularities on $\rho_i$}\label{fig:focus}
\end{figure}

We can obtain a consistent scattering diagram $\mathfrak{D'}$ on $B'$ from the canonical scattering diagram on $B$ by requiring consistency conditions around the focus-focus singularities and the origin $o \in B'$. Let us first consider the consistency condition around the focus-focus singularities. For $0\leq j \leq l_{i}$, the local model we have near $\rho_{ij}$ comes from blowing up on the divisor $\bar{D}_{i}$ at $j$ distinct points, with exceptional divisors $E_{i1}, \cdots, E_{ij}$ for $j\geq1$. Let
\begin{equation}
    \kappa_{\rho_{ij}} = \left[ D_{i} \right] + \sum_{k=j+1}^{l_{i}} \left[  E_{ik}  \right]. 
\end{equation}
When $j=l_i$, $f_{\rho_{il_{i}}}$ agrees with $f_{\rho_i}$ in \cite{GHK}:  
\begin{equation}
   f_{\rho_i} = f_{\rho_{{il_{i}}}} =g_{\rho_i}  \prod_{s=1}^{l_i} (1+z^{\left[E_{is}\right]}X_{i}^{-1}).    
\end{equation}
The consistency around focus-focus singularities translates into the following compatibility condition: Given any $\rho_{ij}$ and $\rho_{ij'}$
in $\rho_{i}$, we
have 
\begin{equation} 
z^{\kappa_{\rho_{ij'}}}f_{\rho_{ij'}}=z^{m_{\rho_{ij'}\rho_{ij}}}\cdot z^{\kappa_{\rho_{ij}}}f_{\rho_{ij}}.\label{eq:compa-bary}
\end{equation}
Then, the kink $\kappa_{\rho_{ij}}$ together with the compatibility condition \ref{eq:compa-bary} determines the wall-crossing functions 
\begin{equation}
    f_{\rho_{ij}}=g_{\rho_i} \prod_{s=j+1}^{l_i}(1+z^{[-E_{is}]}X_i) \prod_{s'=1}^{j}(1+z^{[E_{is'}]}X_i^{-1})  
\end{equation}
for $0 \leq j \leq l_{i}$.

For the consistency condition around the origin $o\in B'$, we need to add additional walls emanating from $o$ with support in the interior of 2-dimensional cones in $\Sigma$, i.e., walls \emph{of codimension 0}. Near $o$, the scattering diagram on $B'$ agrees with the asymptotic scattering diagram $\bar{\mathfrak{D}}$, as in Theorem 3.25 of \cite{GHK}, by pushing all focus-focus singularities to infinity. However, as mentioned in \cite{RS19}, the walls of codimension 0 do not play a role in the period computations. Therefore, we will not give explicit descriptions of these walls here. 

Mimicking Construction 2.14 of \cite{GHK}, we can construct a scheme $X^{\circ}_{I,\mathfrak{D'}}$ using the scattering diagram $\mathfrak{D'}$, in agreement with the order-by-order construction of schemes used in \cite{RS19}. The consistency of the scattering diagram $\mathfrak{D'}$ guarantees that the scheme $X^{\circ}_{I,\mathfrak{D'}}$ is well-defined. For the slab $\rho_{ij}$, the corresponding scheme  $\mathrm{Spec}(R_{\rho_{ij,I}}) $ is given by 
\[
R_{\rho_{ij},I}=R_{I}\left[\Lambda_{\rho_i}\right]\left[Z_{\rho_{ij}}^{+},Z_{\rho_{ij}}^{-}\right]/\left(Z_{\rho_{ij}}^{+}Z_{\rho_{ij}}^{-}-z^{\kappa_{\rho_{ij}}}f_{\rho_{ij}}\right).
\]
Here, $Z_{\rho_{ij}}^{+} = z^{\zeta_{j}}$ where $\zeta_{j}$ is a primitive generator of $\Lambda_{\sigma_{+}}/\Lambda_{\rho_{i}}$
pointing from $\rho_{i}$ to $\sigma_{+}$ and $Z_{\rho_{ij}}^{-} = z^{\tilde{\zeta_{j}}}$ where $\tilde{\zeta}_{j}$ is the parallel transport of $-\zeta_{j}$ via a path passing through the interior of $\rho_{ij}$. For more details, see Section 2.2 of \cite{GHK}.

\begin{defn}
Denote by $U_{sim}$ the locus in ${\rm Spec}(\mathbb{C}[E])$ where $B'$ has only focus-focus singularities, i.e., where points in $\mathcal{A}$ do not collide as shown in Figure \ref{fig:focus}
\end{defn}

\subsection{Singular cycles  from tropical 1-cycles} \label{ssec:sing-trop}
Denote the singular locus of the affine structure $B'$ by $\Delta$ and the embedding of the regular locus $B'\setminus{\Delta}$ into $B'$ by $\iota$. We also let $\check{\Lambda}$ be the sheaf of integral cotangent vectors $\mathcal{H}om\left( \Lambda,\underline{\mathbb{Z}}\right)$ on $B'\setminus{\Delta}$. 

\begin{defn}
A \emph{tropical 1-cycle} or \emph{tropical curve} is a twisted singular one-cycle on $B'\setminus \Delta$ with coefficients in the sheaf of integral tangent vectors $\Lambda$. 
\end{defn} 
More concretely, a tropical one cycle can be viewed as an oriented graph $\Gamma$ with a map $f:\Gamma\rightarrow B'\setminus \Delta$ and a section $\xi_e\in (f|_e)^* \Lambda$ for each edge $e\subset \Gamma$. Furthermore, at each vertex $v$, the balancing condition is satisfied, i.e. $\sum_{e \ni v} \varepsilon_{e,v} \xi_e=0$. Here,  $\varepsilon_{e,v}=1$ if the edge is pointing towards $v$ and $\varepsilon_{e,v}=-1$ otherwise. Throughout the paper, we assume that $\xi_{e}$ is a primitive element in $(f\mid_e)^* \Lambda$. 

To a tropical $1$-cycle $\beta_{\trop}$, we will give a corresponding singular cycle in $X_{a,0}$. For brevity, we explain how this is done in the case that every edge $e$ is embedded into a single maximal cell. We refer to \cite{RS19} for the general construction. In Proposition $2.1$, \cite{RS19} constructs a degenerate momentum map, $\mu:X_{a,0}\rightarrow B'$, where $X_{a,0}$ is the singular central fiber. Each maximal cell $\sigma$ has a momentum map $\mu_{\sigma}:X_{\sigma}\rightarrow \sigma$. Because we are in the case where the gluing data is trivial, the $\mu_{\sigma}$'s agree on codimension one strata and we can form a global degenerate momentum map, which is used to construct singular cycles from tropical ones. Fix a tropical cycle $\beta_{\trop}$. Then, for each edge $e$, we choose a section $S_e\subset X_{\sigma}$ of $\mu_{\sigma}$ over $e$. We choose these sections so that we have compatibility at vertices.  If the edge $e$ has a vertex $v$ on a codimension 1 cell $\rho$, we further require the condition that $S_e$ is contained in the closure of an orbit of the 1-parameter subgroup of $\rm{Spec}(\mathbb{C}[\Lambda_{\sigma}])$ that fixes $X_\rho \subset X_\sigma$ pointwise. Notice that this condition is equivalent to saying that for any $m \in \Lambda_{\rho}$, the monomial $z^m$ is constant on $S_e$.  To an edge $e$ carrying the tangent vector $\xi_e$, we then associate to it a chain $\beta_e$ as the orbit of $S_e$ under the subgroup of $\Hom(\Lambda_{\sigma},U(1))$ which maps  $\xi_e$ to $1$. At a vertex $v$, the balancing condition of the tropical cycle implies that the negative of boundaries over $v$ of the chains $\beta_e$ with $e\ni v$ bound an $n$-chain $\Gamma_v$ over $v$. Then, the associated singular cycle in $X_{a,0}$ associated with $\beta_{\trop}$ is:
\[ \beta=\sum_e \beta_e+\sum_v \Gamma_v.\]
For a vertex $v$ and an adjacent edge $e$, we will denote the point of intersection of $S(e)$ with $\mu^{-1}(v)$ by $S(v)$. The singular 2-cycle $\beta$ is well-defined up to integral multiples of the fiber class of the degenerate moment map.

We will specifically be interested in tropical cycles that correspond to exceptional curves of a toric model.

\begin{lem}\label{focusfocusvector}
If a tropical curve goes around a focus-focus singularity, as shown  
in Figure \ref{fig:tropicalfocus}, then the edge leading away from the singularity must carry a vector $\zeta$ parallel to $\Lambda_{\rho}$. Moreover, $\zeta$ is primitive if and only if the cycle carries a primitive
generator of $\Lambda/\Lambda_{\rho}$ when it starts to surround
the singularity. 
\end{lem}

\begin{proof}
Let $e_{\rho}$ be a generator of $\Lambda_{\rho}$ and $\check{d}_{\rho}$
be a primitive vector such that $\left(\check{d}_{\rho}\right)^{\perp}=\Lambda_{\rho}$.
Let $\xi$ be the vector the tropical 1-cycle carries when it starts
to go around the focus-focus singularity and $\xi^{'}$ the vector
after applying the monodromy transformation to $\xi$. By the balancing
condition, 
\begin{align*}
\xi+\zeta & =\xi^{'}=\xi+<\check{d}_{\rho},\xi>e_{\rho}\\
\zeta & =<\check{d}_{\rho},\xi>e_{\rho}
\end{align*}
Therefore, $\zeta$ is parallel to $\Lambda_{\rho}$ and it is a primitive
generator of $\Lambda_{\rho}$ if and only if $\xi$ is a primitive
generator of $\Lambda/\Lambda_{\rho}$. 
\end{proof}

\begin{figure}
\begin{minipage}{\linewidth} \begin{center}  \begin{tikzpicture}[circ/.style={shape=circle, inner sep=2pt, draw, node contents=}] \draw  node (x) at (0,0) [circ]; \draw [ decoration={markings,  mark=at position 0.25 with {\arrow{>}}, mark=at position 0.75 with {\arrow{>}}},         postaction={decorate}](0,0) circle [radius=1cm]; \draw  [dashed] (-2,0)-- (x); \draw (x) -- (3.5,0); \node [label=right:$\xi^{'}$] at (-60:1cm){}; \node [label=above:${\xi}$] at ([yshift=8pt] 60:1cm){}; \node [label=right:$\zeta$] at (0.8,1.3){}; \draw [->] (-50:1cm) -- ([xshift=-12pt, yshift=12pt] -50:1cm); \draw [->] (60:1cm) -- ([yshift=12pt] 60:1cm); \draw  [decoration={markings, mark=at position 0.5 with {\arrow{>}}}, postaction={decorate}](45:1cm) .. coordinate [pos=.3] (a) controls (1.5,1) and  (2,1.5) .. ([xshift=3cm, yshift=1cm] 45:1cm );  \draw [->] (a) -- +(-12pt,0);  \draw [->] (-1,-2.5) -- (-2,-2.5);   \draw [->] (-1,-2.5) -- (-1,-1.5);  \node [label=above:$e_{\rho}$] at (-2,-2.5){};  \node [label=right:$d_{\rho}$] at (-1,-1.7){};  \node [label=left:$\rho$] at (-2,0){};
\draw [<-] (3.5,1.25) [out=300, in=60] to (3.5,-1.25);
\node [label=below:$\xi \mapsto \xi+  \langle {{\check{d}}_{\rho}, \xi} \rangle  e_{\rho}$ ] at (3.5,-1.25) {};
\end{tikzpicture} \captionof{figure}{A tropical cycle surrounding a focus-focus singularity.} \label{fig:tropicalfocus}
\end{center} \end{minipage}
\end{figure}


By rescaling $F$ if necessary, we can assume that $F$ satisfies the condition in Construction \ref{cons:marking}, i.e., for each edge $F_i$, there exists a maximal cone $\sigma$ in $\Sigma$ that contains two consecutive lattice points $a_i \succ b_i$ on $F_i$. 
\begin{defn}[Exceptional tropical $1$-cycles]\label{def:excep-1-cycle}
Let $F_i$ be the edge parallel to $\rho_i$. If $F$ is a parallel configuration, then $\underline{F}$ has two 1-cells contained in $F_i$, one in the 2-cell $\sigma_{i-1,i} \cap F$ and the other in the 2-cell $\sigma_{i-2,i-1} \cap F$. Denote these two 1-cells by $R_i, R_i'$ respectively. 
Then, consider a tropical $1$-cycle contained in the  2-cell $\sigma_{i-1,i} \cap F$ as in Figure \ref{exceptionaltrop} that starts on the interior of $R_i$ and then goes around a focus-focus singularity on the ray $\rho_{i}$. The labeling of the tropical $1$-cycle is shown in Figure \ref{exceptionaltrop} where $d_{\rho_i}$ is a primitive generator of $\Lambda/\Lambda_{\rho_i}$. 

If $F$ is not a parallel configuration, we do a similar construction. The only difference is that the tropical $1$-cycle may not be contained in a single maximal cell in $\Sigma \cap F$ anymore since $F$ is no longer a parallel configuration. In particular, the edge $e_0$ that starts from $F_i$ may pass through multiple rays in $\Sigma$. Moreover, we require that $e_0$ starts from a 1-cell contained in $F_i$ that contains two consecutive lattice points. By our assumption about $F$, this condition can be satisfied. The labeling on each edge remains the same. Moreover, for the simplicity of computations we will do later, we assume that when the tropical 1-cycle crosses a ray $\rho_i$, it always crosses the slab $\rho_{il_i}$ as in Figure \ref{fig:focus} and it never crosses the same ray twice. 

We call such a tropical 1-cycle as we constructed, an \emph{exceptional tropical 1-cycle}.
\end{defn}
Moreover, the following lemma follows from \cite{bauer}:
\begin{lem}[Lemma 4.10, Lemma $4.11$, \cite{bauer}]
Exceptional tropical 1-cycles have self-intersection $-1$ and two exceptional 1-tropical cycles that go around different focus-focus singularities intersect trivially.
\end{lem}

Thus, for each focus-focus singularity $o_{ij}$ on the ray $\rho_i$, there is an exceptional tropical 1-cycle $\beta_{\mathrm{trop},ij}$ $(1 \leq j \leq l_i)$ going around it with self-intersection -1. \par 
\begin{figure} 
\begin{minipage}{\linewidth} \begin{center}  \begin{tikzpicture}[circ/.style={shape=circle, inner sep=2pt, draw, node contents=}, scale=1.2] \draw  node (x) at (0,0) [circ]; \draw [ decoration={markings,  mark=at position 0.07 with {\arrow{<}}, mark=at position 0.17 with {\arrow{<}} },         postaction={decorate}](0,0) circle [radius=1cm]; \draw  [dashed] (x)-- (-3,0); \draw (x) -- (3,0); \node [label=right:$d_{\rho_i}+\nu_{i}$] at (-75:1cm){}; \node [label=above:$d_{\rho_i}$] at ([xshift=-3pt, yshift=6pt] 75:1cm){}; \node [label=right:$\nu_{i}$] at (0.8,1.2){}; \draw [->] (-60:1cm) -- ( [xshift=-12pt, yshift=12pt] -60:1cm); \draw [->] (75:1cm) -- ( [yshift=12pt]75:1cm); \draw  [decoration={markings, mark=at position 0.5 with {\arrow{<}}}, postaction={decorate}](45:1cm) .. coordinate [pos=.3] (a) controls (1.5,1) and  (2,1.5) .. ([xshift=3cm, yshift=1cm] 45:1cm ); \fill [black] (45:1cm) circle (2pt); \fill [black] (180:1cm) circle (2pt); \fill [black] (270:1cm) circle (2pt); \fill [black] (0:1cm) circle (2pt); \fill [black] ([xshift=3cm, yshift=1cm] 45:1cm )circle (2pt); \node [label=below:$v_{0}$] at ([xshift=3cm, yshift=1cm] 45:1cm ){};
\node [label=right:$v_{1}$] at (44:0.9cm){}; \node [label=left:$v_{4}$] at (170:0.9cm){}; \node [label=below:$v_{3}$] at (270:0.9cm){}; \node [label=right:$v_{2}$] at (10:0.9cm){}; \node [label=left:$e_{1}$] at (18:1.1cm){}; \node [label=below:$e_{3}$] at (90:1.1cm){}; \node [label=right:$e_{2}^{'}$] at (215:1cm){}; \node [label=above:$e_{2}$] at (309:1cm){}; \node [label=below:$e_{0}$] at (2.6,1.3){};
\node [label=right:$R_i$] at (-3.6,1.7) {};
\draw [->] (a) -- +(-12pt,0); \draw [->] (-2,-1.5) -- (-3,-1.5); \draw [->] (-2,-1.5) -- (-2,-0.5); \node [label=above:$\nu_{i}$] at (-3,-1.5){}; \node [label=right:$d_{\rho_i}$] at (-2,-0.7){}; \node [label=above: $o_{ij}$] at (0,-0.1) {};
\draw (-3,1.7) -- (4,1.7);
\end{tikzpicture}\caption{An exceptional tropical cycle inside a parallel configuration}\label{exceptionaltrop} \end{center} \end{minipage}
\end{figure}



\begin{cons}\label{def:singular-2-cycles}
(Construction of singular 2-cycles from tropical 1-cycles.) Consider the toric model $p:(Y,D)\rightarrow (\bar{Y},\bar{D})$ where $\bar{D}_i$ is blown up $l_i$ times and the locus $S^{GS}\subset \mathrm{Spec}(\mathbb{C}[P])$, which is transverse to the GS locus $T^{GS}$. For a fixed $a \in U_{sim} \subset \mathrm{Spec}(\mathbb{C}[E])$, consider the 1-parameter family $T_a = \{a\} \times \mathbb{A}^1 \subset S^{GS}$. By the construction in Section 2 and Appendix A of \cite{RS19}, there exists an analytic disc $\mathbb{D}_a \subset \mathbb{A}^1$ around $0\in \mathbb{A}^{1}$ such that for any $t \in \left\{ a \right\} \times \mathbb{D}_a$, given a tropical 1-cycle $\beta_{\mathrm{trop}}$ with $\beta$ a corresponding singular 2-cycle on $X_{a,0}$, there is a singular 2-cycle $
\beta_t$ on $\bar{\mathcal{X}_t}$ that is a deformation of $\beta$. 

Now, if we restrict to the locus $U_{sim} \subset \mathrm{Spec}(\mathbb{C}[E])$, on a ray $\rho_i$, we can produce $l_i$ exceptional tropical curves $\beta_{\mathrm{trop},ij}$ $(1 \leq j \leq l_i)$ such that $\beta_{\mathrm{trop},ij}$ goes around the focus-focus singularity $o_{ij}$. The labeling on our exceptional tropical 1-cycles guarantees that when we apply the construction in Subsection \ref{ssec:sing-trop} for each tropical 1-cycle $\beta_{\mathrm{trop},ij}$,  we get a corresponding connected singular 2-cycle on $X_{a,0}$  for each fixed parameter $a \in U_{sim}$.
Then, by considering the union of $\left\{ a \right\} \times \mathbb{D}_a$ for each $a\in U_{sim}$ and then taking the intersection with the locus in $T_Y$ that parametrizes smooth fibers, we get an analytic open neighborhood $W \subset S^{GS}$ parametrizing smooth fibers such that for any $s \in W \cap T_Y$, the singular 2-cycle associated with the exceptional tropical 1-cycle $\beta_{\mathrm{trop},ij}$ on the central fiber deforms to a corresponding singular 2-cycle $\beta_{ij,s}$ on $\bar{\mathcal{X}}_s$. By Lemma 7.11 of \cite{Sym}, the self-intersection number of $\beta_{ij,s}$ is also -1. We claim for each $s \in W\cap T_Y$, there is an exceptional curve $E_{ij,s}$ in $\bar{\mathcal{X}}_s$ representing the same homological class as $\beta_{ij,s}$, which follows from the lemma below.
\end{cons}

\begin{lem}\label{lem:(-1)-class}
For each $s \in W\cap T_Y$, there is an exceptional curve in $\bar{\mathcal{X}}_s$ representing the same homological class as $\beta_{ij,s}$. Moreover, $(\bar{\mathcal{X}}_s, \cD_s)$ is deformation equivalent to the original pair $(Y,D)$. 
\end{lem}
\begin{proof}
Without loss of generality, we could assume that the original pair $(Y,D)$ is generic. Let $\omega$ the class of the $D$-ample divisor $W$ we used to construct $F=P(W)$. By Theorem 5.4 and Lemma 5.6 of \cite{EF21}, there exists an almost toric fibration (in the sense of Definition 4.5 of \cite{Sym}) $(Y, D, \omega) \rightarrow B'$ such that a (-1)-sphere representing the class of $E_{ij}$ fibers over the path $\gamma_{ij}$ in $B'$ which is identified pair of edges of the triangle with the vertex $o_{ij}$ and base on $\bar{F}_i$ that we deleted while constructing $F$ from $\bar{F}=P(\bar{W})$ (c.f. Subsection \ref{ssec:cano-poly}). By Theorem 0.1, Theorem 0.2 of \cite{Ar21} and Lemma 8.8 of \cite{AAP24}, $(\bar{\cX}_s, \cD_s)$ also admits an almost toric fibration over the same base $B'$ such that a (-1)-sphere representing the class of $\beta_{ij,s}$ is mapped to $\gamma_{ij}$. By Corollary 5.4 of \cite{Sym}, $(\bar{\cX}_s, \cD_s)$ is diffeomorphic to $(Y,D)$ via a diffeomorphism $f$ such that $f^{*}([D_i]) = [\cD_{i,s}]$ and $f^{*}([E_{ij}]) = [\beta_{ij,s}]$ for each $i$ and $j$. Then, by Theorem 5.14 of \cite{Fr2}, the diffeomorphism $f$ will preserve the generic ample cones and hence $(\bar{\cX}_s, \cD_s)$ and $(Y,D)$ are deformation equivalent. Since $\bar{\cX}_s \setminus \cD_s$ is affine and smooth, $(\bar{\cX}_s, \cD_s)$ has no $(-2)$-curves and its generic ample cone is equal to its ample cone. Since $(Y,D)$ is generic, its generic ample cone is also equal to the ample cone. By Corollary 4.3 of \cite{Fr2}, $[E_{ij}]$'s generate distinct faces of the ample cone of $(Y,D)$. Since $f$ preserves the ample cones,  $[\beta_{ij,s}]$'s also generate distinct faces of the ample cone of $(\bar{\cX}_s, \cD_s)$ and are presented by exceptional curves by Corollary 4.3 of \cite{Fr2}.
 
\end{proof}

\begin{remark}
In the above construction, given a point $s \in W\cap T_Y$, $s$ is contained in a unique 1-parameter family $T_a$ and the exceptional tropical 1-cycle $\beta_{\mathrm{trop}, ij}$ depends on $a$. However, for expositional simplicity, we do not emphasize this dependence.
\end{remark}

\begin{lem} \label{lem:full-dim}
Let $W$ be the same as in Construction \ref{def:singular-2-cycles}. Then, under the $T^D$ action, $T^{D} \cdot ({W \cap T_{Y}})$ is a full-dimensional subset of $T_Y$.
\end{lem}
\begin{proof}
The weight map as in Definition \ref{def:weight-map} induces a homomorphism $\varphi': T^{D} \rightarrow T_Y$. Let $\mathrm{Aut}^{0}(D)$ be the identity component of the automorphism group of $D$. Using the identification $\mathrm{Aut}^{0}(D) \simeq \mathbb{G}_{m}^{n} \simeq {T^D}$ and the identification $T_Y \simeq T_{Y}^{*}$, the homomorphism $\varphi': T^{D} \rightarrow T_Y$ coincides with the homomorphism $\varphi:\mathrm{Aut}^{0}(D) \rightarrow T_{Y}^{*}$ as defined in Lemma 2.5 of \cite{moduli}. 

Consider the injective homomorphism $T_{\bar{Y}} \hookrightarrow T_Y$ induced by the pullback map $\Pic(\bar{Y}) \hookrightarrow \Pic(Y)$. Its image is the locus where $z^{[E_{ij}]}$ is equal to $1$ for all $i,j$. Moreover, the composition $T_{\bar{Y}} \hookrightarrow T_Y \twoheadrightarrow T_{\bar{Y}}$ is the identity map. Therefore, given $t \in T_Y$, $t$ is determined by its  image under the surjective homomorphism $T_Y \twoheadrightarrow T_{\bar{Y}}$ and $z^{[E_{ij}]}(t)$ for all $i,j$.  

Now, consider the relative torus $T^{\bar{D}}$ and the injective homomorphism $T^{\bar{D}}\hookrightarrow T^{D}$ induced by the map $e_{D_i} \mapsto e_{\bar{D}_i}$. Also consider the surjective homomorphism $T_Y \twoheadrightarrow T_{\bar{Y}}$ induced by $z^{[\bar{C}]} \mapsto z^{p^{*}(\bar{C})}$. The composition $\phi:T^{\bar{D}} \hookrightarrow T^{D} \overset{\varphi'}{\rightarrow} T_Y \twoheadrightarrow T_{\bar{Y}} \overset{\sim}{\rightarrow} T_{\bar{Y}^{*}}$ agrees with 
the surjective homomorphism $\bar{\varphi}: T^{\bar{D}} \rightarrow T_{\bar{Y}^{*}}$ of Lemma 2.7 of \cite{moduli} since $p^{*}(\bar{C})\cdot D_i = \bar{C} \cdot \bar{D}_i$ for any $\bar{C}\in A_{1}(\bar{Y})$ and any $i$.
Thus, the statement of the Lemma follows from the fact that the composition $T^{D} \overset{\varphi'}{\rightarrow} T_Y \twoheadrightarrow T_{\bar{Y}}$ is surjective.  
\end{proof} 

Let $e(Y \setminus D)$ be the topological Euler characteristic of $Y\setminus D$. Recall that the charge $Q(Y,D)$ of a Looijenga pair $(Y,D)$ is
\[
Q(Y,D) = 12 - D^{2} - n
\]
where $n$ is the number of irreducible components of $D$.  Moreover, we have $Q(Y,D) = e (Y\setminus D) \geq 0$ with equality if and only if $(Y,D)$ is toric.

\begin{lem}\label{lem:def-type} There exists an analytic open neighborhood $U\subset T_Y$ parametrizing smooth fibers such that for every  $s\in U$, there is a collection of $(-1)$-curves on $\bar{\mathcal{X}}_{s}$ that has the same combinatorial type as the toric model $p: (Y,D) \rightarrow (\bar{Y}, \bar{D})$. Consequently, the surface pairs of the family are deformation equivalent to $(Y,D)$. 
\end{lem}
\begin{proof}
We continue from Construction \ref{def:singular-2-cycles} and use the same notation as there. The deformation type of the family over $W$ follows from Lemma \ref{lem:(-1)-class}. For each $s\in W\cap T_Y$, let $Z_s = \bar{\mathcal{X}}_s$. By constructing an exceptional tropical curve for each focus-focus singularity, for each $s \in W \cap T_Y$,  we obtain a collection of (-1)-curves on $(Z_s, \cD_s)$ with the same combinatorics as the exceptional collection of $p$. Let $(\bar{Z}_{s},\bar{\cD}_s)$ be the pair we get after blowing down the collection of (-1)-curves constructed from exceptional tropical 1-cycles. 

We know the self-intersection numbers of the boundary components $\cD_{i,s} = \cD_i|_{Z_{s}}$ from Lemma \ref{boundaryselfint}.  Then by a simple computation, we know that the charge of  $(\bar{Z}_{s},\bar{\cD}_s)$ is zero.  This implies that the pair $(\bar{Z}_{s},\bar{\cD}_s)$ is toric and is isomorphic to the pair $(\bar{Y},\bar{D})$. Thus, $(Z_s,
\cD_s)$ and $(Y,D)$ have the same toric model with the same combinatorics.

 Using the equivariant $T^D$ action, by Lemma \ref{lem:full-dim} we can extend $W \cap T_Y$ to an analytic open neighborhood $U $ in $T_Y$ such that for each $s \in U$, the statement of the proposition holds for $\bar{\mathcal{X}}_{s}$. 
\end{proof}

\begin{prop} \label{prop:mar-family} 
Consider a positive Looijenga pair $(Y,D)$ and its mirror family $\cX \rightarrow \Spec(\bC[\NE(Y)])$. Let $(\bar{\cX},\cD)\rightarrow \Spec(\bC[\NE(Y)])$ be the canonical compactification of $\mathcal{X}$ as defined in Subsection \ref{ssec:cano-poly}. Then, the restriction of the compactified mirror family to the algebraic torus $T_Y=\Pic(Y)\otimes \mathbb{G}_m$ is a family $(\bar{\cX}\mid_{T_Y},\cD \mid_{T_Y})$ of generalized pairs deformation equivalent to $(Y,D)$. Furthermore, $(\bar{\cX} \mid _{T_Y},\cD \mid_{T_Y})$ comes with a natural marking $p_i$ of the boundary $\cD$ and a marking $\mu$ by $\Pic(Y)$. 
\end{prop}
\begin{proof}
It suffices to consider the case where $(Y,D)$ has a toric model. Indeed, by Lemma \ref{lem:compa-para-tor-bl}, if $(Y',D') \rightarrow (Y,D)$ is a toric blowup, a marking of the Picard group for $(\bar{\cX'}\mid_{T_{Y'}}, \cD' \mid_{T_{Y'}})$ will induce a marking of the Picard group for $(\bar{\cX}\mid_{T_Y}, \cD\mid_{T_Y})$.

The marking of the boundary $\mathcal{D}\mid_{T_Y}$ is already done in Section \ref{sec:marking-boundary} and the deformation type of the compactified mirror family restricted to $T_Y$ is established in Lemma \ref{lem:def-type}. It remains to show that we get a marking of the Picard group by $\mathrm{Pic}(Y)$. Let $U \subset T_{Y}$ be the analytic open neighborhood as in Lemma \ref{lem:def-type}. For any $s\in U$, let $E_{ij,s}$ be the exceptional curve on $\bar{\cX}_s$ corresponding to the tropical 1-cycle $\beta_{\mathrm{trop},ij}$. By letting $\cO_{\bar{\cX}_s}(E_{ij,s})$ go to $\cO_{Y}(E_{ij})$ and $\cO_{\bar{\cX}_s}(\cD_{i,s})$ go to $\cO_Y(D_i)$, we get a marking of the Picard group of $\bar{\cX}_s$ by $\mathrm{Pic}(Y)$. 

Note that the marking of the compactified mirror family over the analytic neighborhood $U$ can be extended to the entire base $T_Y$. Indeed,  we can  use parallel transport on the base to obtain a configuration of $(-1)$-curves of the combinatorial type of the toric model $p:(Y,D)\rightarrow (\bar{Y},\bar{D})$ on the minimal resolution of $\bar{\mathcal{X}}_s$ over any point $s\in T_Y$. If $\bar{\mathcal{X}}_s$ is singular, there is an ambiguity, coming from  the existence of vanishing cycles, on the choice of the marking of the Picard group of the minimal resolution of $\bar{\mathcal{X}}_s$. However, since these vanishing cycles are internal (-2)-curves on the minimal resolution, we still obtain a marking $\mu$ of the Picard group, in the sense of Definition \ref{def:gen-mar}, for each fiber in the family $\bar{\mathcal{X}} \rightarrow \Spec(\bC[\NE(Y)])$.

\end{proof}

\subsection{Period Integrals}
Let $S^\circ \subset S=\Spec(\bC[P])$ be the locus where $\bar{\mathcal{X}}_s$ is smooth. By Proposition 2.31 of \cite{GHK}, the dualizing sheaf $\omega_{\mathcal{X}/S}$ is trivial. Moreover, the fiberwise compactification $(\bar{\mathcal{X}},\mathcal{D})$  is such that $\omega_{\bar{\mathcal{X}}/S} (\mathcal{D})$ is trivial by Proposition \ref{compa-fiber}. Let $\Omega$ be the global section of $\omega_{\bar{\mathcal{X}}/S} (\mathcal{D})$ normalized so that for each $s\in S^\circ$, $\int_{\Gamma_s} \Omega = (2\pi i)^2$, where $\Gamma_s$ is the class of a real 2-torus that generates the kernel of the exact sequence 
\begin{equation}
    0 \rightarrow \mathbb{Z} \rightarrow H_{2}(\mathcal{X}_s, \mathbb{Z}) \rightarrow Q_s \rightarrow 0,
\end{equation}\label{eq:exa}
where $Q_s = \left\langle D_1, D_2 \cdots, D_n \right\rangle^{\perp}  \subset H_{2}(\bar{\mathcal{X}}_s,\mathbb{Z})$.

Let $W \subset S^{GS}$ be the analytic open neighborhood as in Construction \ref{def:singular-2-cycles}. Then, as shown in Proposition \ref{prop:mar-family}, for each $s \in W \cap T_Y$, $(\bar{\mathcal{X}}_s, \cD_s)$ is a pair deformation equivalent to $(Y,D)$ together with a marking $\mu: \mathrm{Pic}(\bar{\mathcal{X}}_s) \rightarrow \mathrm{Pic(Y)}$. In Subsection \ref{ssec:per-excep}, we apply the formulas for period integrals in \cite{RS19} and prove the following proposition:

\begin{prop} \label{prop:mar-per-id} 
Consider the marked family as in Proposition \ref{prop:mar-family}. Then, given $s\in W\cap T_Y$,
the marked period of the element in $\Pic(\bar{\cX}_s)$ marked by $\cO_Y(E_{ij})$ is given by $z^{E_{ij}}(s)$. 
\end{prop}
As we mentioned in Remark \ref{rm:id-base}, we identify $T_Y$ with $T_{Y}^{*}$. 

Recall that we choose $\sqrt{I}=G$. Let $S_{I}^{GS}=S^{GS}\times_{S}S_{I}$ where $S_I = \mathrm{Spec}(\mathbb{C}[P]/I)$. Recall, we have 
\begin{equation}
     f_{\rho_i} = f_{\rho_{il_{i}}}=  \prod_{s=1}^{l_i} (1+z^{\left[E_{is}\right]}X_{i}^{-1}) \mod G.    
\end{equation}
When we restrict $R_{I,\rho_{il_i}}$ to $S_{I}^{GS}$, it is the quotient by the ideal generated by 
\[
Z_{\rho_{il_i}}^{+}Z_{\rho_{il_i}}^{-}-t^{\kappa_{\rho_i}}\left(z^{-\sum_{k=1}^{l_i}[E_{ik}]}\cdot f_{\rho_{i}}\right).
\]
Here $\kappa_{\rho_i} = [\bar{D}_{i}]\cdot H\in\mathbb{N}$. 
Let $f_{\rho_{i}}^{'}=z^{-\sum_{k=1}^{l_i}[E_{ik}]}\cdot f_{\rho_{i}}$.
Notice that restricting to $S^{GS}$, we have 
\[
 \kappa_{\rho_{ij}} = \kappa_{\rho_i} = [\bar{D}_i]\cdot H 
\]
for any $j$ with $0 \leq j \leq l_i$.

Again, it suffices to consider the cases where the $o_{ij}$ are positioned as in Figure \ref{fig:focus} since computations for all other cases are exactly the same.  Recall that in Section 2 of \cite{RS19}, the period integral involves the \emph{complex Ronkin function} associated with the function $f'_{\rho_i}$  for a ray $\rho_i$ in $\Sigma_{(Y,D)}$. Let $x \in \rho_i \setminus {\mathcal{A}}$. Then, the complex Ronkin function at $x$ is defined by:
\[ \mathcal{R}(z^{-m_x}f'_{\rho_i},x):=\frac{1}{2\pi i}\int_{\mu^{-1}(x)}\frac{\log (z^{-m_{x}}f'_{\rho_i}(X_i))}{X_i}dX_{i}.\]
Here, $m_{x} \in \Lambda_{\rho_i}$ is the unique element such that the restriction $z^{-m_{x}}f'_{\rho_i}: \mu^{-1}(x) \rightarrow \mathbb{C}^*$ is contractible. More concretely, given a point $x_j \in \mathrm{Int}\rho_{ij}$, we have $-m_{x_j} = (l_i - j)\nu_i$ and 
\begin{equation}
     z^{-m_{x_j}}f'_{\rho_i} = z^{-\sum_{k=1}^{l_i}[E_{ik}]}\prod_{s=j+1}^{l_i}(X_s + z^{\left[E_{is}\right]}) \prod_{s'=1}^{j} (1+z^{\left[E_{is'}\right]}X_{i}^{-1}) \mod G.
\end{equation}


We compute the Ronkin function
$\mathcal{R}(z^{-m_{x_j}}f'_{\rho_{i}},x_j)$. Notice that $z^{-m_{x_j}}f'_{\rho_i} = z^{-\sum_{k=1}^{j}[E_{ik}]}f_{\rho_{ij}}$. Since
$f_{\rho_{ij}}$
already satisfies the normalized condition as defined in Proposition
1.6 of \cite{RS19}, we have
\begin{equation}\label{eq:slab-fcn}
\mathcal{R}(z^{-m_{x_j}}f_{\rho_{i}}^{'},x_j)=\log(z^{-\sum_{k=1}^{j}[E_{ik}]}).
\end{equation}

From now on, we make the assumption that all intersection points of a tropical 1-cycle with slabs are vertices. Suppose we have a vertex $v$ of $\beta_{\trop}$ such that it lies on a slab contained in the ray $\rho_i$. Following the orientation of the $\beta_{\trop}$, let $e$ and $e'$ be the adjacent edges to $v$ with $\beta_{\rm{trop}}$ oriented from $e$ to $e'$. Now let $\sigma$ be the maximal cell such that $e\in \sigma$ and let $\check{d}_e$ be the primitive generator of $\Lambda_{\rho}^{\perp}$ that evaluates positively on tangent vectors pointing form $\rho$ into $\sigma$. We weight the complex Ronkin function at $v$ by $\langle \check{d}_e,\xi_e\rangle \mathcal{R}(z^{-m_{v}}f_{\rho_{i}}^{'},v)$. The sum of these contributions from Ronkin functions is denoted by $\mathcal{R}(\beta_{\mathrm{trop}})$.

According to \cite{RS19}, the Picard-Lefschetz transformation also contributes to the period integrals. The piecewise linear function $\varphi$ gives rise to a cohomology class, $c_1(\varphi)\in H^1(B'\setminus{\Delta},\check{\Lambda})$. Thus, we can compute $\langle c_1(\varphi),\beta_{\trop}\rangle$. Given a vertex $v$ of $\beta_{\mathrm{trop}}$ lying a slab contained in the ray $\rho_i$, $v$ contributes $\langle \check{d}_e,\xi_e\rangle \cdot \kappa_{\rho_i}$ to $\langle c_1(\varphi),\beta_{\trop}\rangle$, where the notation is as above. The total pairing is obtained by summing over all such contributions from vertices of $\beta_{\trop}$ that lie on slabs. 

The results in \cite{RS19} involve extra data referred to as gluing data. We are in the case in which we work with trivial gluing data so this part of the formula does not affect us. 

Fix $a\in U_{sim}$. Let by $T_{a,I} = T_a \times_{S} S_I$.  Then, $T_{a,I}$ is isomorphic to $\mathrm{Spec}(\mathbb{C}[t]/(t^{k+1}))$ for some $k$. Let $\bar{\mathcal{X}_I}$ be the base change of $\bar{\mathcal{X}}$ to $S_I$. Then, $ \bar{\mathcal{X}}_I \mid_{T_{a,I}} /T_{a,I}$ is a flat 1-parameter family deformation of $X_{0,a}$. View $\Omega$ as a relative holomorphic volume form on $\mathcal{X}_I \mid_{T_{a,I}}$. Then, the following theorem follows from the main theorem of \cite{RS19}:
\begin{thm}[Theorem $1.7$, \cite{RS19}] \label{thm:per-int}
Let $\beta_{\trop}$ be a tropical 1-cycle on $B'$ not touching the boundary and
$\beta$ an associated singular 2-cycle on $X_{0,a}$. Then, 
\[ \exp\left(\frac{1}{2\pi i}\int_{\beta} \Omega_{} \right)=\exp(\mathcal{R}(\beta_{\trop}))\cdot t^{\langle c_1(\varphi),\beta_{\trop}\rangle}.\]
\end{thm} 
The above theorem implies that computations done on the central fiber of the toric degeneration are sufficient to give formulas for period integrals for the entire family as formal functions on the base. Since our mirror family is algebraic, the period integrals we get by applying the formulas in \cite{RS19} will be analytic. 

Slightly different from \cite{RS19}, we consider tropical 1-cycles that touch the boundary of $B'$. For a singular 2-cycle on $X_{a,0}$ associated with an exceptional tropical 1-cycle, we need to modify it to a singular 2-chain with boundary to integrate it against $\Omega$. See the next subsection for details.

\subsection{Periods for Exceptional Cycles} \label{ssec:per-excep}
In this subsection, we will first prove Proposition \ref{prop:mar-per-id}.

Given $s \in W \cap T_Y$, consider the image of the degenerate moment map of the corresponding central fiber. Consider the exceptional tropical 1-cycle $\beta_{\mathrm{trop},ij}$ that goes around the focus-focus singularity $o_{ij}$ with the labeling as in Figure \ref{exceptionaltrop}.
Following the computations done in Subsection 3.4, 3.5 and 3.6 in \cite{RS19}, we have
\begin{align*}
\frac{1}{2\pi i}\left(\int_{\beta_{e_{1}}+\beta_{e_{2}}}\Omega+\int_{\Gamma_{e_{2}}}\Omega\right)= & \left\langle d_{e_{1}},\xi_{e_{1}}\right\rangle \left(\kappa_{\rho_i}\log t+\mathcal{R}\left(z^{-m_{{v_2}}}f_{\rho_{i}}^{'},v_{2}\right)\right)\\
 & +\log z^{\xi_{e_{2}}}\left(S\left(v_{3}\right)\right)-\log z^{\xi_{e_{1}}}\left(S\left(v_{1}\right)\right)
\end{align*}
\begin{align*}
\frac{1}{2\pi i}\left(\int_{\beta_{e_{2}^{'}}+\beta_{e_{3}}}\Omega+\int_{\Gamma_{e_{3}}}\Omega\right)= & \left\langle d_{e_{2}^{'}},\xi_{e_{2}^{'}}\right\rangle \left(\kappa_{\rho_i}\log t+\mathcal{R}\left(z^{-m_{{v_4}}}f_{\rho_{i}}^{'},v_{4}\right)\right)\\
 & +\log z^{\xi_{e_{3}}}\left(S\left(v_{1}\right)\right)-\log z^{\xi_{e_{2}^{'}}}\left(S\left(v_{3}\right)\right).
\end{align*}
Here, $\Gamma_{e_2}$ and $\Gamma_{e_3}$ are the slab add-ins as in subsection 3.4 of \cite{RS19}. Also, we have $\left\langle d_{e_{1}},\xi_{e_{1}}\right\rangle =1$ since $\xi_{e_1}$ is a primitive generator of $\Lambda/\Lambda_{\rho_i}$ and points into the maximal cell containing $e_1$. Similarly, we have  $\left\langle d_{e_{2}^{'}},\xi_{e_{2}^{'}}\right\rangle =-1$. 
Since $z^{\xi_{e_{2}}}=z^{\xi_{e_{2}^{'}}}$, we have 
\begin{align*}
 & \frac{1}{2\pi i}\left(\int_{\beta_{e_{1}}+\beta_{e_{2}}}\Omega+\int_{\Gamma_{e_{2}}}\Omega+\int_{\beta_{e_{2}^{'}}+\beta_{e_{3}}}\Omega+\int_{\Gamma_{e_{3}}}\Omega\right)\\
= & \mathcal{R}\left(z^{-m_{{v_2}}}f_{\rho_{i}}^{'},v_{2}\right)-\mathcal{R}\left(z^{-m_{{v_4}}}f_{\rho_{i}}^{'},v_{4}\right)  +\log z^{\xi_{e_{3}}}\left(S\left(v_{1}\right)\right)-\log z^{\xi_{e_1}}\left(S\left(v_{1}\right)\right).
\end{align*}
It only remains to specify the contribution of the Ronkin function. By Equation \ref{eq:slab-fcn}, we get that
\begin{equation}\label{eq:localfocuscomputation}
\mathcal{R}\left(z^{-m_{{v_2}}}f_{\rho_{i}}^{'},v_{2}\right)-\mathcal{R}\left(z^{-m_{{v_4}}}f_{\rho_{i}}^{'},v_{4}\right)  =z^{\left[E_{ij}\right]}.
\end{equation}
We emphasize that the above formula is true regardless of the configuration of positions of the focus-focus singularities on the ray $\rho_i$. At the trivalent vertex $v_1$, the chain $\Gamma_{v_1}$, as described in Subsection \ref{ssec:sing-trop}, is non-empty. By Lemma 3.3 of \cite{RS19}, we have
\begin{equation}
    \frac{1}{2 \pi i} \int_{\Gamma_{v_1}} \Omega = \pi i \mod 2 \pi i.
\end{equation}
Let  $\beta_1 = \beta_{e_1} + \beta_{e_2} + \Gamma_{e_2} + \beta_{e_{2}{'}} +\beta_{e_3} + \Gamma_{e_3} + \Gamma_{v_1}$. We conclude that 
\begin{align}\label{eq:q}
    \exp \left( \frac{1}{2\pi i}\int_{\beta_{1}}\Omega  \right)   = -z^{[E_{ij}]} \cdot \frac{z^{\xi_{e_{3}}}\left(S\left(v_{1}\right)\right)}{z^{\xi_{e_{1}}}\left(S\left(v_{1}\right)\right)}
 = -\frac{z^{[E_{ij}]}}{z^{\xi_{e_{0}}}\left(S\left(v_{1}\right)\right)}.
\end{align}
The second equality in the above follows from the balancing condition $\xi_{e_{3}}+\xi_{e_{0}}=\xi_{e_{1}}$ at $v_1$. Since $\xi_{e_{0}}$ is parallel to $\rho_{i}$ and $F_{i}$ is parallel
to $\rho_{i}$, $\beta_{e_{0}}$ is closed at $S(v_{0})$.

Let $\beta_{ij,s}$ be the deformation of the singular 2-cycle $\beta_{1}+ \beta_{e_0}$ on $\bar{\mathcal{X}}_s$ as in Construction \ref{def:singular-2-cycles}. The holomorphic form $\Omega$ has poles along $\mathcal{D}$. In the following construction, after fixing a marking of $\mathcal{D}$, we will replace $\beta_{ij,s}$ with a singular 2-chain with boundary in $\bar{\mathcal{X}}_s$ whose period will give the marked period of $E_{ij,s}$.

\begin{cons} \label{def:bou-chain} We follow the notation in Definition \ref{def:excep-1-cycle}.  For convenience, we can assume that for a fixed $i$, all the $\beta_{\mathrm{trop}, ij}$ start from the same 1-cell $\varrho \subset F_i$. In particular, if $F$ is a parallel configuration, then $\varrho= R_i$. Let $q \succ q'$ be two adjacent lattice points on $\varrho$. Let $D_{\varrho}$ be the boundary component corresponding to $\varrho$ on the central fiber. Notice that on $D_{\varrho}$, the boundary component corresponding to $\varrho$ on the central fiber, $\frac{\theta_q}{\theta_{q'}}$ restricts to the toric monomial $z^{\xi_{e_0}}$. 

By Lemma \ref{extendsection}, $\frac{\theta_q}{\theta_{q'}} =-1$ gives a section $p'_i: \cD_i \rightarrow \Spec(\mathbb{C}[\NE(Y)])$. If $F$ is a parallel configuration, then $p'_i$ agrees with the canonical marking $p_i$ in Construction \ref{cons:marking}. If $F$ is not a parallel configuration, then $p'_i$ and $p_i$ might differ. The difference is given by the contribution from bending parameters of rays in $\Sigma$ that will intersect with the forward ray of $\varrho$, as in Construction \ref{cons:marking}.

Let $p_0$ be the point on $D_{\varrho}$ such that $\frac{\theta_q}{\theta_{q'}}(p_0) = z^{\xi_{e_0}}(p_0)=-1$. 
Let $\gamma$ be an arc in $D_{\varrho}$  connecting $p_0$ and $S(v_0)$. We orient $\gamma$ from $p_0$ to $S(v_0)$.
Let $N_{\gamma}$ be the tube over $\gamma$ in a small tubular neighborhood of $D_{\varrho}$. We delete a small disc in  $\beta_{e_{0}}$ around $S(v_0)$ and glue it to $N_{\gamma}$ to get a new 
chain $\tilde{\beta}_{e_{0}}$. Thus, on the central fiber, we modify the singular 2-cycle $\beta_1 + \beta_{e_0}$ to $\beta_1 + \tilde{\beta}_{e_0}$, a singular 2-chain with boundary. 

On $\bar{\mathcal{X}}_s$, by deforming $\tilde{\beta}_{e_0}$,
we modify $\beta_{ij,s}$ to a singular 2-chain $\beta'_{ij,s}$ whose boundary is a circle over $p'_i(s)$ contained in a small tubular neighborhood of $\cD_{i,s}$. Let $\gamma_{s}$ be an arc in $\cD_{i,s}$ connecting $p_i(s)$ and $p'_i(s)$. Let $N_{\gamma_{s}}$ be a tubular neighborhood of $\gamma_{s}$ in $\cD_{i,s}$. Gluing $N_{\gamma_{s}}$ to $\beta'_{ij,s}$, we get a singular 2-chain $\tilde{\beta}_{ij,s}$ whose boundary is a circle over $p_i(s)$ contained in a small tubular neighborhood of $\cD_{i,s}$.
\end{cons}

\begin{remark}
$\tilde{\beta}_{e_0}$ and $\tilde{\beta}_{ij,s}$ are only well-defined up to multiples of classes of real 2-tori whose integrals against $\Omega$ are integral multiples of $(2\pi i)^2$. After we take the exponential map, this ambiguity disappeared. 
\end{remark}

We prove the following lemma:
\begin{lem}\label{lem:experiod} 
Using the canonical marking of the boundary defined in Construction \ref{cons:marking}, we have
\begin{equation} \label{eq:r}
    \exp{\left( \frac{1}{2\pi i}\int_{\tilde{\beta}_{ij,s}}\Omega \right)}=   z^{[E_{ij}]}(s).
\end{equation} \label{eq:a}
\end{lem}

\begin{proof} 
If $F$ is a parallel configuration, then $p_i(s)=p'_i(s)$. Moreover, since the tropical 1-cycle $\beta_{\mathrm{trop},ij}$ is contained in a single maximal cell  $F \cap \sigma_{i,i-1}$ and does not intersect with any ray in $\Sigma$, there is no contribution from $\langle c_1(\varphi),\,\beta_{\mathrm{trop},ij}\rangle$. By the standard Stokes' Theorem and Equation \ref{eq:q},
\begin{align*}
& \exp \left( \frac{1}{2\pi i}\int_{\beta_{1} + \tilde{\beta}_{e_{0}}}\Omega \right)
= - \frac{z^{[E_{ij}]} }{z^{\xi_{e_{0}}}\left(S\left(v_{1}\right)\right)} \cdot \frac{z^{\xi_{e_{0}}}\left(S\left(v_{1}\right)\right) }{z^{\xi_{e_{0}}}\left( p_0 \right)} 
=  - \frac{  z^{[E_{ij}]} }{z^{\xi_{e_{0}}}\left( p_0 \right)} 
=  z^{[E_{ij}]}.
\end{align*}
By Theorem 1.7 and Corollary 4.6 of \cite{RS19}, we obtain Equation \ref{eq:r}.

If $F$ is not a parallel configuration, then $p_i(s)$ and $p'_i(s)$ might differ.  We follow notation in Construction \ref{def:bou-chain}. Let $f'_{\cD_i} = \frac{\theta_q}{\theta_{q'}}$. It suffices to consider the case where $\beta_{\mathrm{trop},ij}$ crosses a single ray $\rho_{i_0}$. The general cases follow from exactly the same computation.  Let $\delta_{i_0} = 
\left\langle n_{\rho_{i_0}} \,,\,\overrightarrow{q' q} \right\rangle$ where $n_{\rho_{i_0}}\in \check{\Lambda}_{\rho_{i_0}}$ is the unique primitive element annihilating $\rho_{i_0}$ and is positive on the parallel transport of $\overrightarrow{q' q}$.
By our assumption that  $\beta_{\mathrm{trop},ij}$ will cross the slab $\rho_{i_0l_{i_0}}$, together with the period formula and Equation \ref{eq:slab-fcn}, the contribution to the period when  $\beta_{\mathrm{trop},ij}$ crosses $\rho_{i_0}$ is 
\[
\left(t^{\kappa_{\rho_{i_0}}} \cdot \mathcal{R}(z^{-m_{x}}f_{\rho_{i_0}}^{'},x) \right)^{-\delta_{i_0}} = \left(   t^{\kappa_{\rho_{i_0}}} \cdot  z^{-\sum_{k=1}^{l_{i_0}}[E_{i_0k}]} \right)^{-\delta_{i_0}}
\]
where $x$ is the crossing point of $\beta_{\mathrm{trop},ij}$ on $\rho_{i_0 l_{i_0}}$. Recall that $f'_{\cD_i}$ is the canonical coordinate on $\cD_{i}$ that was defined in Construction \ref{cons:marking}. By Equation \ref{eq:diff-mar}, restricted to $S^{GS}$, we have
\[
f_{\cD_i} = \left(  t^{\kappa_{\rho_{i_0}}} \cdot  z^{-\sum_{k=1}^{l_{i_0}}[E_{i_0k}]} \right)^{\delta_{i_0}} \cdot f'_{\cD_i}
\]
In particular, we have \[
f_{\cD_i}(p_i(s)) = -1 = \left(  t^{\kappa_{\rho_{i_0}}} \cdot  z^{-\sum_{k=1}^{l_{i_0}}[E_{i_0k}]}(s) \right)^{\delta_{i_0}} f'_{\cD_i}(p_i(s))
\]
It follows from Stokes' Theorem that 
\begin{equation*}
    \exp{\left( \frac{1}{2\pi i}\int_{N_{\gamma_s}} \Omega \right)} = \left( t^{\kappa_{\rho_{i_0}}} \cdot  z^{-\sum_{k=1}^{l_{i_0}}[E_{i_0k}]}(s) \right)^{\delta_{i_0}}
\end{equation*}
Hence, the contribution to the period when $\beta_{\mathrm{trop},ij}$ crosses a ray in $\Sigma$ 
will be canceled with the difference between the marking $p_i$ and $p'_i$. Therefore, we still obtain Equation \ref{eq:r}.
\end{proof}

Now, we give the proof of Proposition \ref{prop:mar-per-id}:
\begin{proof}[Proof of Proposition \ref{prop:mar-per-id}] 
Recall that for $s\in W\cap T_Y$, $E_{ij,s}$ is the exceptional curve on $\bar{\mathcal{X}}_s$ corresponding to the tropical 1-cycle $\beta_{\mathrm{trop},ij}$. Denote the point where the  curve $E_{ij,s}$ touches the boundary divisor $D_i$ by $q_{ij}(s) \in D_i$. We repeat what we have done in Construction \ref{def:cons-excep}. Let $\gamma^{'}$ be an arc in $D_i$ connecting $q_{ij}(s)$ and $p_{i}(s)$ oriented from $p_{i}(s)$ to $q_{ij}(s)$. Let $N_{\gamma^{'}}$ be a tubular neighborhood of $\gamma^{'}$ in $\mathcal{X}_s$ whose boundary circle over $p_i(s)$ is also the boundary of $\tilde{\beta}_{ij,s}$ . Then, by deleting a disc in $E_{ij,s}$ and glue it to $N_{\gamma^{'}}$, we get a new chain $\tilde{\beta}'_{ij,s}$ with the same boundary as $\tilde{\beta}_{ij,s}$. 
By Lemma \ref{lem:per-excep}, $\exp \left({\frac{1}{2\pi i}\int_{\tilde{\beta}'_{ij,s}}} \Omega \right)$ is equal to the marked period of $\cO_{\bar{\cX}_s}(E_{ij,s})$, which is marked by $\cO_{Y}(E_{ij})$. 
Now, we want to show that 
\begin{align}\label{eq:eq-chain}
    \exp \left({\frac{1}{2\pi i}\int_{\tilde{\beta}_{ij,s}'}} \Omega \right) = \exp \left({\frac{1}{2\pi i}\int_{\tilde{\beta}_{ij,s}}} \Omega \right). 
\end{align}
With the equality, it would then follow immediately from Lemma \ref{lem:experiod} that the marked period of $E_{ij,s}$ is equal to $z^{[E_{ij}]}(s)$ if we use the canonical marking given in Section \ref{sec:marking-boundary}. 
Observe that $[\tilde{\beta}_{ij,s}'] - [\tilde{\beta}_{ij,s}]$ is homologically trivial in $H_2\left( \bar{\mathcal{X}}_s, \mathbb{Z}\right)$. Thus, in $H_2\left(\mathcal{X}_s, \mathbb{Z}\right)$, 
\[
[\tilde{\beta}_{ij,s}'] - [\tilde{\beta}_{ij,s}] \in \mathrm{ker} \left(H_2\left(\mathcal{X}_s, \mathbb{Z}\right) \rightarrow H_2\left( \bar{\mathcal{X}}_s, \mathbb{Z}\right) \right).
\]
We know that $\mathrm{ker} \left(H_{2}\left(\mathcal{X}_s, \mathbb{Z}\right) \rightarrow H_{2}\left( \bar{\mathcal{X}}_s, \mathbb{Z}\right) \right)$ is generated by $\Gamma_s$, the generator of the kernel of the exact sequence \ref{eq:exa}. By our scaling of $\Omega$, we conclude that 
\[
\frac{1}{2\pi i} \left(\int_{\tilde{\beta}_{ij,s}'}\Omega - \int_{\tilde{\beta}_{ij,s}}\Omega \right)\in 2\pi i \mathbb{Z}
\]
After exponentiation, we obtain the desired equality \ref{eq:eq-chain}.
\end{proof}

The kernel of the weight map $w:\Pic(Y)^{*} \rightarrow \chi(T^{D})$ via the identification $\Pic(Y)^{*} \simeq \Pic(Y)$ is identified as $D^\perp$. Therefore, $T_{Y}/T^{D}$ is identified as $T_{(D^\perp)^{*}}$, the period domain in Definition \ref{defn:unmarkedperiod}.
\begin{prop}\label{identity} 
Let $((\bar{\cX}\mid_{T_Y},\cD\mid_{T_Y}),p_i,\mu)\rightarrow T_Y$ be the marked compactified mirror family over the base torus $T_Y$. Then, the induced marked period map is the identity if $(Y,D)$ does not have a  toric blowdown to the del Pezzo surface of degree 1 and an isomorphism otherwise. In either case, given $t \in T_Y$, the period point of $\bar{\cX}_t$ is given by the image of $t$ under the surjective homomorphism $T_Y \twoheadrightarrow T_{(D^\perp)^{*}}$.
\end{prop}

\begin{proof} It suffices to consider the case where $(Y,D)$ has a toric model. 

The marking of the Picard group $\mu$ of the compactified mirror family over $T_Y$ is done in Proposition \ref{prop:mar-family}. Let $U \subset T_{Y}$ be the analytic open neighborhood as in Lemma \ref{lem:def-type}. Recall that we define $U$ to be $T^{D} \cdot (W \cap T_{Y})$. By Lemma \ref{lem:rel-tor-mar} and the weight function, we know that both the marked period of $\cO_{\bar{\cX}_s}(E_{ij,s})$ and $z^{E_{ij}}$ is an eigenfunction of $T^{D}$ with weight $e_{D_i}$. By Proposition \ref{prop:mar-per-id}, we know they coincide on $W \cap T_Y$. Therefore, they must also agree on $U$ by the equivariance of the action of $T^{D}$ on the family. By the uniqueness of analytic continuation, they agree on the entire base $T_{Y}$. 

Suppose $(Y,D)$ is not a toric blowup of the del Pezzo surface of degree $1$. Then,  Proposition \ref{prop:boundaryperiod} together with computations for the marked periods of exceptional curves show that the marked period map is the identity. If there is a  toric blowdown of $(Y,D)$ to the del Pezzo surface of degree $1$, we do not know if the marked period map is an identity, but it is still an isomorphism by Lemma \ref{lem:tor-bl-compa}, Lemma \ref{lem:eigen-wei} and Proposition \ref{prop:boundaryperiod}. The statement about period points in this special case follows from Proposition \ref{prop:del-pezzo-1} in Appendix \ref{delpezzo1}.
\end{proof}

We can now prove the main statement, given in Theorem \ref{main theorem}.
\begin{proof}[Proof of Theorem~\ref{main theorem}]
The compactification of the family comes from the choice of a Weil divisor $W$ such that $W\cdot D_i>0$ for all $i$ and the associated polytope. By Proposition \ref{weilnomatter}, this does not depend on the choice of Weil divisor, giving rise to a canonical compactification. The marking of the boundary is given by Construction \ref{cons:marking}. The deformation type is determined by the proof of Proposition \ref{prop:mar-family}. Finally, the induced marked period map is an isomorphism by Proposition \ref{identity}, identifying the marked compactified mirror family with the universal family of generalized marked pairs deformation equivalent to $(Y,D)$. 
\end{proof}

\begin{remark}
Recall that the set of roots $\Phi$ is finite for a positive pair $(Y,D)$. In the positive cases, the finite union of hypertori
\[
T_{\alpha} = \{\varphi \in T_Y \mid \varphi(\alpha)=1\}
\]
for $\alpha \in \Phi$ parametrizes generalized marked pairs with du Val singularities deformation equivalent to $(Y,D)$ in the mirror family, since only minimal resolutions of generalized marked pairs with du Val singularities have internal (-2) curves. 
\end{remark}



 

\begin{section}{Connections to the work of Keel-Yu}\label{sec:keel-yu}

For a variety $X$ over a non-archimedean field $K$, we have an associated $K$-analytic space $X^{\text{an}}$ in the sense of Berkovich \cite{Ber90}. In our setting, the base field $k$ is equipped with the trivial absolute value.

\begin{defn}
    Let $U$ be a connected smooth affine log Calabi-Yau variety over $k$ together with its  canonical volume form $\omega$ (which is unique up to scaling). Define
\begin{align*}
     U^{\mathrm{trop}}(\Z)& :=\{0\} \cup \{
     {\text{divisorial discrete valuations}\, \nu:\,k(U)\setminus \{0\} \rightarrow \Z \mid \nu (\omega) <0 } \}\\
    & := \{0\} \cup  \{(E,m) \mid m \in \Z_{>0}, E \subset (Y\setminus U),\,\omega\, \text{has a pole along}\,E \}
\end{align*}
Here, $k(U)$ is the field of rational functions. In the second expression, $E$ is a divisorial irreducible component of the boundary in some partial compactification $U \subset Y$, and two divisors on two possibly different birational varieties are identified if they give the same valuation on their common field of fractions. 
\end{defn}

The set $U^{\mathrm{trop}}(\Z)$ is another name for the set of integer points $\SK(U,\Z)$ on the Kontsevich-Soibelman essential skeleton $\SK(U)$. 

Suppose $U$  contains a Zariski-open algebraic torus. Let $U \subset Y$ be a simple normal crossing compactification. In \cite{KY}, the authors construct a commutative associative algebra $A$ equipped with a compatible multilinear form. The mirror algebra $A$ is a free $\mathbb{Z}[\NE(Y)]$-module with basis $\SK(U,\Z)$, the integer points on the essential skeleton $\SK(U)$ of $U$. The structure constants of $A$ as an algebra are given by naive counts of non-archimedean analytic discs with the given intrinsic skeleton,  realized as counts associated with trees in $\SK(U)$. 

Given $P_1,\cdots, P_n \in \SK(U,\Z)$, we write the product in the mirror algebra $A$ as
\begin{align}
    \theta_{P_1}  \cdots \theta_{P_n} =\sum_{Q \in \SK(U,\Z)} \sum_{\gamma \in \NE(Y)} \chi(P_1, \cdots,P_n,Q,\gamma) z^{\gamma}\theta_Q 
\end{align} 
where $\gamma \in \NE(Y)$ is a curve class.

 Here, we reproduce the heuristic behind the structure constants as in Remark $1.7$ of \cite{KY} and refer the readers to Section $1.1$ of \cite{KY} for the technical definitions of these constants. Fix $P_1,...,P_n,Q\in \SK(U,\Z)$ and $\gamma\in \NE(Y)$. Write $P_j=m_j v_j$. By passing to a blow-up $\pi:(\tilde{Y},\tilde{D})\rightarrow (Y,D)$, we can always assume that $v_j$ has divisorial center $D_j\subset\tilde{D}$. Suppose $Q\neq 0$. Then heuristically,  $\chi(P_1,...,P_n,Q,\gamma)$ counts   analytic maps
\[ g:[\mathbb{D},(p_1,...,p_n)]\rightarrow \tilde{Y}^{\operatorname{an}}\]
such that
\begin{enumerate}
    \item $\mathbb{D}$ is a closed unit disk over a non-archimedean field extension $k'/k$ and $p_1,...,p_n$ are distinguished $k'$-points on $\mathbb{D}$,
    \item $g(p_j)\in D_j^{\circ}$ for all $j$ such that $P_j\ne 0$,
    \item $g^{-1}(\tilde{D})=\sum m_j p_j$,
    \item $g(\partial \mathbb{D})=Q\in U^{\operatorname{an}}$,
    \item The derivative of $g$ at $\partial \mathbb{D}$ is $Q$,
    \item $(\pi \circ g)_{*}=\gamma$.
\end{enumerate}


Now, let us narrow the construction of \cite{KY} to the dimension $2$ case. In this context, every smooth affine log Calabi-Yau surface $U$ with maximal boundary appears as the interior of a positive Looijenga pair $(Y,D)$. Moreover, $\SK(U,\mathbb{Z})$ is equal to $B(\Z)$. The following proposition from \cite{HKY} gives a geometric interpretation of the structure coefficients of the mirror algebra constructed in \cite{GHK} via the construction of \cite{KY}:
\begin{prop}[Proposition 4.10, \cite{HKY}]\label{prop:compar-ky}
Let $U$ be a smooth affine log Calabi-Yau surface with maximal boundary. Fix a Looijenga pair $(Y,D)$ that compactifies $U$. Then, (after base change to $\mathbb{C}[\NE(Y)]$) the mirror algebra constructed in \cite{GHK} associated with $(Y,D)$ agrees with that constructed in \cite{KY} using the compactification $U\subset Y$.
\end{prop}

Combining Theorem \ref{main theorem} and Proposition \ref{prop:compar-ky}, we obtain the following theorem of the paper which is a combination of Theorem \ref{thm:mai-main-intro} and Theorem \ref{thm:main-main-intro-2}:
\begin{thm}\label{thm:main-main-thm}
Let  $U$ be a smooth affine log Calabi-Yau surface with maximal boundary. Fix a Looijenga pair $(Y,D)$ that compactifies $U$ and a marking of $D$, then 
there is a  decomposition 
\begin{align}\label{eq:rs}
    H^0(U,\cO_U)=\bigoplus_{q\in U^{\operatorname{trop}}(\Z)} \mathbb{C}\cdot \theta_q 
\end{align}
of the vector space of regular functions on $U$  into one-dimensional subspaces with distinguished basis elements of theta functions. The multiplication rule of theta functions is given by the geometric construction in \cite{KY}. Different choices of compactifications and markings will only result in the rescaling of theta functions and not affect the  vector space decomposition. The decomposition is canonical up to the choice of one of the two possible orientations of $U^{\trop}(\Z)$.  
\end{thm}

\begin{proof}
Fix a Looijenga pair $(Y,D)$ that compactifies $U$ and a choice of marking of $D$. Let $t' \in T_Y$ be the corresponding marked period point. Let $\phi: T_Y \rightarrow T_Y$ be the marked period map of the compactified mirror family. Recall that by Proposition \ref{identity}, the isomorphism $\phi$ is the identity except for the case of degree 1 del Pezzo surfaces. Let $t =\phi^{-1}(t')$. Then, $(\bar{\cX}_t, \cD_t)$ is isomorphic to $(Y,D)$ as a marked pair and in particular, $\cX_t$ is isomorphic to $U$. By evaluating the structure coefficients of $A$ at $t$, i.e., evaluate $z^{[C]}$ at $t$ for $[C] \in \NE(Y)$, we get a basis for $H^0(\cX_t,\cO_{\cX_t})$ given by theta functions parametrized by $U^{\trop}(\Z)$. By Proposition \ref{prop:compar-ky}, the multiplication rule for theta functions in \cite{GHK} is equivalently given by the geometric construction of  \cite{KY}. By Lemma \ref{lem:rel-tor-mar}, changing the marking of $\cD_t$ corresponds to the action of the relative torus $T^{D}$.  Since theta functions are eigenfunctions of the action of $T^{D}$ by Theorem \ref{thm:rel-equi}, different choices of marking of $D$ will only rescale the basis elements of $H^0(\cX_t,\cO_{\cX_t})$.

Change the marking of $\Pic(\bar{\cX}_t)$ by $\Pic(Y)$ is equivalent to pre-composing with a lattice automorphism of $\Pic(Y)$ that preserves $D_i$ for each $i$ and the cone $C^{++}$ (cf. Definition \ref{def:cones}). Since the canonical scattering diagram in \cite{GHK} is invariant under such an automorphism, changing the marking of $\Pic(\bar{\cX}_t)$ will not change the algebra we get. 

If we use another Looijenga pair $(Y',D')$ to compactify $U$, we can always find a common toric blowup $(Y'',D'')$ of both $(Y,D)$ and $(Y',D')$.  By Lemma \ref{lem:tor-bl-compa} and Lemma \ref{lem:compa-para-tor-bl}, toric blowups of fibers in $(\bar{\cX}\mid_{T_Y},\cD\mid_{T_Y})$,$(\bar{\cX}'\mid_{T_Y},\cD'\mid_{T_Y})$ can be identified as fibers in $(\bar{\cX}''\mid_{T_Y},\cD''\mid_{T_Y})$. Since markings of the compactified mirror families are compatible with the toric blowup operation, we conclude that different compactifications will again result only in the rescaling of the basis elements of theta functions.  

It remains to identify $\cX^{\mathrm{trop}}_{t}(\Z)$ with $U^{\trop}(\Z)$. We recall the identification of $\cX^{\mathrm{trop}}_{t}(\Z)$ with $U^{\trop}(\Z)$ as described at the beginning of Section 4 of \cite{Man16}. Given a primitive element $q$ in $U^{\trop}(\Z)$, let $L$ be an immersed line that escapes to infinity parallel to $q$. Then, the generalized half space $Z(L)$ gives a partial compactification of $\cX_t$ with an irreducible boundary divisor. The compactification is independent of the distance of $L$ to $0\in B$ and therefore depends only on $q$. Thus, we can define an identification of $U^{\trop}(\Z)$ with $\cX^{\mathrm{trop}}_{t}(\Z)$  such that  given an element $mq \in U^{\trop}(\Z)$ where $m\in \Z_{>0}$ and $q$ a primitive element, $mq$ is sent to the pair $(D_q,m)$ where $D_q$ is the irreducible divisor of the partial compactification of $\cX_t$ determined by $q$ that we just described.  This identification depends only on the convention of the orientation of immersed lines on $B$ that we choose (in this paper, we choose the orientation so that $0$ is always on the left of an immersed line). Notice that fixing the orientation of immersed lines is equivalent to fixing the orientation of $U^{\trop}(\Z)$.  The decomposition \ref{eq:rs} is canonical up to this choice of the two possible orientations of $U^{\trop}(\Z)$. 
\end{proof}

\begin{remark}
In the mirror symmetry program for affine log Calabi-Yau varieties (with maximal boundary), it is expected that $H^0(U,\cO_U)$ has a vector space basis parametrized by the integer tropical  points of its mirror and the structure constants of $H^0(U,\cO_U)$, as an algebra, should be given by the enumerative geometry of its mirror. Here, the surprising fact is that for a smooth affine log Calabi-Yau surface with maximal boundary, it appears as a fiber in its own mirror family. However, since dual Lagrangian torus fibrations in dimension 2 are topologically equivalent, this is consistent with the SYZ mirror symmetry picture. 

\end{remark}

\end{section}

\appendix
\section{Degree $1$ del Pezzo}\label{delpezzo1}
We now cover the exceptional case of the degree $1$ del Pezzo. This is the pair $(Y,D)$ where $Y=\operatorname{Bl}_8(\mathbb{P}^2)$ and $D$ is an irreducible nodal curve whose self-intersection number is $1$. 

We construct tropical cycles that correspond to $(-2)$-classes that have trivial intersections with the boundary and then compute the periods. For pictorial convenience, we work on a toric blowup of $(Y,D)$ instead. First, we find a toric blowup of $(Y,D)$ that has a toric model. 
\begin{lem}
There is a toric blowup $(Y',D')\rightarrow (Y,D)$ such that $(Y',D')$ has a toric model given by $(F_3,\bar{D})$, where $F_3$ is the Hirzebruch surface and $\bar{D}$ is its toric boundary.
\end{lem}
\begin{proof}
First, blow up the unique node on $D$. This gives rise to an exceptional divisor, $D_3$. There are two tangent lines to the node in $D$, call these $L_1$ and $L_2$. The strict transform of $L_1$ intersects $D_3$ and the strict transform $D_1$ of $D$. We blow up at this intersection point to get a new divisor, $D_2$. Now the strict transform of $L_1$ only intersects $D_2$. We do the same thing with $L_2$, which gives a new boundary divisor, $D_4$. We then get a new pair $(Y',D')$ where $D'$ has four irreducible components, $D_1,D_2,D_3,D_4$ with (respective) self intersections $(-5,-1,-3,-1)$. Now, blowdown the strict transforms of $L_i$ and also the strict transforms of the original eight exceptional curves on $D$ to get a new pair, $(\bar{Y},\bar{D})$ which has boundary intersections $(3,0,-3,0)$. This is the Hirzebruch Surface, $F_3$. So we have taken a toric blowup of the original pair such that it has a toric model.
\end{proof}

Now consider the GS locus corresponding to the toric model $p:(Y',D')\rightarrow (\bar{Y},\bar{D})$. We want to find tropical curves on $B_{(Y',D')}$ that correspond to a basis of $D^{\perp} \simeq D'^{\perp}$. First, we compute what $D^{\perp}$ is. Let $\bar{C}$ be the class corresponding to the unique curve of self-intersection $-3$ of $F_3$ and $\bar{F}$ a fiber of $F_3$ (under its standard ruling). Also, denote their respective pullbacks under $p$ by $C$ and $F$. Then we have 
\[\Pic(\bar{Y})=\langle \bar{C}, \bar{F}\rangle\]
\[ \Pic(Y')=\langle C,F, E_1,...,E_{10}\rangle,\]
Where, $E_1,...,E_{8}$ are the exceptional curves on $D_1$, $E_9$ is the exceptional curve on $D_2$ and $E_{10}$ is the exceptional curve on $D_4$. Denote the corresponding toric boundary components by $\bar{D}_i$. Then:
\[ D_1=p^{*}(\bar{D}_1)-E_1-...-E_8=3F+C-E_1-...-E_8\]
\[ D_2=F-E_9, \hspace{1cm}D_3=C,\hspace{1cm}D_4=F-E_{10}. \]
Because $(Y,D)$ is positive, $D^{\perp}$ will be generated by roots. A quick computation shows that a basis for $D^{\perp}$ can be given by the following $(-2)$-classes:
\[ D^{\perp}=\langle E_1-E_2,...,E_7-E_8,3F+C-E_1-E_2-E_3-E_9-E_{10}\rangle.\]

First, we produce tropical cycles corresponding to each $(-2)$-curve. The tropical cycle $\beta^{i}_{\trop}$ that first winds around the focus-focus singularity $o_{1(i+1)}$ and then winds around $o_{1i}$, as depicted in Figure \ref{fig:root1}, corresponds to $E_i-E_{i+1}$ and $\beta'_{\trop}$ in Figure \ref{fig:root2} corresponds to $3F+C-E_1-E_2-E_3-E_9-E_{10}$. The vectors on each edge are given by Lemma \ref{focusfocusvector}.  \\

\begin{figure}
\begin{minipage}{\linewidth} \begin{center}  \begin{tikzpicture}[circ/.style={shape=circle, inner sep=2pt, draw, node contents=}] 
\draw  node (x) at (-2,0) [circ]; 
\draw  node (y) at (2,0) [circ]; 
\draw (x) -- (y); \draw [dashed] (x) -- (-4.5,0); \draw [dashed] (y) -- (4.5,0);
\draw [ decoration={markings,  mark=at position 0.13 with {\arrow{<}}, mark=at position 0.37 with {\arrow{<}} }, postaction={decorate}](y) circle [radius=1cm];          
\draw [ decoration={markings,  mark=at position 0.13 with {\arrow{>}}, mark=at position 0.37 with {\arrow{>}} },         postaction={decorate}](x) circle [radius=1cm];          
\draw  [decoration={markings, mark=at position 0.5 with {\arrow{>}}}, postaction={decorate}](-2,1) .. coordinate [pos=.5] (a) controls (-1,2) and  (1,2) .. (2,1 );
\fill [black] ( [xshift=1cm] y) circle (2pt); 
\fill [black] ( [xshift=-1cm] y) circle (2pt); 
\fill [black] ( [xshift=1cm] x) circle (2pt); 
\fill [black] ( [xshift=-1cm] x) circle (2pt); 
\fill [black] (2,1) circle (2pt); 
\fill [black] (-2,1) circle (2pt);
\node [label=below:$v_{1}$] at (2,1.1){}; 
\node [label=right:$v_{2}$] at ([xshift=0.9cm,yshift=0.2cm] y){}; 
\node [label=left:$v_{3}$] at ([xshift=-0.9cm,yshift=0.2cm] y){}; 
\node [label=below:$C_{2}$] at (2,-1){};
\node [label=below:$w_{1}$] at (-2,1.1){};
\node [label=right:$w_{2}$] at ([xshift=0.9cm,yshift=0.2cm] x){};
\node [label=left:$w_{3}$] at ([xshift=-0.9cm,yshift=0.2cm] x){}; 
\node [label=below:$C_{1}$] at (-2,-1){};
\node [label=above:$e$] at (0,1.8){}; 
\node [label=below:$o_{1(i+1)}$] at (-2,0.1){}; 
\node [label=below:$o_{1i}$] at (2,0.1){};
\draw [->] ([xshift=2cm] 150:1cm)   -- ( [xshift=2cm, yshift=0.5cm] 150:1cm); \draw [->] ([xshift=2cm] 60:1cm)   -- ( [xshift=1.5cm, yshift=0.5cm] 60:1cm);
\node [label=right:$d_{\rho}+\nu_{1}$] at ( [xshift=1.6cm, yshift=0.3cm] 60:1cm){}; 
\node [label=above:$d_{\rho}$] at ( [xshift=2cm, yshift=0.3cm] 150:1cm){}; 
\draw [<-] (-1.5,1.6) -- (-1,1.6);
\node [label=above:$\nu_{1}$] at (-1.5,1.5){};
\draw [->] ([xshift=-2cm] 120:1cm)   -- ( [xshift=-1.5cm, yshift=0.5cm] 120:1cm); 
\draw [->] ([xshift=-2cm] 30:1cm)   -- ( [xshift=-2cm, yshift=0.5cm] 30:1cm);
\node [label=above:$d_{\rho}$] at ([xshift=-2cm, yshift=0.3cm] 30:1cm){}; 
\node [label=above:$d_{\rho}-\nu_{1}$] at ( [xshift=-2cm, yshift=0.3cm] 150:1cm){};
\draw [->] (-3.5,1.5) -- (-4.5,1.5); 
\draw [->] (-3.5,1.5) -- (-3.5,2.5); 
\node [label=above:$\nu_{1}$] at (-4.5,1.4){}; 
\node [label=right:$d_{\rho}$] at (-3.6,2.3){};
\node [label=left:$e_{2}$] at (-1,0.4){}; 
\node [label=left:$e_{1}$] at (-2.2,0.4){};
\node [label=left:$e_{3}$] at (3,0.4){}; 
\node [label=left:$e_{4}$] at (1.8,0.4){};
\end{tikzpicture} 
\caption{The tropical cycle $\beta^{i}_{\trop}$}\label{fig:root1}
\end{center} \end{minipage}
\end{figure}

While balancing conditions for $\beta^{i}_{\trop}$ are obvious, we note that for $\beta'_{\trop}$, the unique vertex which is not at a focus-focus singularity also satisfies the balancing condition. On an edge leaving a focus-focus singularity, the edge carries a primitive vector, $\zeta_{\rho}$, parallel to $\Lambda_{\rho}$. Thus, at the vertex, we have the sum of vectors:
\[ 3\zeta_{\rho_1}-\zeta_{\rho_2}-\zeta_{\rho_4}=0\]
since $v_4+3v_1+v_2=0$.

\begin{lem} \label{lem:1}
The tropical cycles $\beta^{i}_{\trop}$ and $\beta'_{\trop}$ have self-intersection $-2$. 
\end{lem}

\begin{proof}
This follows from Lemma 4.10 and Lemma 4.11 of \cite{bauer}.
\end{proof}

\begin{figure} 
\begin{minipage}{\linewidth} \begin{center} 
\vspace{1 cm}
\begin{tikzpicture}[scale = 0.5]
  \draw [->] (0,0) -- (-2,6) node [anchor=east] {$\rho_2$};
  \draw [->] (0,0) -- (6,0) node[anchor=west] {$\rho_4$};
  \draw [->] (0,0)-- (0,6) node[anchor=south,yshift=.1cm] {$\rho_3$};
  \draw [->] (0,0) --(0,-6) node[anchor=north] {$\rho_1$};
 \draw (3,0) circle (5pt);
  \draw (-1,3) circle (5pt);
  \draw (0,-2) circle (5pt);
  \draw (0,-3) circle (5pt);
  \draw (0,-4) circle (5pt);
  \fill [black] (1,-1) circle (3pt); 
  
    \node (a) at (3, 0) {};
\node (b) at (-1, 3) {};
\node (c) at (0, -2) {};
\node (d) at (0, -3) {};
\node (e) at (0, -4) {};
\node (f) at (1, -1) {};

\draw[->] (a)  to [out=-120,in=45, looseness=1](f);
\draw[->] (b)  to [out=-90,in=-180, looseness=1] (f);
\draw[->] (c)  to [out=-45,in=-35, looseness=1] (f);
\draw[->] (d)  to [out=-45,in=-25, looseness=1] (f);
\draw[->] (e)  to [out=-45,in=-15, looseness=1] (f);

\end{tikzpicture}
\caption{ The tropical cycle $\beta'_{\trop}$. On $\rho_1$, $\beta'_{\trop}$ winds around focus-focus singularities given by $1+z^{[E_i]}X_{1}^{-1}=0 (i=1,2,3)$. On $\rho_2$, it winds around the focus-focus singularity given by $1+z^{[E_9]}X_{2}^{-1}=0$ and on $\rho_3$, around that given by  $1+z^{[E_{10}]}X_{4}^{-1}=0$ } \label{fig:root2}
\end{center}\end{minipage}
\end{figure}

To show that we have actually constructed a set of cycles that correspond to a basis of $D^{\perp}$, we compute the intersection matrix.

\begin{lem}
The intersection form given by $\beta'_{\trop},\beta^{1}_{\trop},.\beta^{2}_{\trop},...,\beta^{7}_{\trop}$ coincides with that of the lattice $E_8(-1)$.
\end{lem}

\begin{proof}
It suffices to compute all the intersection numbers. The self-intersection numbers are done in Lemma \ref{lem:1}. Consider $\beta^{i}_{\trop}\cdot \beta^{i+1}_{\trop}$. Then, there is a single intersection between the two cycles at the $(i+1)^{th}$ focus-focus singularity. Since the orientations of the two loops around the singularity are opposite, the intersection is $+1$. It is clear that $\beta^{i}_{\trop}\cdot \beta^{j}_{\trop}$ do not intersect if $|i-j|>1$. Finally, we check intersections with $\beta'_{\trop}$. There is no intersection between $\beta'_{\trop}$ and $\beta^{i}_{\trop}$ for $i>3$. Consider the case now where $i\in \{1,2\}$. There is an intersection at both the $i^{th}$ and $(i+1)^{th}$ focus-focus singularity. Note that the loops around the $i^{th}$ and $(i+1)^{th}$ focus-focus singularities have opposite orientations and the edges of $\beta'_{\trop}$ corresponding to those singularities are identical in terms of the vectors they carry and the orientation. Thus, the intersection contributions will cancel so $\beta'_{\trop}\cdot \beta^{i}_{\trop}=0$. Finally, consider $\beta'_{\trop}\cdot \beta^{3}_{\trop}$. Now there is only one intersection, around the third focus-focus singularity and it has the intersection number $1$.
To summarize, the intersection matrix (with the ordered basis as in the lemma) is given by:\\
\begin{center}
$
\begin{pmatrix}
-2 & 0 & 0 & 0 & 0 & 0 & 0 & 0\\
0 & -2 & 1 & 0 & 0 & 0 & 0 & 0\\
0 & 1 & -2 & 1 & 0 & 0 & 0 & 0\\
0 & 0 & 1 & -2 & 1 & 0 & 0 & 0\\
0 & 0 & 0 & 1 & -2 & 1 & 0 & 0\\
0 & 0 & 0 & 0 & 1 & -2 & 1 & 0\\
0 & 0 & 0 & 0 & 0 & 1 & -2 & 1\\
0 & 0 & 0 & 0 & 0 & 0 & 1 & -2
\end{pmatrix}$
\end{center}

\end{proof}

\begin{lem}\label{lem:a}
Given $s \in W$,  let $\beta_{i,s}$ be the singular cycle on $\bar{\mathcal{X}}_s$ corresponding to $\beta^{i}_{\trop}$. Then:
\[ \exp{ \left( \frac{1}{2\pi i} \int_{\beta_{i,s}}\Omega \right)}=z^{[E_i]-[E_{i+1}]}(s).\]
\end{lem} 

\begin{proof}
This follows by applying Theorem \ref{thm:per-int} and also Equation \ref{eq:localfocuscomputation}.
\end{proof}

\begin{lem} \label{lem:b}
Let $W$ be as given in Construction \ref{def:singular-2-cycles}. Then,
Given $s =(a,t) \in W \cap T_Y$,  let $\beta'_s$ be the singular cycle on $\bar{\mathcal{X}}_s$ corresponding to $\beta'_{\trop}$. Then:
\[  \exp{ \left( \frac{1}{2\pi i} \int_{\beta'_s}\Omega\right)}=t^{\kappa_{\rho_1}}z^{\left[ -E_1-E_2-E_3-E_9-E_{10} \right]}(a).\] 
\end{lem}
\begin{proof}
This follows from Theorem \ref{thm:per-int} and local contributions  of Ronkin functions as shown in Equation \ref{eq:localfocuscomputation}.
\end{proof}

\begin{prop} \label{prop:del-pezzo-1}
Given $s \in T_Y$, the period point of $\bar{\cX}_s$ is given by the image of $s$ under the surjective homomorphism $T_Y \twoheadrightarrow T_{(D^{\perp})^*}$.

\end{prop}
\begin{proof}
It follows from intersection pairing that $\beta'_s$ is in the same homology class as the curve class marked by $ p^{*}(\bar{D}_1)-E_1-E_2-E_3-E_9-E_{10}$ and $\beta_{i,s}$ the curve class marked by $E_i -E_{i+1}$ for $s\in T_Y$.

Suppose $s\in U$ where $U = T^{D} \cap (W \cap T_Y)$. Recall that for internal (-2)-curves, their holomorphic periods are equal to algebraic ones. Then, it follows from Lemma \ref{lem:a} and Lemma \ref{lem:b} that the statement of the lemma holds restricted to $W$. Then, the lemma follows from Lemma \ref{lem:rel-tor-mar} and the uniqueness of the analytic continuation. 
\end{proof}

\end{document}